\documentclass[11pt,a4paper,leqno]{amsart}
\usepackage[english]{babel}

\usepackage[hidelinks]{hyperref}

\usepackage[latin1]{inputenc}
\usepackage[T1]{fontenc}
\usepackage{amsfonts}
\usepackage{amsmath}
\usepackage{amssymb}
\usepackage{eurosym}
\usepackage{mathrsfs}
\usepackage{palatino}
\usepackage{color}
\usepackage{esint}
\usepackage{url}
\usepackage{verbatim}

\usepackage{enumerate}
\allowdisplaybreaks[4]

\newcommand{\R}{\mathbb{R}}
\newcommand{\C}{\mathbb{C}}

\newcommand{\calC}{\mathcal{C}}

\newcommand{\N}{\mathbb{N}}
\newcommand{\Z}{\mathbb{Z}}
\newcommand{\E}{\mathbb{E}}

\newcommand{\scrF}{\mathscr{F}}

\newcommand{\calR}{\mathcal{R}}

\newcommand{\bla}{\big \langle}
\newcommand{\bra}{\big \rangle}

\numberwithin{equation}{section}

\newcommand{\ud}[0]{\,\mathrm{d}}

\newcommand{\esssup}[0]{\operatornamewithlimits{ess\,sup}}

\newcommand{\BMO}[0]{\operatorname{BMO}}
\newcommand{\bmo}[0]{\operatorname{bmo}}

\newcommand{\loc}[0]{\operatorname{loc}}

\renewcommand{\Re}[0]{\operatorname{Re}}

\newcommand{\good}[0]{\operatorname{good}}

\newcommand{\ch}[0]{\operatorname{ch}}

\newcommand{\calD}[0]{\mathcal{D}}

\newcommand{\wt}[1]{{\widetilde{#1}}}

\swapnumbers
\theoremstyle{plain}
\newtheorem{thm}[equation]{Theorem}
\newtheorem{lem}[equation]{Lemma}
\newtheorem{prop}[equation]{Proposition}

\theoremstyle{definition}
\newtheorem{defn}[equation]{Definition}

\newtheorem{exmp}[equation]{Example}

\theoremstyle{remark}
\newtheorem{rem}[equation]{Remark}

\author{Kangwei Li}
\address[K.L.]{BCAM (Basque Center for Applied Mathematics), Alameda de Mazarredo 14, 48009 Bilbao, Spain} 
\email{kli@bcamath.org}

\author{Henri Martikainen}
\address[H.M.]{Department of Mathematics and Statistics, University of Helsinki, P.O.B. 68, FI-00014 University of Helsinki, Finland}
\email{henri.martikainen@helsinki.fi}

\author{Emil Vuorinen}
\address[E.V.]{Department of Mathematics and Statistics, University of Helsinki, P.O.B. 68, FI-00014 University of Helsinki, Finland}
\email{emil.vuorinen@helsinki.fi}

\title{Bilinear Calder\'on--Zygmund theory on product spaces}

\makeatletter
\@namedef{subjclassname@2010}{%
  \textup{2010} Mathematics Subject Classification}
\makeatother

\subjclass[2010]{42B20}
\keywords{Calder\'on--Zygmund operators, bi-parameter analysis, bilinear analysis, dyadic model operators, representation theorems, weighted estimates, commutators, multipliers} 

\begin{document}

\begin{abstract}
We develop a wide general theory of bilinear bi-parameter singular integrals $T$. First, we prove a dyadic representation theorem starting from $T1$ assumptions and apply it to show many estimates, including $L^p \times L^q \to L^r$ estimates in the full
natural range together with weighted estimates and mixed-norm estimates.
Second, we develop commutator decompositions and show estimates in the full range for commutators and iterated commutators, like $[b_1,T]_1$ and $[b_2, [b_1, T]_1]_2$, where
$b_1$ and $b_2$ are little BMO functions. Our proof method can be used to simplify and improve linear commutator proofs, even in the two-weight Bloom setting.
We also prove commutator lower bounds by using and developing the recent median method.
\end{abstract}

\maketitle

\section{Introduction}
\subsection{Estimates for singular integrals in various settings}
\subsubsection*{Linear and multilinear theory}
A singular integral operator (SIO) refers to a linear operator $T$ for which
$$
\langle Tf_1, f_2 \rangle := \int_{\R^n} Tf_1(x) f_2(x)\ud x = \iint_{\R^n \times \R^n} K(x,y) f_1(y) f_2(x)\ud y \ud x
$$ for some suitable kernel $K$ and for all nice functions $f_1, f_2$ that are disjointly supported.
A fundamental problem is to prove that $T$ is bounded in $L^p$, $p \in (1,\infty)$, under the assumption that $T$ is bounded in some $L^{p_0}$.
A key feature of the classical theory is that the following argument achieves this: prove weak $(1,1)$ via Calder\'on--Zygmund decomposition, interpolate and dualise. See any standard textbook such as Stein \cite{Stein:book}.

We discuss the bilinear theory of SIOs. A heuristic model of a bilinear SIO $T$ in $\R^n$ is
$T(f_1, f_2)(x) := \tilde T(f_1 \otimes f_2)(x,x)$, where $x \in \R^n$, $f_i \colon \R^n \to \C$, $(f_1 \otimes f_2)(x_1,x_2) = f_1(x_1) f_2(x_2)$
 and $\tilde T$ is a linear SIO in $\R^{2n}$. More precisely, a bilinear SIO $T$ has a kernel $K$ satisfying estimates that are obtained from the above heuristic via the linear estimates, and if $\operatorname{spt} f_i \cap \operatorname{spt} f_j = \emptyset$ for some $i,j$ then
$$
\langle T(f_1, f_2), f_3 \rangle =  \iiint_{\R^{3n}} K(x,y,z) f_1(y) f_2(z) f_3(x) \ud y \ud z \ud x.
$$

A fundamental aspect brought to play by the bilinearity is that now the natural range for boundedness is
$L^p \times L^q \to L^r$ for $p, q \in (1, \infty]$ and $r \in (1/2,\infty)$ satisfying the scaling $1/p + 1/q = 1/r$. If $p,q,r \in (1,\infty)$ we talk about the Banach range,
and otherwise about the quasi--Banach range, where we may have $r < 1$.
There are a lot of ways, including bilinear interpolation, sparse domination and good lambda type arguments, to go from the boundedness
 $L^{p_0} \times L^{q_0} \to L^{r_0}$, where $p_0, q_0, r_0$ are fixed, to the whole range. Some of these methods go through the endpoint estimate $L^1 \times L^1 \to L^{1/2, \infty}$, which again follows via the Calder\'on--Zygmund decomposition. See e.g. Grafakos--Torres \cite{GT}. More modern aspects are cited and used e.g. in \cite{LMOV}.

\subsubsection*{Multi-parameter theory}
We deal with bi-parameter theory.
An example of a bi-parameter SIO in the product
space $\R^n \times \R^m$ is $T_n \otimes T_m$, where
$T_n$ and $T_m$ are linear SIOs in $\R^n$ and $\R^m$, respectively, and $(T_n \otimes T_m)(f_1 \otimes f_2)(x) = T_n f_1(x_1) T_m f_2(x_2)$
for $x = (x_1, x_2) \in \R^{n+m}$ and two functions $f_1, f_2$ defined in $\R^n$ and $\R^m$, respectively. Noticing that $T_n \otimes T_m =
T_n^1 T_m^2$, where e.g. $T_n^1 f(x) = T_n(f(\cdot, x_2))(x_1)$, it follows by Fubini that
$
\| T_n \otimes T_m\|_{L^p \to L^p} \le \|T_n \|_{L^p \to L^p} \|T_m\|_{L^p \to L^p}.
$
Theory of non-tensor form bi-parameter SIOs is interesting. The classical references Fefferman--Stein \cite{FS} and Journ\'e \cite{Jo2} deal
with convolution form and general kernels, respectively.
The kernel structure of a non-tensor form SIO can be deduced from the tensor product case. We require two kinds of kernel representations depending
whether we have separation in both $\R^n$ and $\R^m$ (a full kernel representation), or just in $\R^n$ or $\R^m$ (a partial kernel representation).
This viewpoint on the kernels is recent, see \cite{Ma1} by one of us and Pott-Villarroya \cite{PV}.
Grau de la Herr\'an \cite{Grau} showed that the operator-valued formulations of Journ\'e \cite{Jo2} are equivalent to the more tangible
full and partial kernel assumptions of \cite{Ma1}.

The one-parameter machinery for achieving $L^p$ boundedness from boundedness in a single $L^{p_0}$ via weak $(1,1)$ does not work as this endpoint is false.
An alternative route via the boundedness $L^{\infty} \to \BMO_{\textup{prod}}$, where $ \BMO_{\textup{prod}}$ is the product BMO of Chang and Fefferman \cite{CF1, CF2}, exists. That this can be interpolated via the implication $T^* \colon H^1 \to L^1$ requires Chang--Fefferman \cite{CF3}, which heavily involves
product Hardy space machinery.
The estimate $L^{\infty} \to \BMO_{\textup{prod}}$ is also very important for the $T1$ theory. This is proved via a difficult argument involving Journ\'e's covering theorem \cite{Jo1}. For a modern proof see e.g. Hyt\"onen--Martikainen \cite{HM}, which works even in the non-doubling situation. For dyadic viewpoints on product $\BMO$ and $H^1$ see Pipher--Ward \cite{PW} and Treil \cite{Tr}.

An alternative way to prove $T \colon L^p \to L^p$, $p \in (1,\infty)$, is to show that $T1$ assumptions
imply $L^p$ boundedness (and not just $L^2$). This avoids the $H^1$ theory and follows e.g. from the representation theorem \cite{Ma1}. If a bi-parameter SIO $T$
and its partial adjoint $T_1$ (defined via $\langle T_1(f_1 \otimes f_2), g_1 \otimes g_2 \rangle = \langle T(g_1 \otimes f_2), f_1 \otimes g_2 \rangle$)
are bounded in some $L^{p_0}$, possibly with a different $p_0$, the estimate $L^{\infty} \to \BMO_{\textup{prod}}$ implies the $T1$ assumptions
for $T, T^*, T_1, T_1^*$, and so $T$ is bounded in all $L^p$. The partial adjoint has to be mentioned as it need not be bounded even if $T$ is.
We will show that at least this viewpoint has a useful analog in the bilinear bi-parameter setting.
Of course, the original main point of $T1$ theorems also holds: they show the boundedness just from testing conditions.

\subsubsection*{Multilinear multi-parameter theory}
Again, we focus on the bilinear bi-parameter setting.
A model of a bilinear bi-parameter singular integral in $\R^n \times \R^m$ is
$$
(T_n \otimes T_m)(f_1 \otimes f_2, g_1 \otimes g_2)(x) := T_n(f_1, g_1)(x_1)T_m(f_2, g_2)(x_2),
$$
where $f_1, g_1 \colon \R^n \to \C$, $f_2, g_2 \colon \R^m \to \C$,
$x = (x_1, x_2) \in \R^{n+m}$ and $T_n$, $T_m$ are bilinear SIOs defined in $\R^n$ and $\R^m$, respectively.
Unlike in the linear bi-parameter case, the theory of the tensor products is already non-trivial -- we will return to this point.
As in the linear bi-parameter case, the natural kernel structure of a general bilinear bi-parameter SIO can be deduced from the tensor product case. These are the operators that we study in this paper. Difficult estimates for some concrete bilinear bi-parameter SIOs have had interesting applications -- we get to these later. A systematic abstract theory in this setting has been missing, and we claim that developing such theory has big advantages.

The most fundamental question is the boundedness of a bilinear bi-parameter SIO $T$
in the full natural range $L^p \times L^q \to L^r$, where again $p, q \in (1, \infty]$ and $r \in (1/2,\infty)$ satisfy $1/p + 1/q = 1/r$.
An easy to explain corollary of our theory is that if $T$ is a bilinear bi-parameter SIO (which refers
just to possessing the kernel structure),
\begin{equation}\label{eq:bbb}
\|T(f_1, f_2)\|_{L^{r_0}} \lesssim \|f_1\|_{L^{p_0}} \|f_2\|_{L^{q_0}}, \qquad 1/p_0 + 1/q_0 = 1/r_0 < 1,
\end{equation}
and the same estimate holds for the partial adjoints of $T$ (maybe with different exponents $p_0, q_0, r_0$), then
we have in the full range that
$
\|T(f_1, f_2)\|_{L^r} \lesssim \|f_1\|_{L^p} \|f_2\|_{L^q}.
$
If $T(1,1) = 0$, and the same holds for the adjoints and the partial adjoints, we call $T$ free of full paraproducts (this terminology
will become clear later). In this case
\begin{equation}\label{eq:tensor}
\|T(f_1 \otimes f_2, g_1 \otimes g_2)\|_{L^{r_0}} \lesssim \|f_1\|_{L^{p_0}} \|f_2\|_{L^{p_0}} \|g_1 \|_{L^{q_0}}\|g_2 \|_{L^{q_0}}
\end{equation}
for some $p_0, q_0,r_0$ with $1/p_0 + 1/q_0 = 1/r_0 \le 1$ gives the boundedness of $T$ everywhere.

In general, we get the boundedness of SIOs in the full range under boundedness and cancellation conditions,
which consist of product BMO type conditions as in $T(1,1) \in \BMO_{\textup{prod}}$ (this follows from \eqref{eq:bbb} but also from weaker assumptions) and some mild conditions for which the tensor boundedness \eqref{eq:tensor} is sufficient. \emph{The story we have been trying to tell is that while this type of theory can be used
to prove the boundedness via the $T1$ assumptions (which is as fundamental here as it is in the classical theory),
it is also a machine giving the boundedness in the full range from the boundedness
in a fixed tuple $1/p_0 + 1/q_0 = 1/r_0 < 1$.}

Bilinear bi-parameter multipliers (that are special SIOs, see the Appendix) were studied by Muscalu--Pipher--Tao--Thiele \cite{MPTT}, and
they applied their theory to Leibniz rule type estimates relevant in non-linear PDE. The difficulty was to establish the boundedness
in the quasi--Banach range, as the boundedness in the Banach range was easy -- see however the above point  in cursive.
Other references for multilinear multi-parameter multipliers are Chen--Lu \cite{CL} (e.g. a weighted bound and some restricted smoothness results), Grafakos--He--Nguyen--Yan \cite{GHNY} (restricted smoothness),  Benea--Muscalu \cite{BM1, BM2} and Di Plinio--Ou \cite{DO} (e.g. mixed-norm estimates). We also mention Lacey--Metcalfe \cite{LM}.

In the free of full paraproducts case we also get bilinear weighted
estimates involving  bi-parameter Muckenhoupt $A_p$ classes, and we can in fact use this as the technology to get the full range via
extrapolation. We also prove mixed-norm estimates, where
$L^p(\R^{n+m})$ is replaced by $L^{p_1}(\R^n; L^{p_2}(\R^m))$. Such estimates are topical (e.g. the multipliers are free of full paraproducts).
For a bilinear bi-parameter SIO free of full paraproducts the simple condition \eqref{eq:tensor} (or some weaker conditions) yields all these difficult estimates.

In Journ\'e \cite{Jo3} tensor products $U_n \otimes U_m$ of ``multilinear singular integral forms'' are studied -- see Christ--Journ\'e \cite{CJ}
for the definition of these forms that they call $\delta$-$n$ SIF.
It is stated that a tensor product of general bilinear operators, both bounded $L^{\infty} \times L^2 \to L^2$,
need not be bounded $L^{\infty} \times L^2 \to L^2$.
It was proved that this is true, however, for their singular forms, and this was further applied to certain Cauchy type operators improving a previous application of the bi-parameter $T1$ theorem \cite{Jo3}.
We mention that we get that $T_n \otimes T_m$ -- a tensor product of bounded bilinear SIOs --  is weighted bounded in the full range, even without paraproduct free assumptions -- see Remark \ref{rem:tensorfullb}.

\subsubsection*{Commutators and methodology}
Our second contribution deals with commutator estimates and methods to prove them.
Commutator estimates have remained at the heart of modern harmonic analysis since the result of Coifman--Rochberg--Weiss \cite{CRW}
showing that $$
\|b\|_{\BMO} \lesssim \|[b,T]\|_{L^p \to L^p} \lesssim \|b\|_{\BMO}, \textup{ where } [b,T]f := bTf - T(bf),
$$
for a class of non-degenerate CZOs $T$. Both the upper and lower bounds of this estimate are non-trivial, and
allow deep extensions -- notably to the multi-parameter setting, see e.g. Ferguson--Sadosky \cite{FS} and Ferguson--Lacey \cite{FL}.
We contribute to the methodology by
developing commutator decompositions and estimates. Then we prove upper and lower bound commutator estimates in the bilinear bi-parameter framework.

A technical backbone for all of the above is obtained by first developing further the probabilistic--dyadic methods, which were initially pioneered by Nazarov--Treil--Volberg \cite{NTV} in their study of non-doubling singular integrals. 
These methods have previously also been significantly developed by many leading analysts to prove state of the art weighted results, and also to prove results in the border of harmonic analysis and geometric measure theory (see Tolsa's book \cite{To:book}).
Our modern approach is key to our results, but also offers many new avenues for future research in the bilinear bi-parameter setup -- such avenues e.g. include Bloom type two-weight estimates and non-doubling analysis.
We now move on to give a more detailed account of the field and our theory.

We now embark on the second phase of the introduction with full details.
\subsection{Dyadic representation theorems}
A representation theorem aims to represent SIOs by using some natural dyadic model operators (DMOs). The aim
is to reduce questions to these simpler DMOs. Such theorems have reshaped the modern thinking of singular integrals,
and the development of these methods is the reason why many problems in the bi-parameter scene are more attainable today.

Recently, representation theorems have been shown to hold
for both the bi-parameter and bilinear SIOs, but not in the simultaneous presence of both these difficulties. Moreover, the theory of the DMOs themselves and the usage of the representation is much more complicated in the bilinear bi-parameter framework.
The bi-parameter representation theorem is by one of us \cite{Ma1} and the bilinear case
is by Li--Martikainen--Ou--Vuorinen \cite{LMOV}. The representations provide precise structural information, which is often of key importance.
That is the case also in this paper. $T1$ theorems are byproducts.

In the linear bi-parameter context the representation theorem \cite{Ma1} has
proved to be very useful e.g. in connection with bi-parameter commutators and weighted analysis, see e.g. Holmes--Petermichl--Wick \cite{HPW} and Ou--Petermichl--Strouse \cite{OPS}. In the bilinear bi-parameter theory quasi--Banach estimates provide a key additional challenge not
present in the linear situation, and similarly weighted estimates are harder to obtain.

The representation theorems were originally motivated by the sharp weighted $A_p$ theory.
Petermichl \cite{Pe} showed a representation for the Hilbert transform, and Hyt\"onen \cite{Hy} proved a representation theorem for all bounded linear one-parameter SIOs.
While the rise of sparse domination methods has somewhat reduced the need of one-parameter representation theorems, this is not the case in multi-parameter situations, where a fully satisfying theory of sparse domination is missing (although see Barron--Pipher \cite{BP}).

There are three types of DMOs that are relevant for us: bilinear bi-parameter shifts, partial paraproducts and full paraproducts. We show, starting from $T1$ type assumptions, that our SIOs have a representation as an average of a rapidly decaying sum of DMOs. While the formulation is naturally of similar nature as in the previous representation theorems, the underlying dyadic structure is extremely involved in this situation, and the proof requires a very delicate decomposition.

The following terminology is convenient. A bilinear bi-parameter SIO should have the suitable kernel representations -- see Definition \ref{defn:SIO}.
A SIO satisfying the $T1$ type assumptions -- see Definition \ref{defn:CZO} -- is a Calder\'on--Zygmund operator (CZO).
\begin{thm}\label{thm:repIntro}
Suppose $T$ is a bilinear bi-parameter CZO. Then we have
$$
\langle T(f_1,f_2), f_3\rangle = C_T \mathbb{E}_{\omega}\mathop{\sum_{k = (k_1, k_2, k_3) \in \Z_+^3}}_{v = (v_1, v_2, v_3) \in \Z_+^3} \alpha_{k, v} 
\bla U^{k,v}_{\omega}(f_1, f_2), f_3 \bra,
$$
where $\omega = (\omega_1, \omega_2)$ is associated to random dyadic grids $\calD_{\omega} = \mathcal{D}^n_{\omega_1} \times \mathcal{D}^m_{\omega_2}$,
 $C_T \lesssim 1$, the numbers $\alpha_{k, v} > 0$ decay exponentially in complexity $(k,v)$, and
$U^{k,v}_{\omega}$ denotes some bilinear bi-parameter dyadic model operator of complexity $(k,v)$
defined in the lattice $\mathcal{D}_{\omega}$.
\end{thm}
The free of full paraproducts assumption that appears in some statements below implies precisely that $T$
has a representation with shifts and partial paraproducts only. This requires that $T(1,1) = 0$
in the sense that $\langle T(1,1), h_I \otimes h_J \rangle = 0$ (where $h_I$ is a cancellative Haar function on a cube $I$)
and the same for the adjoints and partial adjoints.

Multipliers are paraproduct free CZOs
(even the partial paraproducts vanish). An example of an operator that is free of full paraproducts but not necessarily of partial paraproducts
is $T_n \otimes T_m$, where $T_n$ is free of paraproducts.
Regarding terminology it seems that continuous analogs of what we call shifts are often called paraproducts. For us a paraproduct
is an operator involving BMO philosophies. The shifts are the simplest DMOs as their boundedness is based on the size of the individual coefficients only.
\subsection{Boundedness properties of $T$}
The following theorem states our estimates for CZOs.
\begin{thm}
Suppose $T$ is a bilinear bi-parameter CZO.
Then we have
$$
\|T(f_1, f_2)\|_{L^r(\R^{n+m})} \lesssim \|f_1\|_{L^p(\R^{n+m})} \|f_2\|_{L^q(\R^{n+m})} 
$$
for all $1 < p, q \le \infty$ and $1/2 < r < \infty$ satisfying $1/p+1/q = 1/r$. 

Suppose further that $T$ is free of full paraproducts.
Then we have the weighted estimate
$$
\|T(f_1, f_2)\|_{L^r(v_3)} \le C([w_1]_{A_p(\R^n \times \R^m)}, [w_2]_{A_q(\R^n \times \R^m)}) \|f_1\|_{L^p(w_1)} \|f_2\|_{L^q(w_2)} 
$$
for all $1 < p, q < \infty$ and $1/2 < r < \infty$ satisfying $1/p+1/q = 1/r$, and for all bi-parameter weights $w_1 \in A_p(\R^n \times \R^m)$, $w_2 \in A_q(\R^n \times \R^m)$ with $v_3 := w_1^{r/p} w_2^{r/q}$. In the unweighted case we also have the mixed-norm estimates
$$
\|T(f_1, f_2)\|_{L^{r_1}(\R^n; L^{r_2}(\R^m))} \lesssim \|f_1\|_{L^{p_1}(\R^n; L^{p_2}(\R^m))}\|f_2\|_{L^{q_1}(\R^n; L^{q_2}(\R^m))}
$$
for all $1 < p_i, q_i \le \infty$ and $1/2 < r_i < \infty$ with $1/p_i + 1/q_i = 1/r_i$, except that if $r_2 < 1$ we have
to assume $\infty \not \in \{p_1, q_1\}$.
\end{thm}
The weighted estimates are first proved for the shifts and partial paraproducts, and then moved to full paraproduct free CZOs
using the representation and the multilinear extrapolation by Grafakos--Martell (and Duoandikoetxea) \cite{DU, GM}. Extrapolation
is needed as quasi--Banach estimates cannot be moved to $T$ directly via the representation
due to the presence of the averaging. The mixed-norm estimates follow from the weighted estimates and operator-valued analysis.

Weighted estimates for singular integrals in the linear bi-parameter setting are already quite difficult: see Fefferman--Stein \cite{FS} and Fefferman \cite{Fe}, \cite{Fe2}. Recently, Holmes--Petermichl--Wick \cite{HPW} showed weighted bounds for bi-parameter singular integrals using the representation theorem \cite{Ma1}. We use this approach also here, but the nature of the bilinear weights poses a problem with the usage of duality (notice e.g. that if $w_1, w_2 \in A_4$ then $v_3 := w_1^{1/2} w_2^{1/2} \in A_4$ while
we need to work in $L^2(v_3)$).  Therefore, the proof of the weighted boundedness of DMOs cannot proceed
as in the linear case. This is a problem, as many bi-parameter proofs, especially those related to full paraproducts, are
based on $H^1$-$\BMO$ type duality arguments.

We manage to bound the shifts and partial paraproducts with duality free proofs with a careful usage of the $A_{\infty}$ condition (so we only exploit the fact that $v_3 \in A_{\infty}$). For example, we use the $A_{\infty}$ extrapolation by Cruz-Uribe--Martell--P\'erez \cite{CUMP} and known lower square function bounds valid for $A_{\infty}$ weights.
The weighted proof of the partial paraproducts also exploits sparse bounds of the one-parameter bilinear paraproducts.

Regarding the first part of the theorem, we still need to show unweighted bounds in the full range for (averages of) full paraproducts. Here we use a different set of tools than above:
weak type arguments and interpolation. The general method is from the paper \cite{MPTT}, but especially in the commutator setting, where we also use these weak type methods, the setup is much more complicated.

\subsection{Commutator estimates}
\subsubsection*{Upper bounds}
We define the commutators
$$
[b,T]_1(f_1,f_2) = bT(f_1, f_2) - T(bf_1, f_2) \, \textup{ and } \, [b,T]_2(f_1, f_2) = bT(f_1,f_2) - T(f_1, bf_2). 
$$
A function $b$ is in little BMO, $\bmo(\R^{n+m})$, if $b(\cdot, x_2)$ and $b(x_1, \cdot)$ are uniformly in BMO.
By studying DMOs we prove the following upper bound for commutators.
\begin{thm}\label{thm:comUP}
Let $1 < p,q \le \infty$ and $1/2 < r < \infty$ satisfy $1/p + 1/q = 1/r$, and let $b \in \bmo(\R^{n+m})$.
Suppose $T$ is a bilinear bi-parameter CZO.
Then we have
$$
\|[b, T]_1(f_1, f_2)\|_{L^r(\R^{n+m})} \lesssim \|b\|_{\bmo(\R^{n+m})} \|f_1\|_{L^p(\R^{n+m})} \|f_2\|_{L^q(\R^{n+m})},
$$
and similarly for $[b, T]_2$.
Analogous results hold for iterated commutators like $[b_2, [b_1, T]_1]_2$.
\end{thm}

We discuss the proof of Theorem \ref{thm:comUP} later. The theory of commutator estimates is extremely vast and important.
By bounding commutators of shifts Ou, Petermichl and Strouse \cite{OPS} proved that $[b,T]$ is $L^2$ bounded, when $T$ is a bi-parameter SIO as in \cite{Ma1}
and free of paraproducts.
This is the important base case for more complicated multi-parameter commutator estimates -- i.e., the proof method using the DMOs lends itself to more complicated estimates.
For example, suppose that $T_1$ and $T_2$ are paraproduct free linear bi-parameter singular integrals satisfying the assumptions of the representation theorem \cite{Ma1} in $\R^{n_1} \times \R^{n_2}$ and $\R^{n_3} \times \R^{n_4}$ respectively.
Then according to \cite{OPS} we have
\begin{align*}
\|[T_1, [b, T_2]]f&\|_{L^2(\prod_{i=1}^4 \R^{n_i})} \lesssim \max\big( \sup_{x_2, x_4} \|b(\cdot, x_2, \cdot, x_4)\|_{\BMO_{\operatorname{prod}}},
 \sup_{x_2, x_3} \|b(\cdot, x_2, x_3, \cdot)\|_{\BMO_{\operatorname{prod}}}, \\
 & \sup_{x_1, x_4} \|b(x_1, \cdot, \cdot, x_4)\|_{\BMO_{\operatorname{prod}}},
  \sup_{x_1, x_3} \|b(x_1, \cdot, x_3, \cdot)\|_{\BMO_{\operatorname{prod}}} \big) \|f\|_{L^2(\prod_{i=1}^4 \R^{n_i})}.
\end{align*}
Related to this see also Dalenc--Ou \cite{DaO}. The dyadic methods are also of crucial importance in many other deeper instances, like in the Bloom type estimates discussed below.

The paper \cite{OPS} was eventually generalised to concern all bi-parameter singular integrals satisfying $T1$ conditions
by Holmes--Petermichl--Wick \cite{HPW}.  Importantly, they also prove \emph{Bloom type two-weight bounds} for commutators of the particular form $[b,T]$.
With a Bloom type inequality we understand the following. Given some operator $A^b$, the definition of which depends naturally on some function $b$,
we seek for a two-weight estimate
$$
\|A^b\|_{L^p(\mu) \to L^p(\lambda)} \lesssim_{[\mu]_{A_p}, [\lambda]_{A_p}} \|b\|_{\BMO(\nu)},
$$
where $p \in (1,\infty)$, $\mu, \lambda \in A_p$, $\nu := \mu^{1/p}\lambda^{-1/p}$, and $\BMO(\nu)$ is some suitable weighted $\BMO$ space.
Usually $A^b$ is some commutator like $[b,T]$ in \cite{HPW}.

Recently, in a sequel of the current paper \cite{LMV:Bloom} we gave
an efficient proof of \cite{HPW} exploiting our modified commutator decompositions from this paper (see below), and generalised the result of \cite{HPW} to iterated commutators 
of the form $[b_k,\ldots[b_2, [b_1, T]]\ldots]$. This kind of iteration in the Bloom setting is not obvious with the appearance of weighted BMO spaces like $\bmo(\nu^{1/k})$ --
see Lerner--Ombrosi--Rivera-R\'ios \cite{LOR2} for the one-parameter setting and sparse domination.
Related previous one parameter commutator results include
Holmes--Lacey--Wick \cite{HLW, HLW2} and  Lerner--Ombrosi--Rivera-R\'ios \cite{LOR1}.
In an upcoming work we also use the commutator decomposition strategy of the current paper
to generalise Dalenc--Ou \cite{DaO} to the Bloom setting: 
$
\| [T_n^1, [b, T_m^2]] \|_{L^p(\mu) \to L^p(\lambda)} \lesssim_{[\mu]_{A_p}, [\lambda]_{A_p}} \|b\|_{\BMO_{\textup{prod}}(\nu)}
$
for a bounded singular integral $T_n$ in $\R^n$ and a bounded singular integral $T_m$ in $\R^m$. That is,
we handle product BMO type commutators in the Bloom setting.

The linear bi-parameter results discussed so far rely on the bi-parameter representation theorem \cite{Ma1} and its multi-parameter generalisation by Y. Ou \cite{Ou}, and on
increasingly sophisticated ways to bound the appearing model operators and their commutators. Similarly, we rely on our new bilinear bi-parameter representation theorem.
However, in Theorem \ref{thm:comUP} we do not consider Bloom type theory -- such bilinear bi-parameter Bloom theory is left as a very interesting question.
In the one-parameter setting multilinear Bloom type
results appear in Kunwar--Ou \cite{KO} (where sparse domination is used).

\subsubsection*{Proof of Theorem \ref{thm:comUP} and commutator decompositions}
The main challenge in going from \cite{OPS} to \cite{HPW}, apart from the Bloom setting, appeared to be that the various paraproducts include non-cancellative Haar functions and have
a more complicated structure than the shifts.
In \cite{HPW} everything was reduced to a so called remainder term, which entailed always expanding $bf$ in the bi-parameter sense. However, this remainder term has a particularly nice structure only when there are no non-cancellative Haar functions. 

Our guideline is to expand $bf$ using bi-parameter martingales in $\langle bf, h_{I} \otimes h_J\rangle$, using one-parameter martingales in
$\langle bf, h_{I}^0 \otimes h_J\rangle$ (or $\langle bf, h_I \otimes h_J^0\rangle$), and not to expand at all in $\langle bf, h_{I}^0 \otimes h_J^0\rangle$ (here $h_I^0$ is a non-cancellative Haar function).
Moreover, when a non-cancellative Haar function appears, a suitable average of $b$ is added and subtracted.
Working like this the proof of the Banach range boundedness of bilinear commutators is surprisingly short. For the quasi--Banach range
we use weak type estimates as already mentioned, and this is quite delicate.

\subsubsection*{The Cauchy trick and its limitations}
The Cauchy integral trick by Coifman--Rochberg--Weiss \cite{CRW} is a trick, where
one e.g. defines $F(z) = e^{bz}T(e^{-bz}f_1, f_2)$, notices that formally $F'(0) = [b,T]_1(f_1, f_2)$, uses the Cauchy integral
formula and then hopes to employ \emph{weighted} bounds of $T$ to get bounds for $[b,T]_1$. Weighted bounds for $T$ are always needed as some weights of the form $e^{\Re(bz)}$ appear in this argument.
This trick is tied to the BMO properties of $b$ by requiring exponential integrability as given by the classical inequality of John--Nirenberg.
It still works in the bi-parameter setting if $b$ is in little BMO, however.

The simplest reason why we do not use this trick is that we currently get
weighted bounds for $T$ only in the free of full paraproducts case, while we prove the commutator result for all $T$.
Also, our proof of the weighted bound for the partial paraproducts utilises one-parameter sparse domination, and therefore does not work in three or higher parameters.
Thus, we think that another proof for the boundedness of the partial paraproducts in the full range is of interest. While we do not show
such proofs explicitly, we perform weak type arguments in this more complicated commutator setting from which they can easily be deduced. 

The most important reason for developing and using the dyadic commutator estimates is that they can be used in many instances where the Cauchy trick fails. The trick does not work for the product BMO type commutators discussed in the linear situation above, and
we are also currently investigating certain natural bilinear analogs of product BMO type commutators.
Moreover, in the Bloom setting one has to directly prove the first order case no matter what, and in the most general little bmo type iterations $b$ is not in $\bmo$ (recall that the Cauchy trick relies on the classical BMO property of $b$). So in the Bloom setting the Cauchy trick is not really useful even in the little BMO context.

\subsubsection*{Lower bounds}
Lower bounds for commutators of particular singular integrals (like the Riesz transform) complement the theory of the upper bounds discussed above. For the history and interesting new results and methods see the recent paper by Hyt\"onen \cite{Hy5}.

In the multi-parameter setting the deepest lower bounds concern the product BMO type commutators. The Hilbert case $[H_1, [b, H_2]]$
in considered in Ferguson--Lacey \cite{FL} -- see also Lacey--Petermichl--Pipher--Wick \cite{LPPW, LPPW2} for the Riesz setting. Ou--Petermichl--Strouse \cite{OPS} also contains general bounds of such type. Notice that e.g. in \cite{OPS} the proved lower bounds for Riesz transforms go through
some well-chosen bi-parameter CZOs of not tensor product type -- this is yet another example of the use of general theory.

Little BMO type lower bounds are not as difficult as the product BMO ones.
In this paper we are dealing with bilinear little BMO type commutators. In the linear bi-parameter case at least \cite{HPW} contains a little BMO type lower bound, even in the Bloom setting.
Guo--Lian--Wu \cite{GLW} proved lower bounds in the bilinear one-parameter setting -- even for exponents in the quasi--Banach
range. Previous bilinear results like Chaffee \cite{Ch} worked only in the Banach range.
However, we do not try to adapt any of these methods to the bilinear bi-parameter setting, since a bilinear bi-parameter version
of the wonderful median method of Hyt\"onen \cite{Hy5} and Lerner--Ombrosi--Rivera-R\'ios \cite{LOR2} can be developed.
This allows us to state the lower bounds for quite general non-degenerate kernels using very weak off-diagonal type testing conditions -- see Section \ref{sec:LowerBounds} for the statements.

Sections \ref{sec:def}, \ref{sec:defbilinbiparSIO} and \ref{sec:defbilinbiparmodel} are devoted to definitions and setting the stage.
In Section \ref{sec:ProofOfRep} we prove the representation theorem. Sections \ref{sec:ModelBou} and \ref{sec:TBou}
contain the estimates for the DMOs and the CZOs, respectively. Sections \ref{sec:com1} and \ref{sec:iterated} give the commutator upper bounds.

\subsection*{Acknowledgements}
K. Li is supported by Juan de la Cierva - Formaci\'on 2015 FJCI-2015-24547, by the Basque Government through the BERC
2018-2021 program and by Spanish Ministry of Economy and Competitiveness
MINECO through BCAM Severo Ochoa excellence accreditation SEV-2013-0323
and through project MTM2017-82160-C2-1-P funded by (AEI/FEDER, UE) and
acronym ``HAQMEC''.

H. Martikainen is supported by the Academy of Finland through the grants 294840 and 306901, the three-year research grant 75160010 of the University of Helsinki,
and is a member of the Finnish Centre of Excellence in Analysis and Dynamics Research.

E. Vuorinen is supported by the Academy of Finland through the grant 306901 and by the Finnish Centre of Excellence in Analysis and Dynamics Research.

We thank Francesco Di Plinio and Yumeng Ou for useful conversations when the last two named authors visited MIT in December 2017. We
thank MIT for the hospitality. 
\section{Basic definitions and results}\label{sec:def}
\subsection{Basic notation}
We denote $A \lesssim B$ if $A \le CB$ for some constant $C$ that can depend on the dimension of the underlying spaces, on integration exponents, and on various other constants appearing in the assumptions. We denote $A \sim B$ if $B \lesssim A \lesssim B$.

We work in the bi-parameter setting in the product space $\R^{n+m}$.
In such a context $x = (x_1, x_2)$ with $x_1 \in \R^n$ and $x_2 \in \R^m$.
We often take integral pairings with respect to one of the two variables only:
If $f \colon \R^{n+m} \to \C$ and $h \colon \R^n \to \C$, then $\langle f, h \rangle_1 \colon \R^{m} \to \C$ is defined by
$
\langle f, h \rangle_1(x_2) = \int_{\R^n} f(y_1, x_2)h(y_1)\ud y_1.
$

\subsection{Dyadic notation, Haar functions and martingale differences}
We denote a dyadic grid in $\R^n$ by $\calD^n$ and a dyadic grid in $\R^m$ by $\calD^m$.
The dyadic rectangles are denoted by $\calD = \calD^n \times \calD^m$.
If $I \in \calD^n$, then $I^{(k)}$ denotes the unique dyadic cube $S \in \calD^n$ so that $I \subset S$ and $\ell(S) = 2^k\ell(I)$. Here $\ell(I)$ stands for side length. Also, $\text{ch}(I)$ denotes the dyadic children of $I$ -- this means that $I' \in \ch(I)$ if $(I')^{(1)} = I$.  The measure of a cube $I$ is simply denoted by $|I|$ no matter what dimension we are in.

When $I \in \calD^n$ we denote by $h_I$ a cancellative $L^2$ normalised Haar function. This means the following.
Writing $I = I_1 \times \cdots \times I_n$ we can define the Haar function $h_I^{\eta}$, $\eta = (\eta_1, \ldots, \eta_n) \in \{0,1\}^n$, by setting
$
h_I^{\eta} = h_{I_1}^{\eta_1} \otimes \cdots \otimes h_{I_n}^{\eta_n}, 
$
where $h_{I_i}^0 = |I_i|^{-1/2}1_{I_i}$ and $h_{I_i}^1 = |I_i|^{-1/2}(1_{I_{i, l}} - 1_{I_{i, r}})$ for every $i = 1, \ldots, n$. Here $I_{i,l}$ and $I_{i,r}$ are the left and right
halves of the interval $I_i$ respectively. If $\eta \in \{0,1\}^n \setminus \{0\}$ the Haar function is cancellative: $\int h_I^{\eta} = 0$. We usually suppress the presence of $\eta$
and simply write $h_I$ for some $h_I^{\eta}$, $\eta \in \{0,1\}^n \setminus \{0\}$. Then $h_Ih_I$ can stand for $h_I^{\eta_1} h_I^{\eta_2}$, but we always treat
such a product as a non-cancellative function.

For $I \in \calD^n$ and $f \in L^1_{\textup{loc}}(\R^n)$ we define the martingale difference
$$
\Delta_I f = \sum_{I' \in \textup{ch}(I)} \big[ \bla f \bra_{I'} -  \bla f \bra_{I} \big] 1_{I'}.
$$
Here $\bla f \bra_I = \frac{1}{|I|} \int_I f$. We also write $E_I f = \bla f \bra_I 1_I$.
Now, we have $\Delta_I f = \sum_{\eta \ne 0} \langle f, h_{I}^{\eta}\rangle h_{I}^{\eta}$, or suppressing the $\eta$ summation, $\Delta_I f = \langle f, h_I \rangle h_I$, where $\langle f, h_I \rangle = \int f h_I$. A martingale block is defined by
$$
\Delta_{K,i} f = \mathop{\sum_{I \in \calD^n}}_{I^{(i)} = K} \Delta_I f, \qquad K \in \calD^n.
$$

Next, we define bi-parameter martingale differences. Let $f \colon \R^n \times \R^m \to \C$ be locally integrable.
Let $I \in \calD^n$ and $J \in \calD^m$. We define the martingale difference
$$
\Delta_I^1 f \colon \R^{n+m} \to \C, \Delta_I^1 f(x) := \Delta_I (f(\cdot, x_2))(x_1).
$$
Define $\Delta_J^2f$ analogously, and also define $E_I^1$ and $E_J^2$ similarly.
We set
$$
\Delta_{I \times J} f \colon \R^{n+m} \to \C, \Delta_{I \times J} f(x) = \Delta_I^1(\Delta_J^2 f)(x) = \Delta_J^2 ( \Delta_I^1 f)(x).
$$
Notice that $\Delta^1_I f = h_I \otimes \langle f , h_I \rangle_1$, $\Delta^2_J f = \langle f, h_J \rangle_2 \otimes h_J$ and
$ \Delta_{I \times J} f = \langle f, h_I \otimes h_J\rangle h_I \otimes h_J$ (suppressing the finite $\eta$ summations).
Martingale blocks are defined in the natural way
$$
\Delta_{K \times V}^{i, j} f  =  \sum_{I\colon I^{(i)} = K} \sum_{J\colon J^{(j)} = V} \Delta_{I \times J} f = \Delta_{K,i}^1( \Delta_{V,j}^2 f) = \Delta_{V,j}^2 ( \Delta_{K,i}^1 f).
$$

\subsection{Weights}
A weight $w(x_1, x_2)$ (i.e. a locally integrable a.e. positive function) belongs to the bi-parameter $A_p$ class, $A_p(\R^n \times \R^m)$, $1 < p < \infty$, if
$$
[w]_{A_p(\R^n \times \R^m)} := \sup_{R} \bla w \bra_R \bla w^{1-p'} \bra_R^{p-1} < \infty,
$$
where the supremum is taken over $R = I \times J$, where $I \subset \R^n$ and $J \subset \R^m$ are cubes
with sides parallel to the axes (we simply call such $R$ rectangles).
We have
$$
[w]_{A_p(\R^n\times \R^m)} < \infty \textup { iff } \max\big( \esssup_{x_1 \in \R^n} \,[w(x_1, \cdot)]_{A_p(\R^m)}, \esssup_{x_2 \in \R^m}\, [w(\cdot, x_2)]_{A_p(\R^n)} \big) < \infty,
$$
and that $\max\big( \esssup_{x_1 \in \R^n} \,[w(x_1, \cdot)]_{A_p(\R^m)}, \esssup_{x_2 \in \R^m}\, [w(\cdot, x_2)]_{A_p(\R^n)} \big) \le [w]_{A_p(\R^n\times \R^m)}$, while
the constant $[w]_{A_p}$ is dominated by the maximum to some power.
Of course, $A_p(\R^n)$ is defined similarly as $A_p(\R^n \times \R^m)$ -- just take the supremum over cubes $Q$. For the basic theory
of bi-parameter weights consult e.g. \cite{HPW}.

Also, recall that $w \in A_{\infty}(\R^n)$ if
$$
[w]_{A_{\infty}(\R^n)} = \sup_Q \Big( \frac{1}{|Q|} \int_Q w \Big) \operatorname{exp}\Big( \frac{1}{|Q|} \int_Q \log w^{-1} \Big) < \infty,
$$
where the supremum is taken over all the cubes $Q \subset \R^n$. We will only use
the fact that $A_p \subset A_{\infty}$, and that certain key estimates hold for the larger class of $A_{\infty}$ weights. 

We record the following standard weighted square function estimates.
\begin{lem}\label{lem:standardEst1}
For $p \in (1,\infty)$ and $w \in A_p(\R^n \times \R^m)$ we have
\begin{align*}
\| f \|_{L^p(w)}
& \sim_{[w]_{A_p(\R^n \times \R^m)}} \Big\| \Big( \mathop{\sum_{I \in \calD^n}}_{J \in \calD^m} |\Delta_{I \times J} f|^2 \Big)^{1/2} \Big\|_{L^p(w)} \\
&\sim_{[w]_{A_p(\R^n \times \R^m)}}  \Big\| \Big(  \sum_{I \in \calD^n} |\Delta_I^1 f|^2 \Big)^{1/2} \Big\|_{L^p(w)}
\sim_{[w]_{A_p(\R^n \times \R^m)}} \Big\| \Big(  \sum_{J \in \calD^m} |\Delta_J^2 f|^2 \Big)^{1/2} \Big\|_{L^p(w)}.
\end{align*}
\end{lem}

\subsection{Maximal functions}
Given $f \colon \R^{n+m} \to \C$ and $g \colon \R^n \to \C$ we denote the dyadic maximal functions
by
$$
M_{\calD^n}g := \sup_{I \in \calD^n} \frac{1_I}{|I|}\int_I |g|\, \textup{ and } \,
M_{\calD} f:= \sup_{R \in \calD}  \frac{1_R}{|R|}\int_R |f|.
$$
The non-dyadic variants are simply denoted by $M$, as it is clear what is meant from the context.
We also set $M^1_{\calD^n} f(x_1, x_2) =  M_{\calD^n}(f(\cdot, x_2))(x_1)$. The operator $M^2_{\calD^m}$ is defined similarly.
For various maximal functions $M$ we define $M_s$ by setting $M_s f = (M |f|^s)^{1/s}$.

Standard weighted estimates involving maximal functions are recorded below.
\begin{lem}\label{lem:standardEst2}
For $p, s \in (1,\infty)$ and $w \in A_p$ we have the Fefferman--Stein inequality
$$
\Big\| \Big( \sum_j |M f_j |^s \Big)^{1/s} \Big\|_{L^p(w)} \le C([w]_{A_p}) \Big\| \Big( \sum_{j} | f_j |^s \Big)^{1/s} \Big\|_{L^p(w)}.
$$
We also have
$$
\| \varphi_{\calD^n}^1 f\|_{L^p(w)} 
\sim_{[w]_{A_p}}
\Big\| \Big( \sum_{I \in \calD^n} \frac{1_I}{|I|} \otimes [M \langle f, h_I \rangle_1]^2 \Big)^{1/2} \Big\|_{L^p(w)} 
\lesssim_{[w]_{A_p}} \|f\|_{L^p(w)},
$$
where
$
\varphi_{\calD^n}^1 f  := \sum_{I \in \calD^n} h_I \otimes M \langle f, h_I \rangle_1.
$
The function $\varphi_{\calD^m}^2 f$ is defined in the symmetric way and satisfies the same estimates.
\end{lem}

\subsection{BMO spaces}\label{ss:bmo}
We say that a locally integrable function $b \colon \R^n \to \C$ belongs to the dyadic BMO space $\BMO(\calD^n)$ if
$$
\|b\|_{\BMO(\calD^n)} := \sup_{I \in \calD^n} \frac{1}{|I|} \int_I |b - \langle b \rangle_I| < \infty.
$$
The ordinary space $\BMO(\R^n)$ is defined by taking the supremum over all cubes. 

We say that a locally integrable function $b \colon \R^{n+m} \to \C$ belongs to the dyadic little BMO space $\bmo(\calD)$,
where $\calD = \calD^n \times \calD^m$, if
$$
\|b\|_{\bmo(\calD)} := \sup_{R \in \calD} \frac{1}{|R|} \int_R |b - \langle b \rangle_R| < \infty.
$$
The non-dyadic space $\bmo(\R^{n+m})$ is defined in the natural way -- take the supremum over all rectangles. We have
$$
\|b\|_{\bmo(\calD)} \sim \max\big( \esssup_{x_1 \in \R^n} \, \|b(x_1, \cdot)\|_{\BMO(\calD^m)}, \esssup_{x_2 \in \R^m}\, \|b(\cdot, x_2)\|_{\BMO(\calD^n)} \big).
$$
We also have the John--Nirenberg property
$$
\|b\|_{\bmo(\R^{n+m})} \sim \sup_{R \subset \R^{n+m}} \Big( \frac{1}{|R|} \int_{R} |b - \langle b \rangle_{R}|^p \Big)^{1/p}, \qquad 1 < p < \infty.
$$
Moreover, we need to know that $\bmo(\R^{n+m}) \subset \BMO_{\textup{prod}}(\R^{n+m})$. The reader can consult e.g. \cite{HPW, OPS}.
A short proof of a weighted version of this inclusion is included in \cite{LMV:Bloom}.

Finally, we have the product BMO space. Set
$$
\|b\|_{\BMO_{\textup{prod}}(\calD)} := 
\sup_{\Omega} \Big( \frac{1}{|\Omega|} \mathop{\sum_{I \times J \in \calD}}_{I \times J \subset \Omega} |\langle b, h_I \otimes h_J\rangle|^2  \Big)^{1/2},
$$
where the supremum is taken over those sets $\Omega \subset \R^{n+m}$ such that $|\Omega| < \infty$ and such that for every $x \in \Omega$ there exist
$I \times J \in \calD$ so that $x \in I \times J \subset \Omega$.
The non-dyadic product BMO space $\BMO_{\textup{prod}}(\R^{n+m})$ can be defined using the norm defined by the supremum over all dyadic grids of
the above dyadic norms. 

\subsection{Commutators}
We set 
$$
[b,T]_1(f_1,f_2) = bT(f_1, f_2) - T(bf_1, f_2) \, \textup{ and } \, [b,T]_2(f_1, f_2) = bT(f_1,f_2) - T(f_1, bf_2). 
$$
These are understood in a situation, where we e.g. know that $T \colon L^3(\R^{n+m}) \times L^3(\R^{n+m}) \to L^{3/2}(\R^{n+m})$, and
$b$ is locally in $L^3$. Then we study the case that $f_1$ and $f_2$ are, say, bounded and compactly supported, so that
e.g. $bf_2 \in L^3(\R^{n+m})$ and $bT(f_1,f_2) \in L^1_{\loc}(\R^{n+m})$.

\subsection{Duality}\label{ss:duality}
For a bilinear operator $T$ acting on functions defined on the product space $\R^{n+m} = \R^n \times \R^m$ we define the following duals
\begin{align*}
\langle T(f_1 \otimes& f_2, g_1 \otimes g_2), h_1 \otimes h_2\rangle \\
&= \langle T^{1*}(h_1 \otimes h_2, g_1 \otimes g_2), f_1 \otimes f_2\rangle
= \langle T^{2*}(f_1 \otimes f_2, h_1 \otimes h_2), g_1 \otimes g_2\rangle \\
& = \langle T^{1*}_1(h_1 \otimes f_2, g_1 \otimes g_2), f_1 \otimes h_2\rangle
= \langle T^{2*}_1(f_1 \otimes f_2, h_1 \otimes g_2), g_1 \otimes h_2\rangle  \\
&= \langle T^{1*}_2(f_1 \otimes h_2, g_1 \otimes g_2), h_1 \otimes f_2\rangle 
= \langle T^{2*}_2(f_1 \otimes f_2, g_1 \otimes h_2), h_1 \otimes g_2\rangle \\
&= \langle T^{1*, 2*}_{1,2}(h_1 \otimes f_2, g_1 \otimes h_2), f_1 \otimes g_2\rangle
= \langle T^{1*, 2*}_{2,1}(f_1 \otimes h_2, h_1 \otimes g_2), g_1 \otimes f_2\rangle.
\end{align*}
One should understand that one cannot claim that the boundedness of $T$ would e.g. imply the boundedness of $T^{1*}_1$ -- partial
duals like this need not be bounded even if $T$ is. This is a standard caveat of bi-parameter analysis.

\section{Bilinear bi-parameter SIOs and CZOs}\label{sec:defbilinbiparSIO}
We begin by formulating what it means for $T$ to be a bilinear bi-parameter singular integral operator (SIO).
Let $\scrF_n$ consist of all the finite linear combinations of indicators of cubes $I \subset \R^n$,
define $\scrF_m$ analogously, and let $\scrF$ consist of all the finite linear combinations of indicators of rectangles
$R \subset \R^{n+m}$.
Assume that we have bilinear operators $T$, $T^{1*}$, $T^{2*}$,
$T^{1*}_1$, $T^{2*}_1$, $T^{1*}_2$, $T^{2*}_2$, $T^{1*, 2*}_{1,2}$ and $T^{1*, 2*}_{2,1}$, each acting on tuples $(f_1, f_2) \in \scrF \times \scrF$ and mapping them
into locally integrable functions, and assume that we have the duality identities of Section \ref{ss:duality} for all $f_1, g_1, h_1 \in \scrF_n$ and $f_2, g_2, h_2 \in \scrF_m$.

\subsection{Kernel representations and SIOs}\label{ss:structural}
Let $f_i = f_i^1 \otimes f_i^2$, where $f_i^1 \in \scrF_n$ and $f_i^2 \in \scrF_m$.
\subsubsection{Full kernel representation}\label{ss:FK}
Here we assume that spt$\, f_i^1 \cap \textup{spt}\,f_j^1 = \emptyset$ for some $i,j$, \emph{and} spt$\, f_{i'}^2 \cap \textup{spt}\,f_{j'}^2 = \emptyset$ for some $i',j'$.
In this case we demand that
$$
\langle T(f_1, f_2), f_3\rangle = \iint_{\R^{n+m}} \iint_{\R^{n+m}} \iint_{\R^{n+m}} K(x,y,z)f_1(y) f_2(z) f_3(x)\ud x \ud y \ud z,
$$
where
$
K \colon (\R^{n+m})^3 \setminus \{ (x,y,z) \in (\R^{n+m})^3\colon x_1 = y_1 = z_1 \textup{ or } x_2 = y_2 = z_2\} \to \C
$
is a kernel satisfying a set of estimates which we specify next. Note that this implies kernel representations also for $T^{1*}$, $T^{2*}$,
$T^{1*}_1$, $T^{2*}_1$, $T^{1*}_2$, $T^{2*}_2$, $T^{1*, 2*}_{1,2}$ and $T^{1*, 2*}_{2,1}$. We denote their kernels by $K^{1*}$, $K^{2*}$ and so forth. They
all have an obvious formula using $K$.

Let $\alpha \in (0,1]$. The kernel $K$ is assumed to satisfy the size estimate
\begin{displaymath}
|K(x,y,z)| \le C\frac{1}{(|x_1-y_1|+|x_1-z_1|)^{2n}}\frac{1}{(|x_2-y_2|+|x_2-z_2|)^{2m}},
\end{displaymath}
the H\"older estimate
\begin{align*}
|K(x,y,z) - K((&x_1, x_2'), y,z) - K((x_1', x_2), y,z) + K(x', y,z)| \\
&\le C\frac{|x_1-x_1'|^{\alpha}}{(|x_1-y_1|+|x_1-z_1|)^{2n+\alpha}}  \frac{|x_2-x_2'|^{\alpha}}{(|x_2-y_2|+|x_2-z_2|)^{2m+\alpha}} 
\end{align*}
whenever $|x_1 - x_1'| \le \max(|x_1-y_1|, |x_1-z_1|)/2$ and $|x_2-x_2'| \le \max(|x_2-y_2|, |x_2-z_2|)/2$,
and the mixed H\"older and size estimate
\begin{align*}
|K(x,y,z) - K((x_1', x_2), y,z)| \le C\frac{|x_1-x_1'|^{\alpha}}{(|x_1-y_1|+|x_1-z_1|)^{2n+\alpha}} \frac{1}{(|x_2-y_2|+|x_2-z_2|)^{2m}}
\end{align*}
whenever $|x_1 - x_1'| \le \max(|x_1-y_1|, |x_1-z_1|)/2$. These estimates are also assumed from the kernels $K^{1*}$, $K^{2*}$, and so forth.

\subsubsection{Partial kernel representations}\label{sec:partial}
Suppose that spt$\, f_i^1 \cap \textup{spt}\,f_j^1 = \emptyset$ for some $i,j$. Then we assume that
$$
\langle T(f_1, f_2), f_3\rangle = \int_{\R^{n}} \int_{\R^{n}} \int_{\R^{n}} K_{f_1^2, f_2^2, f_3^2}(x_1,y_1,z_1)f_1^1(y_1) f_2^1(z_1) f_3^1(x_1)\ud x_1 \ud y_1 \ud z_1,
$$
where
$
K_{f_1^2, f_2^2, f_3^2} \colon \R^{3n} \setminus \{ (x_1,y_1,z_1) \in \R^{3n} \colon x_1 = y_1 = z_1\} \to \C
$
is a kernel satisfying the estimates of a standard bilinear Calder\'on--Zygmund kernel in $\R^n$, but with a constant depending on the functions $f_1^2, f_2^2, f_3^2$.
This means that we have the size condition
$$
|K_{f_1^2, f_2^2, f_3^2}(x_1,y_1,z_1)| \le C(f_1^2, f_2^2, f_3^2) \frac{1}{(|x_1-y_1| + |x_1-z_1|)^{2n}}
$$
and the H\"older estimate
$$
|K_{f_1^2, f_2^2, f_3^2}(x_1,y_1,z_1) - K_{f_1^2, f_2^2, f_3^2}(x_1',y_1,z_1)| \le C(f_1^2, f_2^2, f_3^2)
\frac{|x_1-x_1'|^{\alpha}}{(|x_1-y_1|+|x_1-z_1|)^{2n+\alpha}} 
$$
whenever $|x_1 - x_1'| \le \max(|x_1-y_1|, |x_1-z_1|)/2$. The analogous H\"older estimates in the $y_1$ and $z_1$ slots are also
assumed.

We assume that
$$
C(1_J, 1_J, 1_J) + C(a_J, 1_J, 1_J) + C(1_J, a_J, 1_J) + C(1_J, 1_J, a_J) \le C|J|
$$
for all cubes $J \subset \R^m$
and all functions $a_J \in \scrF_m$ satisfying $a_J = 1_Ja_J$, $|a_J| \le 1$ and $\int a_J = 0$.
Analogous partial kernel representation is assumed when spt$\, f_{i'}^2 \cap \textup{spt}\,f_{j'}^2 = \emptyset$ for some $i',j'$.

\begin{defn}\label{defn:SIO}
If $T$ is a bilinear operator with full and partial kernel structure as defined in Section \ref{ss:FK}
and Section \ref{sec:partial}, respectively, we call $T$ a bilinear bi-parameter SIO.
\end{defn}

\subsection{Reformulation of a bilinear bi-parameter SIO}\label{sec:reform}
We present a vector-valued reformulation of our SIOs, which is useful in certain operator-valued
considerations related to mixed-norm estimates and is also needed in the proof of Proposition \ref{prop:Journe}.

Suppose $T$ is a bilinear bi-parameter SIO. We want to define for $x_1, z_1, y_1 \in \R^{n}$ (such that we don't have $x_1 = y_1 = z_1$) a bilinear SIO $U_1(x_1, y_1, z_1)$ in $\R^m$ so that we have the properties (1) and (2) below.

(1) If $f_1, f_2, f_3 \colon \R^{n+m} \to \C$ are functions in $\scrF$
so that for some $i,j$ we have $\textup{spt}_{\R^n} \,f_i \cap \textup{spt}_{\R^n}\, f_j = \emptyset$, then we have
$$
\langle T(f_1, f_2), f_3\rangle = \int_{\R^n} \int_{\R^n} \int_{\R^n} \bla U_1(x_1, y_1, z_1)(f_1(y_1, \cdot), f_2(z_1, \cdot)), f_3(x_1, \cdot)\bra\ud x_1 \ud y_1 \ud z_1.
$$
Here $\textup{spt}_{\R^n} \,f$ is the closure of those $x_1$ for which there exists $x_2$ so that $f(x_1, x_2) \ne 0$.

(2) For a bilinear SIO $U$ in $\R^m$ denote
\begin{align*}
\|U\| &:= \|U\|_{\textup{CZ}_{\alpha}} + \|U(1,1)\|_{\BMO(\R^m)} + \|U^{1*}(1,1)\|_{\BMO(\R^m)}  \\ &+ \|U^{2*}(1,1)\|_{\BMO(\R^m)} + \sup\{ |\langle U(1_V, 1_V), 1_V\rangle| \colon V \subset \R^m \textup{ cube}\},
\end{align*}
where $\|U\|_{\textup{CZ}_{\alpha}}$ denotes the best constant in the bilinear size and $\alpha$-H\"older estimates (meaning standard bilinear kernel estimates like
in Section \ref{sec:partial}).
We demand the natural size and H\"older estimates in this norm, for instance that
\begin{equation*}
\|U_1(x_1, y_1, z_1) - U_1(x_1', y_1, z_1)\| \lesssim \frac{|x_1-x_1'|^{\alpha}}{(|x_1-y_1| + |x_1-z_1|)^{2n+\alpha}}
\end{equation*}
whenever $|x_1-x_1'| \le \max(|x_1-y_1|, |x_1-z_1|)/2$. Of course,
we also want to define $U_2(x_2, y_2, z_2)$ (a bilinear SIO on $\R^n$) analogously.

In the bi-parameter setting Journ\'e \cite{Jo2} originally used an operator-valued formulation in this spirit.
The proof that in the linear bi-parameter setting assumptions like in \cite{Ma1} with partial and full kernels
imply operator-valued type assumptions like in \cite{Jo2} is by A. Grau de la Herr\'an \cite{Grau}.
This argument can be adapted with ease to our bilinear bi-parameter setting to give the desired $U_1$ satisfying (1) and (2) above. We omit the details.

The conditions in (2) are only useful in conjunction with a quantitative bilinear $T1$ theorem. A suitable version can at least be found in \cite{LMOV}.
Let $U$ be a bilinear SIO in $\R^m$ with $\|U\|_{\textup{CZ}_{\alpha}} < \infty$. Then we have
$$
\|U\|_{L^p(\R^m) \times L^q(\R^m) \to L^r(\R^m)} \lesssim \|U\|
$$
for all $1 < p, q \le \infty$ and $1/2 < r < \infty$ satisfying $1/p + 1/q = 1/r$. Moreover, we have
$
\|U\|_{L^{\infty}(\R^m) \times L^{\infty}(\R^m) \to \BMO(\R^m)} \lesssim \|U\|
$
and
$
\|U\|_{L^1(\R^m) \times L^1(\R^m) \to L^{1/2,\infty}(\R^m)} \lesssim \|U\|.
$

\subsection{Boundedness and cancellation assumptions and CZOs}
\subsubsection*{Weak boundedness assumptions}
We say that $T$ satisfies the weak boundedness property if
$
|\langle T(1_I \otimes 1_J, 1_I \otimes 1_J), 1_I \otimes 1_J\rangle| \le C|I| |J|
$
for all cubes $I \subset \R^n$ and $J \subset \R^m$.

\subsubsection*{Diagonal BMO assumption}
SIO $T$ satisfies the diagonal BMO assumption if the following holds. If $I \subset \R^n$ and $J \subset \R^m$ are cubes
and $a_I \in \scrF_n$ and $a_J \in \scrF_m$ are functions satisfying $a_I = 1_Ia_I$, $|a_I| \le 1$ and $\int a_I = 0$ and
$a_J = 1_Ja_J$, $|a_J| \le 1$ and $\int a_J = 0$, then we have
\begin{align*}
& |\langle T(a_I \otimes 1_J, 1_I \otimes 1_J), 1_I \otimes 1_J\rangle| + |\langle T(1_I \otimes 1_J, a_I \otimes 1_J), 1_I \otimes 1_J\rangle| \\
& + |\langle T(1_I \otimes 1_J, 1_I \otimes 1_J), a_I \otimes 1_J\rangle| + |\langle T(1_I \otimes a_J, 1_I \otimes 1_J), 1_I \otimes 1_J\rangle| \\
&+|\langle T(1_I \otimes 1_J, 1_I \otimes a_J), 1_I \otimes 1_J\rangle| + |\langle T(1_I \otimes 1_J, 1_I \otimes 1_J), 1_I \otimes a_J\rangle| \le C|I| |J|.
\end{align*}

\subsubsection*{Product BMO assumption}
SIO $T$ satisfies the product BMO assumption if it holds
$S(1,1) \in \BMO_{\textup{prod}}(\R^{n+m})$ for all the nine choices
$
S \in \{T, T^{1*}, T^{2*}, T^{1*}_1, T^{2*}_1, T^{1*}_2, T^{2*}_2, T^{1*, 2*}_{1,2}, T^{1*, 2*}_{2,1}\}.
$
\begin{rem}
Recall that $T(1,1) \in \BMO_{\textup{prod}}(\R^{n+m})$ means that $\|T(1,1)\|_{\BMO(\calD)} \lesssim 1$
for all $\calD = \calD^n \times \calD^m$. To make sense of this, we need only to define the pairings
$\langle T(1,1), h_I \otimes h_J\rangle$. This can be done as follows.
It is by definition the sum of the following nine terms:
\begin{equation}\label{eq:FullKernelTerms}
\begin{split}
&\langle T(1_{(3I)^c} \otimes 1_{(3J)^c}, 1 \otimes 1 ), h_I \otimes h_J \rangle,  \,\,
\langle T(1_{(3I)^c} \otimes 1_{3J}, 1 \otimes 1_{(3J)^c} ), h_I \otimes h_J \rangle, \\
&\langle T(1_{3I} \otimes 1_{(3J)^c}, 1_{(3I)^c} \otimes 1), h_I \otimes h_J \rangle, \,\,
\langle T(1_{3I} \otimes 1_{3J}, 1_{(3I)^c} \otimes  1_{(3J)^c}), h_I \otimes h_J \rangle,
\end{split}
\end{equation}
\begin{equation}\label{eq:PartialKernelTerms}
\begin{split}
& \langle T(1_{(3I)^c} \otimes 1_{3J}, 1 \otimes 1_{3J} ), h_I \otimes h_J \rangle, \,\,
\langle T(1_{3I} \otimes 1_{3J}, 1_{(3I)^c} \otimes 1_{3J} ), h_I \otimes h_J \rangle, \\
&\langle T(1_{3I} \otimes 1_{(3J)^c}, 1_{3I} \otimes 1), h_I \otimes h_J \rangle, \,\,
\langle T(1_{3I} \otimes 1_{3J}, 1_{3I} \otimes 1_{(3J)^c}), h_I \otimes h_J \rangle,
\end{split}
\end{equation}
and $\langle T(1_{3I} \otimes 1_{3J}, 1_{3I} \otimes 1_{3J} ), h_I \otimes h_J \rangle$.
However, only the last term is defined from the get go.
The terms in \eqref{eq:FullKernelTerms} can be defined in a natural way using the full kernel of $T$ and the terms in 
\eqref{eq:PartialKernelTerms} via the partial kernels.
\end{rem}

\begin{defn}\label{defn:CZO}
A bilinear bi-parameter SIO $T$
satisfying the weak boundedness property, the diagonal BMO assumption and the product BMO assumption is called a bilinear bi-parameter Calder\'on--Zygmund operator (CZO).
\end{defn}

The following proposition can be shown quite similarly to the linear case -- for a modern proof in the linear setting see both \cite{HM} and \cite{Grau}.
The version below requires the known John--Nirenberg for dyadic product BMO. We omit the details.
\begin{prop}\label{prop:Journe}
If $T$ is a bilinear bi-parameter SIO and for some $r > 1$ we have
$$
\int_{\Omega} |T(1_{\Omega}, 1_{\Omega})|^r \lesssim |\Omega|
$$
for all finite unions of rectangles $\Omega$, then $T(1, 1) \in \BMO_{\textup{prod}}(\R^{n+m})$.
\end{prop}
This is key for the statements made in the Introduction, where boundedness in one tuple (for $T$ and all its adjoints)
implies boundedness in the full range.

\section{DMOs: shifts, partial paraproducts and full paraproducts}\label{sec:defbilinbiparmodel}
We define the dyadic model operators (DMOs) that appear in the bilinear bi-parameter representation theorem.
Let $f_1, f_2 \colon \R^{n+m} \to \C$ be two given functions. The operators are here defined in some fixed grids $\calD^n$ and $\calD^m$.
\subsection{Bilinear bi-parameter shifts}\label{ss:bilinbiparshiftScalar}
For triples of integers $k = (k_1, k_2, k_3)$, $k_1, k_2, k_3 \ge 0$, and $v = (v_1, v_2, v_3)$, $v_1, v_2, v_3 \ge 0$,
and cubes $K \in \calD^n$ and $V \in \calD^m$, define
\begin{align*}
A_{K, k}^{V, v}(f_1,f_2) = \sum_{\substack{I_1, I_2, I_3 \in \calD^n \\ I_1^{(k_1)} = I_2^{(k_2)} = I_3^{(k_3)} = K}} 
&\sum_{\substack{J_1, J_2, J_3 \in \calD^m \\ J_1^{(v_1)} = J_2^{(v_2)} = J_3^{(v_3)} = V}} a_{K, V, (I_i), (J_j)} \\
&\times \langle f_1, h_{I_1} \otimes h_{J_1}\rangle \langle f_2,  h_{I_2} \otimes h_{J_2}\rangle  h_{I_3}^0 \otimes h_{J_3}^0.
\end{align*}
We also demand that the scalars $a_{K, V, (I_i), (J_j)}$ satisfy the estimate
$$
|a_{K, V, (I_i), (J_j)}| \le \frac{|I_1|^{1/2} |I_2|^{1/2}|I_3|^{1/2}}{|K|^2} \frac{|J_1|^{1/2} |J_2|^{1/2}|J_3|^{1/2}}{|V|^2}. 
$$ 
A shift of complexity $(k,v)$ of a particular form (the non-cancellative Haar functions are in certain positions) is
$$
S_{k}^{v}(f_1,f_2) = \sum_{K \in \calD^n} \sum_{V \in \calD^m} A_{K, k}^{V, v}(f_1,f_2).
$$
An operator of the above form, but having the non-cancellative Haar functions $h_I^0$ and $h_J^0$ in some of
the other slots, is also a shift. So there are shifts of nine different types, and we could e.g. also have for all $K, V$ that
\begin{align*}
A_{K, k}^{V, v}(f_1,f_2) = \sum_{\substack{I_1, I_2, I_3 \in \calD^n \\ I_1^{(k_1)} = I_2^{(k_2)} = I_3^{(k_3)} = K}} 
&\sum_{\substack{J_1, J_2, J_3 \in \calD^m \\ J_1^{(v_1)} = J_2^{(v_2)} = J_3^{(v_3)} = V}} a_{K, V, (I_i), (J_j)} \\
&\times \langle f_1, h_{I_1}^0 \otimes h_{J_1}\rangle \langle f_2,  h_{I_2} \otimes h_{J_2}\rangle  h_{I_3} \otimes h_{J_3}^0.
\end{align*}

\subsection{Bilinear bi-parameter partial paraproducts}
First, we need bilinear one-parameter paraproducts. Let $b \colon \R^m \to \C$ be a function and define
\begin{equation*}
A_b(g_1, g_2) := \sum_{V \in \calD^m} \langle b, h_V\rangle \langle g_1 \rangle_V \langle g_2 \rangle_V h_V,
\end{equation*}
where $g_i \colon \R^m \to \C$. An operator $\pi_b$ is called a dyadic bilinear paraproduct in $\R^m$ if it is
of the form $A_b$, $A_b^{1*}$ or $A_b^{2*}$.

Let $k = (k_1, k_2, k_3)$, $k_1, k_2, k_3 \ge 0$. For each $K, I_1, I_2, I_3 \in \calD^n$ we are given a function $b_{K, I_1, I_2, I_3} \colon \R^m \to \C$ such that
$
\| b_{K, I_1, I_2, I_3} \|_{\BMO(\R^m)} \le \frac{|I_1|^{1/2} |I_2|^{1/2}|I_3|^{1/2}}{|K|^2}.
$
A partial paraproduct of complexity $k$ of a particular form is
$$
P_k(f_1, f_2) = \sum_{K \in \calD^n} \sum_{\substack{I_1, I_2, I_3 \in \calD^n \\ I_1^{(k_1)} = I_2^{(k_2)} = I_3^{(k_3)} = K}} h_{I_3}^0 \otimes
\pi_{b_{K, I_1, I_2, I_3}}(\langle f_1, h_{I_1} \rangle_1, \langle f_2, h_{I_2} \rangle_1),
$$ 
where $\pi_{b_{K, I_1, I_2, I_3}} $denotes a bilinear paraproduct in $\R^m$, and is of the same form for all $K,I_1, I_2, I_3$.
Again, an operator of the above form, but having the non-cancellative Haar function $h_I^0$ in some other slot, is also a partial paraproduct. 
Therefore, we have nine different possibilities as usual (the bilinear paraproducts can be of one of the three different types, and the non-cancellative Haar function in $\R^n$ can appear in one of the three slots).

We also have partial paraproducts with shift structure in $\R^m$ and paraproducts in $\R^n$.

\subsection{Bilinear bi-parameter full paraproducts}
Let $b\colon \R^{n+m} \to \C$ be a function with $\|b\|_{\BMO_{\textup{prod}}(\R^{n+m})} \le 1$. A full paraproduct $\Pi_b$ of a particular form is
$$
\Pi_b(f_1, f_2) = \Pi(f_1, f_2) = \sum_{\substack{K \in \calD^n \\ V \in \calD^m}} \lambda_{K,V}^b \langle f_1 \rangle_{K \times V} \langle f_2 \rangle_{K \times V} h_K \otimes h_V,
$$
where the function $b$ determines the coefficients $\lambda_{K,V}^b$ via the formula $\lambda_{K,V}^b := \langle b, h_K \times h_V \rangle$.
Again, an operator of the above form, but having the cancellative Haar functions $h_K$ or $h_V$ in some other slots, is also a full paraproduct. There are nine
different cases as the Haar functions present in the coefficients $\lambda_{K,V}^b$ are not allowed to move, i.e. we always have $\lambda_{K,V}^b := \langle b, h_K \times h_V \rangle$. For example, $\Pi_b$ could also be of the form
$$
\Pi_b(f_1, f_2) = \sum_{\substack{K \in \calD^n \\ V \in \calD^m}} \lambda_{K,V}^b \langle f_1 \rangle_{K \times V} \Big\langle f_2, \frac{1_K}{|K|} \otimes h_V \Big\rangle h_K \otimes \frac{1_V}{|V|}.
$$

\section{The representation theorem}\label{sec:ProofOfRep}
We need the following notation regarding random dyadic grids.
Let $\mathcal{D}_0^n$ and $\mathcal{D}_0^m$ denote the standard dyadic grids on $\R^n$ and $\R^m$ respectively.
For $\omega_1 = (\omega^i_1)_{i \in \Z} \in (\{0,1\}^n)^{\Z}$, $\omega_2 = (\omega^i_2)_{i \in \Z} \in(\{0,1\}^m)^{\Z}$, $I \in \calD^n_0$ and $J \in \calD^m_0$ denote
$$
I + \omega_1 := I + \sum_{i:\, 2^{-i} < \ell(I)} 2^{-i}\omega_1^i \qquad \textup{and} \qquad J + \omega_2 := J + \sum_{i:\, 2^{-i} < \ell(J)} 2^{-i}\omega_2^i.
$$
Then we define the random lattices
$$\calD^n_{\omega_1} = \{I + \omega_1\colon I \in \calD^n_0\} \qquad \textup{and} \qquad
\calD^m_{\omega_2} = \{J + \omega_2\colon J \in \calD^m_0\}.
$$
There is a natural probability product measure $\mathbb{P}_{\omega_1}$ in $(\{0,1\}^n)^{\Z}$ and $\mathbb{P}_{\omega_2}$ in $(\{0,1\}^m)^{\Z}$.
We set $\omega = (\omega_1, \omega_2) \in (\{0,1\}^n)^{\Z} \times (\{0,1\}^m)^{\Z}$,
and denote the expectation over the product probability space by $\E_{\omega} = \E_{\omega_1, \omega_2} = \E_{\omega_1} \E_{\omega_2} = \E_{\omega_2} \E_{\omega_1} =
\iint \ud \mathbb{P}_{\omega_1} \ud \mathbb{P}_{\omega_2}$.

We also set $\calD_0 = \calD^n_0 \times \calD^m_0$. Given $\omega = (\omega_1, \omega_2)$ and $R = I \times J \in \calD_0$ we may set
$$
R + \omega = (I+\omega_1) \times (J+\omega_2) \qquad \textup{and} \qquad \calD_{\omega} = \{R + \omega\colon \, R \in \calD_0\}.
$$
\begin{thm}\label{thm:rep}
Suppose $T$ is a bilinear bi-parameter CZO. Then
$T$ can be extended to act on bounded and compactly supported functions
 $f_i \colon \R^{n+m} \to \C$ (or e.g. $L^3(\R^{n+m})$ functions), and we have
$$
\langle T(f_1,f_2), f_3\rangle = C_T \mathbb{E}_{\omega}\mathop{\sum_{k = (k_1, k_2, k_3) \in \Z_+^3}}_{v = (v_1, v_2, v_3) \in \Z_+^3} \alpha_{k, v} 
\sum_{u}
\bla U^{v}_{k, u, \mathcal{D}_{\omega}}(f_1, f_2), f_3 \bra,
$$
where $C_T \lesssim 1$, $\alpha_{k, v} = 2^{- \alpha \max k_i/2} 2^{- \alpha \max v_j/2}$, the summation over $u$ is finite, and
$U^{v}_{k, u, \mathcal{D}_{\omega}}$ is always either a shift of complexity $(k,v)$,
a partial paraproduct of complexity $k$ or $v$ (this requires $k= 0$ or $v=0$) or a full paraproduct (this requires $k=v=0$) associated with some product
BMO function $b$ satisfying $\|b\|_{\BMO_{\textup{prod}}} \le 1$. 
\end{thm}
\begin{proof}
We show the proof under the additional assumption that $T$ is a priori bounded, say from $L^3 \times L^3 \to L^{3/2}$.
At the end we comment why this is enough.

We show that given $N$ we have for all bounded $f_i$ with spt$\,f_i \subset B(0, c_02^{N})$ that
$$
\langle T(f_1,f_2), f_3\rangle = \lim_{M \to \infty} C_T \mathbb{E}_{\omega}\mathop{\sum_{k = (k_1, k_2, k_3) \in \Z_+^3}}_{v = (v_1, v_2, v_3) \in \Z_+^3} \alpha_{k, v} 
\sum_{u}
\bla U^{v, N, M}_{k, u, \mathcal{D}_{\omega}}(f_1, f_2), f_3 \bra,
$$
where $U^{v, N, M}_{k, u, \mathcal{D}_{\omega}}$ are dyadic model operators
so that all the appearing cubes $I_1, I_2, I_3$ and $J_1, J_2, J_3$ satisfy $2^{-M} \le \ell(I_i), \ell(J_j) \le 2^N$, and
we allow non-cancellative Haar functions in all slots on the top level $2^N$ except for the Haar functions related to paraproducts
$\pi$ (as part of a partial paraproduct) or $\Pi$ (a full paraproduct). A simple limiting argument, which we omit, gives the desired conclusion.

To get started fix $N$ and three bounded functions $f_1$, $f_2$ and $f_3$
with spt$\,f_i \subset B(0, c_02^{N})$.
Define
$$
E_{2^{-M_1}, 2^{-M_2}}^{\omega} f_1 = \sum_{\substack{I_1 \in \calD^n_{\omega_1} \\ \ell(I_1) = 2^{-M_1}}} \sum_{\substack{J_1 \in \calD^m_{\omega_2} \\ \ell(J_1) = 2^{-M_2}}}
 \langle f_1 \rangle_{I_1 \times J_1} 1_{I_1 \times J_1},
$$
and also write $E_{2^{-M}}^{\omega} f_1 = E_{2^{-M}, 2^{-M}}^{\omega} f_1$.
Using the boundedness of $T$ we have
$$
\langle T(f_1, f_2), f_3\rangle = \lim_{M \to \infty} \langle T(E_{2^{-M}}^{\omega} f_1, E_{2^{-M}}^{\omega}f_2), E_{2^{-M}}^{\omega}f_3\rangle.
$$
By dominated convergence we further have that
$$
\langle T(f_1, f_2), f_3\rangle = \lim_{M \to \infty}  \E_{\omega}  \langle T(E_{2^{-M}}^{\omega} f_1, E_{2^{-M}}^{\omega}f_2), E_{2^{-M}}^{\omega}f_3\rangle.
$$
We fix $M$ for now.

Define $\Delta_I^{\le N} = \Delta_I$ if $\ell(I) < 2^N$, $\Delta_I^{\le N} = \Delta_I + E_I$ if $\ell(I) = 2^N$ and $\Delta_I^{\le N} = 0$
if $\ell(I) > 2^N$. As $N$ is fixed for the whole argument,
we will be abusing notation and writing $\Delta_I^1 = \Delta_I^{1, \le N}$, $\Delta_J^2 = \Delta_J^{2, \le N}$ and
$\Delta_{I \times J} = \Delta_I^{1, \le N} \Delta_J^{2, \le N} = \Delta_I^1 \Delta_J^2$. Similarly, $h_I$ can stand
for $h_I$ or $h_I^0$ if $\ell(I) = 2^N$. It is only inside this proof that this can happen.

Now, with this notation we have with any $M_1, M_2$ such that $2^{-M_i} < 2^N$ that
\begin{equation}\label{eq:collapse}
E_{2^{-M_1}, 2^{-M_2}}^{\omega} f_1 = \sum_{\substack{I_1 \in \calD^n_{\omega_1}, J_1 \in \calD^m_{\omega_2} \\ 2^{-M_1} < \ell(I_1) \le 2^N
\\ 2^{-M_2} < \ell(J_1) \le 2^N
}} \Delta_{I_1 \times J_1}f_1.
\end{equation}
Using this with $M_1 = M_2 =  M$ we have for all large $M$ that
\begin{equation*}
\begin{split}
 \langle T(E_{2^{-M}}^{\omega} &f_1, E_{2^{-M}}^{\omega}f_2), E_{2^{-M}}^{\omega}f_3\rangle \\
&=\sum_{\substack{I_1, I_2, I_3 \in \calD^n_{\omega_1} \\ 
2^{-M} < \ell(I_i) \le 2^N}} 
 \sum_{\substack{J_1, J_2, J_3 \in \calD^m_{\omega_2} \\ 
 2^{-M} < \ell(J_i) \le 2^N} }
\langle T(\Delta_{I_1 \times J_1} f_1, \Delta_{I_2 \times J_2} f_2), \Delta_{I_3 \times J_3} f_3 \rangle.
\end{split}
\end{equation*}
Notice that the sums are finite as the functions have compact support. This makes all the reorganisations of sums
and interchanging $\E_{\omega}$ and sums readily true. If the reader is already wondering about the reason
for the exact support condition spt$\,f_i \subset B(0, c_02^{N})$, this is related to the fact that we want to ensure that
the smallest cubes are strictly smaller than $2^N$ in the so called separated sums below. This will become clear at the end.

We split
\begin{align*}
\sum_{\substack{I_1, I_2, I_3 \in \calD^n_{\omega_1} \\ 
2^{-M} < \ell(I_i) \le 2^N}} &= \sum_{\substack{I_3 \in \calD^n_{\omega_1} \\ 2^{-M} < \ell(I_3) \le 2^N}} \sum_{\substack{I_1 \in \calD^n_{\omega_1} \\ \ell(I_3) \leq \ell(I_1) \le 2^N  }}
\sum_{\substack{I_2 \in \calD^n_{\omega_1} \\ \ell(I_3) \leq \ell(I_2) \le 2^N}} \\
&+ \sum_{\substack{I_1 \in \calD^n_{\omega_1} \\ 2^{-M} < \ell(I_1) < 2^N}} \sum_{\substack{I_2 \in \calD^n_{\omega_1} \\ \ell(I_1) \leq \ell(I_2) \le 2^N  }}
\sum_{\substack{I_3 \in \calD^n_{\omega_1} \\ \ell(I_1) < \ell(I_3) \le 2^N }} \\
&+ \sum_{\substack{I_2 \in \calD^n_{\omega_1} \\ 2^{-M} < \ell(I_2) < 2^N}} \sum_{\substack{I_1 \in \calD^n_{\omega_1} \\ \ell(I_2) < \ell(I_1) \le 2^N }}
\sum_{\substack{I_3 \in \calD^n_{\omega_1} \\ \ell(I_2) < \ell(I_3) \le 2^N }}.
\end{align*}
We split the $J_1, J_2, J_3$ sum similarly. This gives us a splitting
$$
\langle T(E_{2^{-M}}^{\omega} f_1, E_{2^{-M}}^{\omega}f_2), E_{2^{-M}}^{\omega}f_3\rangle= \sum_{i=1}^9 \Sigma^i(\omega),
$$
where the enumeration is so that we first multiplied the first $I$ summation with all the three possible $J$ summations to get $\Sigma^1, \Sigma^2, \Sigma^3$, and so on. We consider $N$ and $M$ fixed at this point of the argument, and therefore just write $\Sigma^i$ instead of $\Sigma^i_{N,M}$.
For example, $\Sigma^1(\omega)$ equals
\begin{align*}
\sum_{\substack{I_3 \in \calD^n_{\omega_1} \\ 2^{-M} < \ell(I_3) \le 2^N}}& \sum_{\substack{I_1 \in \calD^n_{\omega_1} \\ \ell(I_3) \leq \ell(I_1) \le 2^N  }}
\sum_{\substack{I_2 \in \calD^n_{\omega_1} \\ \ell(I_3) \leq \ell(I_2) \le 2^N}} \\
& \sum_{\substack{J_3 \in \calD^m_{\omega_2} \\ 2^{-M} < \ell(J_3) \le 2^N}} \sum_{\substack{J_1 \in \calD^m_{\omega_2} \\ \ell(J_3) \leq \ell(J_1) \le 2^N  }}
\sum_{\substack{J_2 \in \calD^m_{\omega_2} \\ \ell(J_3) \leq \ell(J_2) \le 2^N}}
\langle T(\Delta_{I_1 \times J_1} f_1, \Delta_{I_2 \times J_2} f_2), \Delta_{I_3 \times J_3} f_3 \rangle.
\end{align*}
These are the nine main symmetries, which are handled relatively similarly.
We choose to deal with $\Sigma^1(\omega)$. 

\subsubsection*{Adding goodness by a collapse trick}
A cube $I \in \calD^n_{\omega_1}$ is called bad if there exists such a cube $P \in \calD^n_{\omega_1}$ that $\ell(P) \ge 2^r \ell(I)$ and 
$$
d(I, \partial P) \le \ell(I)^{\gamma_n}\ell(P)^{1-\gamma_n}.
$$
Here $\gamma_n = \alpha/(2[2n + \alpha])$, where $\alpha > 0$ appears in the kernel estimates. Otherwise a cube is called good.
We note that
$\pi_{\textrm{good},n} := \mathbb{P}_{\omega_1}(I + \omega_1 \textrm{ is good})$ is independent of the choice of $I \in \mathcal{D}_0^n$. The appearing parameter $r$ is a large enough fixed constant so that $\pi_{\textrm{good}, n} > 0$ (and also $\pi_{\textrm{good}, m} > 0$).
Moreover, for a fixed $I \in \mathcal{D}_0^n$
the set $I + \omega_1$ depends on $\omega_1^i$ with $2^{-i} < \ell(I)$, while the goodness of $I + \omega_1$ depends on $\omega_1^i$ with $2^{-i} \ge \ell(I)$. These notions are independent by the product probability structure.

The way of adding goodness like below is from Li--Martikainen--Ou--Vuorinen \cite{LMOV} -- it is a significantly cleaner way to add goodness to the smallest cubes than the
one originally used by Hyt\"onen \cite{Hy}. In fact, in this generality the original way would probably cause a tedious mess.

We want the smallest cubes (i.e. $I_3$ and $J_3$) to be good. This can be achieved with the averaged sum $\E_{\omega} \Sigma^1(\omega)$.
The following argument requires us to collapse certain sums at this point to gain enough independence. Therefore, we write
$$
\Sigma^1(\omega) = \sum_{\substack{I_3 \in \calD^n_{\omega_1} \\ 2^{-M} < \ell(I_3) \le 2^N}} \sum_{\substack{J_3 \in \calD^m_{\omega_2} \\ 2^{-M} < \ell(J_3) \le 2^N}} \langle T(E_{\ell(I_3)/2, \ell(J_3)/2}^{\omega} f_1, E_{\ell(I_3)/2, \ell(J_3)/2}^{\omega} f_2), \Delta_{I_3 \times J_3} f_3 \rangle,
$$
where we used \eqref{eq:collapse} with $2^{-M_1} = \ell(I_3)/2$ and $2^{-M_2} = \ell(J_3)/2$.

Now the whole pairing depends only on $\omega_1^i$ and $\omega_2^j$ for $2^{-i} < \ell(I_3)$ and $2^{-j} < \ell(J_3)$. However, the goodness of $I_3$
depends on $\omega_1^i$ for $2^{-i} \ge \ell(I_3)$ (and similarly for $J_3$). Using this independence we get
\begin{align*}
&\E_{\omega} \Sigma^1(\omega) \\
&= C_1 \E_{\omega} \sum_{\substack{I_3 \in \calD^n_{\omega_1, \good} \\ 2^{-M} < \ell(I_3) \le 2^N}} \sum_{\substack{J_3 \in \calD^m_{\omega_2, \good} \\ 2^{-M} < \ell(J_3) \le 2^N}}
\langle T(E_{\ell(I_3)/2, \ell(J_3)/2}^{\omega} f_1, E_{\ell(I_3)/2, \ell(J_3)/2}^{\omega} f_2), \Delta_{I_3 \times J_3} f_3 \rangle \\
&=: C_1 \E_{\omega} \Sigma^1_{\good}(\omega),
\end{align*}
where $C_1 = (\pi_{\good, n} \pi_{\good, m})^{-1} < \infty$.

We now fix an arbitrary $\omega = (\omega_1, \omega_2)$ for the rest of the argument, and study $\Sigma^1 := \Sigma^1_{\good}(\omega)$. We will also be writing $\calD^n$ instead
of $\calD^n_{\omega_1}$, and so on. Next, we re-expand the functions $E_{\ell(I_3)/2, \ell(J_3)/2}f_1$ and $E_{\ell(I_3)/2, \ell(J_3)/2}f_2$ to see that $\Sigma^1$ equals
\begin{align*}
\sum_{\substack{I_3 \in \calD^n_{\good} \\ 2^{-M} < \ell(I_3) \le 2^N}}& \sum_{\substack{I_1 \in \calD^n \\ \ell(I_3) \leq \ell(I_1) \le 2^N  }}
\sum_{\substack{I_2 \in \calD^n \\ \ell(I_3) \leq \ell(I_2) \le 2^N}} \\
& \sum_{\substack{J_3 \in \calD^m_{\good} \\ 2^{-M} < \ell(J_3) \le 2^N}} \sum_{\substack{J_1 \in \calD^m \\ \ell(J_3) \leq \ell(J_1) \le 2^N  }}
\sum_{\substack{J_2 \in \calD^m \\ \ell(J_3) \leq \ell(J_2) \le 2^N}}
\langle T(\Delta_{I_1 \times J_1} f_1, \Delta_{I_2 \times J_2} f_2), \Delta_{I_3 \times J_3} f_3 \rangle.
\end{align*}

\subsubsection*{Bilinear collapse argument}
Another collapse of sums follows -- but not the same as before. The double sum $\sum_{\substack{I_1 \in \calD^n \\ \ell(I_3) \leq \ell(I_1) \le 2^N  }} \sum_{\substack{I_2 \in \calD^n \\ \ell(I_3) \leq \ell(I_2) \le 2^N}}$ 
can be organised as
$$
\sum_{\substack{I_1 \in \calD^n \\ \ell(I_3) \leq \ell(I_1)  \le 2^N}} \sum_{\substack{I_2 \in \calD^n \\ \ell(I_1) \leq \ell(I_2) \le 2^N  }}
+\sum_{\substack{I_2 \in \calD^n \\ \ell(I_3) \leq \ell(I_2) < 2^N }} \sum_{\substack{I_1 \in \calD^n \\ \ell(I_2) < \ell(I_1) \le 2^N  }}.
$$
A similar decomposition is performed on the $J_1, J_2$ summations. This decomposes
$
\Sigma^1 = \sum_{i=1}^4 \sigma_{1}^i,
$
where the enumeration is again so that we first multiplied the first $I$ summation with the two possible $J$ summations to get $\sigma^1_{1}, \sigma^2_{1}$.
We can now collapse certain summations in all of the terms $\sigma^i_{1}$, which makes this a useful decomposition. However, the downside is that
we have the main symmetries $\Sigma^1, \ldots, \Sigma^9$, but splitting them further, like we did above in the case of $\Sigma^1$,
causes the following phenomena: whenever a paraproduct like term
appears in some $\sigma^i_{k}$, it has to be combined (i.e. summed up) with the analogous terms from certain $\sigma^j_{k}$ in order to get a simple paraproduct.

We now collapse the innermost $I$ and $J$ summations in all of the terms $\sigma_{1}^i$. This gives
\begin{align*}
\sigma_{1}^{1} =\sum_{\substack{I_3 \in \calD^n_{\good} \\ 2^{-M} < \ell(I_3) \le 2^N}}& \sum_{\substack{I_1, I_2 \in \calD^n \\ \ell(I_3) \le \ell(I_1) = 2\ell(I_2) \le 2^N}} \\
&\sum_{\substack{J_3 \in \calD^m_{\good} \\ 2^{-M} < \ell(J_3) \le 2^N}} \sum_{\substack{J_1, J_2 \in \calD^m \\ \ell(J_3) \le \ell(J_1) = 2\ell(J_2) \le 2^N}}
\langle T(\Delta_{I_1 \times J_1} f_1, E_{I_2 \times J_2} f_2), \Delta_{I_3 \times J_3} f_3 \rangle;
\end{align*}
\begin{align*}
\sigma_{1}^2 = \sum_{\substack{I_3 \in \calD^n_{\good} \\ 2^{-M} < \ell(I_3) \le 2^N}}& \sum_{\substack{I_1, I_2 \in \calD^n \\ \ell(I_3) \le \ell(I_1) = 2\ell(I_2) \le 2^N}} \\
&\sum_{\substack{J_3 \in \calD^m_{\good} \\ 2^{-M} < \ell(J_3) \le 2^N}} \sum_{\substack{J_1, J_2 \in \calD^m \\ \ell(J_3) \le \ell(J_1) = \ell(J_2) < 2^N}}
\langle T(E_{J_1}^2 \Delta_{I_1}^1 f_1, E_{I_2}^1 \Delta_{J_2}^2 f_2), \Delta_{I_3 \times J_3} f_3 \rangle; 
\end{align*}
\begin{align*}
\sigma_{1}^3  = \sum_{\substack{I_3 \in \calD^n_{\good} \\ 2^{-M} < \ell(I_3) \le 2^N}}& \sum_{\substack{I_1, I_2 \in \calD^n \\ \ell(I_3) \le \ell(I_1) = \ell(I_2) < 2^N}} \\
&\sum_{\substack{J_3 \in \calD^m_{\good} \\ 2^{-M} < \ell(J_3) \le 2^N}} \sum_{\substack{J_1, J_2 \in \calD^m \\ \ell(J_3) \le \ell(J_1) = 2\ell(J_2) \le 2^N}}
\langle T(E_{I_1}^1 \Delta_{J_1}^2 f_1, E_{J_2}^2 \Delta_{I_2}^1 f_2), \Delta_{I_3 \times J_3} f_3 \rangle;
\end{align*}
\begin{align*}
\sigma_{1}^4 = \sum_{\substack{I_3 \in \calD^n_{\good} \\ 2^{-M} < \ell(I_3) \le 2^N}}& \sum_{\substack{I_1, I_2 \in \calD^n \\ \ell(I_3) \le \ell(I_1) = \ell(I_2) < 2^N}} \\
&\sum_{\substack{J_3 \in \calD^m_{\good} \\ 2^{-M} < \ell(J_3) \le 2^N}} \sum_{\substack{J_1, J_2 \in \calD^m \\ \ell(J_3) \le \ell(J_1) = \ell(J_2) < 2^N}}
\langle T( E_{I_1 \times J_1} f_1, \Delta_{I_2 \times J_2} f_2), \Delta_{I_3 \times J_3} f_3 \rangle.
\end{align*}
\begin{comment}
\begin{align*}
\sigma_1^1 &= \sum_{\substack{I_3 \in \calD^n_{\good} \\ J_3 \in \calD^m_{\good}}}
\sum_{\substack{I_1 \in \calD^n \\ \ell(I_3) \leq \ell(I_1)  }}
\sum_{\substack{J_1 \in \calD^m \\ \ell(J_3) \leq \ell(J_1)  }} \langle T(\Delta_{I_1 \times J_1} f_1, E_{\ell(I_1)/2, \ell(J_1)/2} f_2), \Delta_{I_3 \times J_3} f_3 \rangle; \\
\sigma_1^2 &= \sum_{\substack{I_3 \in \calD^n_{\good} \\ J_3 \in \calD^m_{\good}}}
\sum_{\substack{I_1 \in \calD^n \\ \ell(I_3) \leq \ell(I_1)  }}
\sum_{\substack{J_2 \in \calD^m \\ \ell(J_3) \leq \ell(J_2)  }}
\langle T(E_{\ell(J_2)}^2 \Delta_{I_1}^1 f_1, E_{\ell(I_1)/2}^1 \Delta_{J_2}^2 f_2), \Delta_{I_3 \times J_3} f_3 \rangle; \\
\sigma_1^3 &= \sum_{\substack{I_3 \in \calD^n_{\good} \\ J_3 \in \calD^m_{\good}}}
\sum_{\substack{I_2 \in \calD^n \\ \ell(I_3) \leq \ell(I_2)  }}
\sum_{\substack{J_1 \in \calD^m \\ \ell(J_3) \leq \ell(J_1)  }} \langle T(E_{\ell(I_2)}^1 \Delta_{J_1}^2 f_1, E_{\ell(J_1)/2}^2 \Delta_{I_2}^1 f_2), \Delta_{I_3 \times J_3} f_3 \rangle; \\
\sigma_1^4 &= \sum_{\substack{I_3 \in \calD^n_{\good} \\ J_3 \in \calD^m_{\good}}}
\sum_{\substack{I_2 \in \calD^n \\ \ell(I_3) \leq \ell(I_2)  }}
\sum_{\substack{J_2 \in \calD^m \\ \ell(J_3) \leq \ell(J_2)  }}  \langle T( E_{\ell(I_2), \ell(J_2)} f_1, \Delta_{I_2 \times J_2} f_2), \Delta_{I_3 \times J_3} f_3 \rangle.
\end{align*}
\end{comment}
\subsubsection*{Bilinear bi-parameter organisation of sums}
We start dealing with $\sigma_1^1$, but due to the need to combine paraproduct like terms, this systemically requires the need to combine
the terms coming from $\sigma_1^1$ with some of the analogous terms coming from $\sigma_1^2$, $\sigma_1^3$ and $\sigma_1^4$. Therefore, we in essence
deal with all of them simultaneously.

With a fixed $I_3$ we split
\begin{align*}
\sum_{\substack{I_1, I_2 \in \calD^n \\ \ell(I_3) \le \ell(I_1) = 2\ell(I_2) \le 2^N}} =
\sum_{\substack{I_1,I_2 \in \calD^n  \\ \ell(I_3)\le \ell (I_1)=2\ell(I_2) \le 2^N \\ \max( d(I_1, I_3), d(I_2, I_3)) > \ell(I_3)^{\gamma_n} \ell(I_2)^{1-\gamma_n}}} &+
\mathop{\sum_{\substack{I_1,I_2 \in \calD^n  \\ \ell(I_3)\le \ell (I_1)=2\ell(I_2) \le 2^N \\ \max( d(I_1, I_3), d(I_2, I_3)) \le \ell(I_3)^{\gamma_n} \ell(I_2)^{1-\gamma_n}}}}_{I_1 \cap I_3 = \emptyset \textup{ or } I_1 = I_3 \textup{ or } I_2 \cap I_3 = \emptyset}  \\
&+\sum_{\substack{I_1,I_2 \in \calD^n \\  \ell(I_1) = 2\ell(I_2) \le 2^N \\ I_3 \subset I_2 \subset I_1}}.
\end{align*}
These parts are titled separated, diagonal and nested, respectively. 
A similar decomposition of  $\sum_{\substack{J_1, J_2 \in \calD^m \\ \ell(J_3) \le \ell(J_1) = 2\ell(J_2)\le 2^N}}$ is done with a fixed $J_3$. This splits
$\sigma_1^1$ into nine parts, which are titled like separated/separated, separated/diagonal, etc.

We begin with the most difficult case -- i.e. Nested/Nested. The remaining cases are easier, except for the fact that
finding a common parent (which is required for the shift structure) is easiest (as it is obvious) in the nested case. We only briefly comment on the other cases
after handling Nested/Nested.
\subsubsection*{Nested/Nested}
We have by definition that
$$
\sigma_{1, \textup{nes}, \textup{nes}}^1 = \sum_{\substack{I_2 \in \calD^n:\, \ell(I_2) < 2^N \\ I_3 \in \calD^n_{\good}: \, \ell(I_3) > 2^{-M} \\ I_3 \subset I_2}}
\sum_{\substack{J_2 \in \calD^m: \, \ell(J_2) < 2^N \\ J_3 \in \calD^m_{\good}:\, \ell(J_3) > 2^{-M} \\ J_3 \subset J_2}} \langle T(\Delta_{I_2^{(1)} \times J_2^{(1)}} f_1, E_{I_2 \times J_2} f_2), \Delta_{I_3 \times J_3} f_3 \rangle.
$$
We write
\begin{align*}
\langle T(\Delta_{I_2^{(1)} \times J_2^{(1)}} f_1, E_{I_2 \times J_2} f_2), &\Delta_{I_3 \times J_3} f_3 \rangle 
= \langle f_1, h_{I_2^{(1)}} \otimes h_{J_2^{(1)}}\rangle \langle f_2 \rangle_{I_2 \times J_2} \langle f_3, h_{I_3} \otimes h_{J_3}\rangle \\
&\times  \langle T(h_{I_2^{(1)}} \otimes h_{J_2^{(1)}}, 1_{I_2 \times J_2}),  h_{I_3} \otimes h_{J_3} \rangle.
\end{align*}
First, we split
\begin{align*}
T(h_{I_2^{(1)}} \otimes h_{J_2^{(1)}}, 1_{I_2 \times J_2}) &=  
T(h_{I_2^{(1)}} \otimes h_{J_2^{(1)}}, 1_{I_2^c} \otimes 1_{J_2^c})
- T(h_{I_2^{(1)}} \otimes h_{J_2^{(1)}}, 1 \otimes 1_{J_2^c}) \\
&-T(h_{I_2^{(1)}} \otimes h_{J_2^{(1)}}, 1_{I_2^c} \otimes 1) 
+ T(h_{I_2^{(1)}} \otimes h_{J_2^{(1)}}, 1).
\end{align*}
Using the identity
$$
h_{I_2^{(1)}} = \langle h_{I_2^{(1)}} \rangle_{I_2} + s_{I_2}, \qquad s_{I_2} := 1_{I_2^c}[ h_{I_2^{(1)}} - \langle h_{I_2^{(1)}} \rangle_{I_2}],
$$
we further write
$$
- T(h_{I_2^{(1)}} \otimes h_{J_2^{(1)}}, 1 \otimes 1_{J_2^c}) = - T(s_{I_2} \otimes h_{J_2^{(1)}}, 1 \otimes 1_{J_2^c}) - \langle h_{I_2^{(1)}}\rangle_{I_2}T(1\otimes h_{J_2^{(1)}}, 1 \otimes 1_{J_2^c})
$$
and
$$
-T(h_{I_2^{(1)}} \otimes h_{J_2^{(1)}}, 1_{I_2^c} \otimes 1) = -T(h_{I_2^{(1)}} \otimes s_{J_2}, 1_{I_2^c} \otimes 1) - \langle h_{J_2^{(1)}}\rangle_{J_2} T(h_{I_2^{(1)}} \otimes 1, 1_{I_2^c} \otimes 1)
$$
and
\begin{align*}
T(h_{I_2^{(1)}} \otimes h_{J_2^{(1)}}, 1) &= T(s_{I_2} \otimes s_{J_2}, 1)
+ \langle h_{I_2^{(1)}} \rangle_{I_2} T(1 \otimes s_{J_2}, 1) \\
&+ \langle h_{J_2^{(1)}} \rangle_{J_2} T(s_{I_2} \otimes 1, 1) 
+ \langle h_{I_2^{(1)}} \rangle_{I_2} \langle h_{J_2^{(1)}} \rangle_{J_2}  T(1,1).
\end{align*}

We have split $\sigma_{1, \textup{nes}, \textup{nes}}^1$ into four essentially different parts
$
\sigma_{1, \textup{nes}, \textup{nes}}^1 = \sum_{u=1}^4 \sigma_{1, \textup{nes}, \textup{nes}, u}^1,
$
where we agree that $ \sigma_{1, \textup{nes}, \textup{nes}, 1}^1$ corresponds to
\begin{align*}
T(h_{I_2^{(1)}} \otimes h_{J_2^{(1)}}, 1_{I_2^c} \otimes 1_{J_2^c}) &- T(s_{I_2} \otimes h_{J_2^{(1)}}, 1 \otimes 1_{J_2^c}) \\
& -T(h_{I_2^{(1)}} \otimes s_{J_2}, 1_{I_2^c} \otimes 1) + T(s_{I_2} \otimes s_{J_2}, 1),
\end{align*}
$ \sigma_{1, \textup{nes}, \textup{nes}, 2}^1$ corresponds to
$$
\langle h_{I_2^{(1)}}\rangle_{I_2}[-T(1\otimes h_{J_2^{(1)}}, 1 \otimes 1_{J_2^c}) + T(1 \otimes s_{J_2}, 1)],
$$
$ \sigma_{1, \textup{nes}, \textup{nes}, 3}^1$ corresponds to
$$
\langle h_{J_2^{(1)}}\rangle_{J_2}[-T(h_{I_2^{(1)}} \otimes 1, 1_{I_2^c} \otimes 1) + T(s_{I_2} \otimes 1, 1) ],
$$
and $ \sigma_{1, \textup{nes}, \textup{nes}, 4}^1$ corresponds to
$
\langle h_{I_2^{(1)}} \rangle_{I_2} \langle h_{J_2^{(1)}} \rangle_{J_2}  T(1,1).
$

We start with the terms $\sigma_{1, \textup{nes}, \textup{nes}, 2}^1$, $\sigma_{1, \textup{nes}, \textup{nes}, 3}^1$ and $\sigma_{1, \textup{nes}, \textup{nes}, 4}^1$.
The first two are handled similarly (they both produce partial paraproducts), and we only deal with $\sigma_{1, \textup{nes}, \textup{nes}, 3}^1$. The term $\sigma_{1, \textup{nes}, \textup{nes}, 4}^1$ produces a  full paraproduct.
\subsubsection*{The term $\sigma_{1, \textup{nes}, \textup{nes}, 3}^1$}
We are looking at the term
\begin{align*}
&\sigma_{1, \textup{nes}, \textup{nes}, 3}^1 = \sum_{\substack{I_2 \in \calD^n:\, \ell(I_2) < 2^N \\ I_3 \in \calD^n_{\good}: \, \ell(I_3) > 2^{-M} \\ I_3 \subset I_2}}
\sum_{\substack{J_3 \in \calD^m_{\good} \\ 2^{-M} < \ell(J_3) < 2^N}} 
\langle f_3, h_{I_3} \otimes h_{J_3}\rangle   \langle -T(h_{I_2^{(1)}} \otimes 1, 1_{I_2^c} \otimes 1)\\
 &+ T(s_{I_2} \otimes 1, 1), h_{I_3} \otimes h_{J_3}\rangle 
 \sum_{\substack{k \ge 0\\ 2^k \le 2^{N-1}/\ell(J_3)}}
 \langle h_{J_3^{(k+1)}} \rangle_{J_3^{(k)}} \langle
 \langle f_1, h_{I_2^{(1)}}\rangle_1, h_{J_3^{(k+1)}}\rangle  \langle f_2 \rangle_{I_2 \times J_3^{(k)}}.
\end{align*}
Notice that 
\begin{align*}
\langle h_{J_3^{(k+1)}} \rangle_{J_3^{(k)}} \langle \langle f_1, h_{I_2^{(1)}}\rangle_1, h_{J_3^{(k+1)}}\rangle &= \langle \Delta_{J_3^{(k+1)}} \langle f_1, h_{I_2^{(1)}}\rangle_1 \rangle_{J_3^{(k)}} \\
&= \langle  \langle f_1, h_{I_2^{(1)}}\rangle_1 \rangle_{J_3^{(k)}} - \langle  \langle f_1, h_{I_2^{(1)}}\rangle_1 \rangle_{J_3^{(k+1)}}
\end{align*}
if $2^k < 2^{N-1}/\ell(J_3)$, and it equals $\langle  \langle f_1, h_{I_2^{(1)}}\rangle_1 \rangle_{J_3^{(k)}}$  if $2^k = 2^{N-1}/\ell(J_3)$.

The sum does not collapse. To achieve this, we need to sum $\sigma_{1, \textup{nes}, \textup{nes}, 3}^1$ with the term
\begin{align*}
&\sigma_{1, \textup{nes}, \textup{nes}, 3}^2 =   \sum_{\substack{I_2 \in \calD^n:\, \ell(I_2) < 2^N \\ I_3 \in \calD^n_{\good}: \, \ell(I_3) > 2^{-M} \\ I_3 \subset I_2}}
\sum_{\substack{J_3 \in \calD^m_{\good} \\ 2^{-M} < \ell(J_3) < 2^{N-1}}} 
\langle f_3, h_{I_3} \otimes h_{J_3}\rangle     \langle -T(h_{I_2^{(1)}} \otimes 1, 1_{I_2^c} \otimes 1)\\
 &+ T(s_{I_2} \otimes 1, 1), h_{I_3} \otimes h_{J_3}\rangle 
\sum_{\substack{k \ge 0\\ 2^k < 2^{N-1}/\ell(J_3)}} \langle  \langle f_1, h_{I_2^{(1)}}\rangle_1 \rangle_{J_3^{(k+1)}}
[\langle f_2 \rangle_{I_2 \times J_3^{(k)}} - \langle f_2 \rangle_{I_2 \times J_3^{(k+1)}}].
\end{align*}
This formula for $\sigma_{1, \textup{nes}, \textup{nes}, 3}^2$ can be seen by calculating similarly as we did with $\sigma_{1, \textup{nes}, \textup{nes}, 3}^1$.
Noting the key cancellation the result is that $\sigma_{1, \textup{nes}, \textup{nes}, 3}^1 + \sigma_{1, \textup{nes}, \textup{nes}, 3}^2$ equals
\begin{align*}
\sum_{\substack{I_2 \in \calD^n \\ I_3 \in \calD^n_{\good} \\  \ell(I_2) < 2^N \\ \ell(I_3) > 2^{-M} \\ I_3 \subset I_2}}&
\sum_{\substack{J_3 \in \calD^m_{\good} \\ 2^{-M} < \ell(J_3) < 2^{N}}}
\langle f_3, h_{I_3} \otimes h_{J_3}\rangle   \langle -T(h_{I_2^{(1)}} \otimes 1, 1_{I_2^c} \otimes 1)
+ T(s_{I_2} \otimes 1, 1), h_{I_3} \otimes h_{J_3}\rangle  \\
&\times\sum_{\substack{k \ge 0\\ 2^k \le 2^{N-1}/\ell(J_3)}}
[\langle  \langle f_1, h_{I_2^{(1)}}\rangle_1 \rangle_{J_3^{(k)} }\langle f_2 \rangle_{I_2 \times J_3^{(k)}} - \langle  \langle f_1, h_{I_2^{(1)}}\rangle_1 \rangle_{J_3^{(k+1)}} \langle f_2 \rangle_{I_2 \times J_3^{(k+1)}}],
\end{align*}
where we understand that $[\langle  \langle f_1, h_{I_2^{(1)}}\rangle_1 \rangle_{J_3^{(k)} }\langle f_2 \rangle_{I_2 \times J_3^{(k)}} - \langle  \langle f_1, h_{I_2^{(1)}}\rangle_1 \rangle_{J_3^{(k+1)}} \langle f_2 \rangle_{I_2 \times J_3^{(k+1)}}]$ is replaced with $\langle  \langle f_1, h_{I_2^{(1)}}\rangle_1 \rangle_{J_3^{(k)}} \langle f_2 \rangle_{I_2 \times J_3^{(k)}}$ if $2^k = 2^{N-1}/\ell(J_3)$. Now, this last sum collapses to
$\langle  \langle f_1, h_{I_2^{(1)}}\rangle_1 \rangle_{J_3 }\langle f_2 \rangle_{I_2 \times J_3}$. Define
$$
b_{I_2, I_3} = C^{-1} \Big(\frac{\ell(I_2)}{\ell(I_3)}\Big)^{\alpha/2} |I_2|^{-1/2} \langle  -T(h_{I_2^{(1)}} \otimes 1, 1_{I_2^c} \otimes 1) + T(s_{I_2} \otimes 1, 1), h_{I_3} \rangle_1
$$
whenever $I_2 \in \calD^n$, $I_3 \in \calD^n_{\good}$ and $I_3 \subset I_2$, and otherwise set $b_{I_2, I_3} = 0$. We have that
$\sigma_{1, \textup{nes}, \textup{nes}, 3}^1 + \sigma_{1, \textup{nes}, \textup{nes}, 3}^2$ equals
$$
C \sum_{\substack{I_2, I_3 \in \calD^n \\ 2^{-M} < \ell(I_2), \ell(I_3) < 2^N \\ I_3 \subset I_2}}
\Big(\frac{\ell(I_3)}{\ell(I_2)}\Big)^{\alpha/2} \langle h_{I_3} \otimes \pi_{\calD^m_{\good}, b_{I_2, I_3}}^{N,M}( \langle f_1, h_{I_2^{(1)}}\rangle_1, \langle f_2, h_{I_2}^0)\rangle_1, f_3\rangle,
$$
where
$$
 \pi_{\calD^m_{\good}, b_{I_2, I_3}}^{N,M}(g_1, g_2) = \sum_{\substack{J_3 \in \calD^m_{\good} \\ 2^{-M} < \ell(J_3) < 2^{N}}}
 \langle b_{I_2, I_3}, h_{J_3} \rangle
 \langle g_1 \rangle_{J_3} \langle g_2 \rangle_{J_3} h_{J_3}.
$$
This is seen to be a sum of (truncated) partial paraproducts with indices $0, 1, k$, $k \ge 1$. This is because of the following lemma, which gives
the required BMO norm bound for the functions $b_{I_2, I_3}$. We give the proof relatively carefully as a model.
\begin{lem}\label{lem:kernel_example}
Let $I_2 \in \calD^n$, $I_3 \in \calD^n_{\good}$ and $I_3 \subset I_2$. Then we have
\begin{align*}
|\langle T(h_{I_2^{(1)}} \otimes 1, 1_{I_2^c} \otimes 1), h_{I_3} \otimes a_J\rangle| + |\langle T(s_{I_2}& \otimes 1, 1), h_{I_3} \otimes a_J\rangle| \\
&\lesssim \Big(\frac{\ell(I_3)}{\ell(I_2)}\Big)^{\alpha/2}\Big( \frac{|I_3|}{|I_2|} \Big)^{1/2} |J|
\end{align*}
for all cubes $J \subset \R^m$ and all functions $a_J$ such that spt$\,a_J \subset J$, $|a_J| \le 1$ and $\int a_J = 0$. Notice that
here $h_{I_3}$ really stands for a cancellative Haar function (as is the case above where $\ell(I_3) < 2^N$), but whether $h_{I_2^{(1)}}$
is cancellative or not does not matter.
\end{lem}
\begin{proof}
We only explicitly deal with $|\langle T(s_{I_2} \otimes 1, 1), h_{I_3} \otimes a_J\rangle|$, and estimate
\begin{align*}
|\langle T(s_{I_2} \otimes 1, 1), h_{I_3} \otimes a_J\rangle| &\le |\langle T(s_{I_2} \otimes 1_{3J}, 1 \otimes 1_{3J}), h_{I_3} \otimes a_J\rangle| \\
&+ |\langle T(s_{I_2} \otimes 1_{3J}, 1 \otimes 1_{(3J)^c}), h_{I_3} \otimes a_J\rangle| \\
&+ |\langle T(s_{I_2} \otimes 1_{(3J)^c}, 1), h_{I_3} \otimes a_J\rangle| = A_1 + A_2 + A_3.
\end{align*}

We first consider the case $\ell(I_3) < 2^{-r}\ell(I_2)$ ($r$ is related to goodness). We have
\begin{align*}
A_1 = \Big|\int_{\R^n} \int_{I_2^c} \int_{I_3} &\big[K_{1_{3J}, 1_{3J}, a_J}(x_1,y_1,z_1) \\
&- K_{1_{3J}, 1_{3J}, a_J}(c_{I_3},y_1,z_1)\big] s_{I_2}(y_1) h_{I_3}(x_1) \ud x_1 \ud y_1 \ud z_1\Big|,
\end{align*}
which gives
\begin{align*}
A_1 &\lesssim C(1_{3J}, 1_{3J}, a_J) \ell(I_3)^{\alpha} \int_{\R^n} \int_{I_2^c} \int_{I_3} \frac{|s_{I_2}(y_1) h_{I_3}(x_1)| \ud x_1\ud y_1 \ud z_1 }{(|c_{I_3} - y_1| + |c_{I_3} - z_1|)^{2n+\alpha}} \\
&\lesssim |J| |I_2|^{-1/2} |I_3|^{1/2} \ell(I_3)^{\alpha} \int_{I_2^c} \frac{\ud y_1}{|c_{I_3}-y_1|^{n+\alpha}} \lesssim \Big(\frac{\ell(I_3)}{\ell(I_2)}\Big)^{\alpha/2}\Big( \frac{|I_3|}{|I_2|} \Big)^{1/2} |J|.
\end{align*}
The last estimate used the fact that $d(I_3, I_2^c) > \ell(I_3)^{\gamma_n}\ell(I_2)^{1-\gamma_n} \ge \ell(I_3)^{1/2} \ell(I_2)^{1/2}$, which follows 
from the goodness of the cube $I_3$ and the fact that $\ell(I_3) < 2^{-r}\ell(I_2)$. The estimate for $A_2$ is very similar -- just use the full kernel representation
and a H\"older estimate of the full kernel instead. The $\R^n$ integrals are handled using exactly the same argument as above, while the $|J|$ is obtained
from the estimate
\begin{equation*}
\ell(J)^{\alpha} \int_{(3J)^c} \int_{3J} \int_J \frac{\ud x_2\ud y_2 \ud z_2 }{(|c_{J} - y_2| + |c_{J} - z_2|)^{2m+\alpha}} \lesssim |J| \ell(J)^{\alpha} \int_{(3J)^c} \frac{\ud z_2}{|c_J-z_2|^{m+\alpha}}
\lesssim |J|.
\end{equation*}
As we just disregarded the $3J$ in the $y_2$ integral above, we see that $A_3$ follows by essentially the same estimate.
We are done with the case $\ell(I_3) < 2^{-r}\ell(I_2)$.

Next, we deal with the remaining case where $\ell(I_3) \sim \ell(I_2)$. Here we use the above splitting to $A_i$, $i = 1, 2,3$, but
also further split each of these to $A_i^k$, $k = 1,2$, where $k=1$ means that we have replaced $s_{I_2}$ with $1_{3I_3}s_{I_2}$ and $k=2$ means that
we have replaced $s_{I_2}$ with $1_{(3I_3)^c}s_{I_2}$. The required estimate for $A_1^1$ follows using the partial kernel representation, the size estimate for the kernel, and the estimate
$
\int_{3I_3 \setminus I_3} \int_{I_3} \frac{\ud x_1 \ud y_1}{|x_1-y_1|^n} \lesssim |I_3|.
$
The estimate for $A_1^2$ follows using the partial kernel representation, the H\"older estimate for the kernel, and standard estimates.
All the remaining estimates follow using similar estimates, we just point out which kernel estimate is used. In $A_2^1$ use the full kernel and a mixed H\"older and size estimate, and
in $A_2^2$ use the full kernel and a H\"older estimate. The remaining terms $A_3^1$ and $A_3^2$ are handled as $A_2^1$ and $A_2^2$ respectively.
\end{proof}
\subsubsection*{The term $\sigma_{1, \textup{nes}, \textup{nes}, 4}^1$}
Here we need to sum up all the terms  $\sigma_{1, \textup{nes}, \textup{nes}, 4}^i$, $i = 1, \ldots, 4$. We first sum up $\sigma_{1, \textup{nes}, \textup{nes}, 4}^1$
and $\sigma_{1, \textup{nes}, \textup{nes}, 4}^2$. Using calculations like above we see that
\begin{align*}
\sigma_{1, \textup{nes}, \textup{nes}, 4}^1 + \sigma_{1, \textup{nes}, \textup{nes}, 4}^2 = &
\sum_{\substack{I_3 \in \calD^n_{\good} \\ 2^{-M} < \ell(I_3) < 2^{N}}}
\sum_{\substack{J_3 \in \calD^m_{\good} \\ 2^{-M} < \ell(J_3) < 2^{N}}} 
\langle f_3, h_{I_3} \otimes h_{J_3}\rangle   \langle T(1, 1), h_{I_3} \otimes h_{J_3}\rangle \\
&\sum_{\substack{k \ge 0\\ 2^k \le 2^{N-1}/\ell(I_3)}} \langle h_{I_3^{(k+1)}} \rangle_{I_3^{(k)}} \langle  \langle f_1, h_{I_3^{(k+1)}}\rangle_1 \rangle_{J_3 }\langle f_2 \rangle_{I_3^{(k)} \times J_3}.
\end{align*}
Notice that here
\begin{align*}
 \langle h_{I_3^{(k+1)}} \rangle_{I_3^{(k)}} \langle  \langle f_1, h_{I_3^{(k+1)}}\rangle_1 \rangle_{J_3}  &= \langle h_{I_3^{(k+1)}} \otimes \langle f_1, h_{I_3^{(k+1)}}\rangle_1 \rangle_{I_3^{(k)} \times J_3} \\
& = \langle \Delta_{I_3^{(k+1)}}^1 f_1 \rangle_{I_3^{(k)} \times J_3} = \langle f_1 \rangle_{I_3^{(k)} \times J_3} - \langle f_1 \rangle_{I_3^{(k+1)} \times J_3}
\end{align*}
if $2^k < 2^{N-1}/\ell(I_3)$, and it equals $\langle f_1 \rangle_{I_3^{(k)} \times J_3}$  if $2^k = 2^{N-1}/\ell(I_3)$.
In exactly the same way we see that
\begin{align*}
\sigma_{1, \textup{nes}, \textup{nes}, 4}^3 + \sigma_{1, \textup{nes}, \textup{nes}, 4}^4 = &
\sum_{\substack{I_3 \in \calD^n_{\good} \\ 2^{-M} < \ell(I_3) < 2^{N-1}}}
\sum_{\substack{J_3 \in \calD^m_{\good} \\ 2^{-M} < \ell(J_3) < 2^{N}}}  
\langle f_3, h_{I_3} \otimes h_{J_3}\rangle   \langle T(1, 1), h_{I_3} \otimes h_{J_3}\rangle \\
&\sum_{\substack{k \ge 0\\ 2^k < 2^{N-1}/\ell(I_3)}}  \langle f_1 \rangle_{I_3^{(k+1)} \times J_3} [ \langle f_2 \rangle_{I_3^{(k)} \times J_3} - \langle f_2 \rangle_{I_3^{(k+1)} \times J_3}].
\end{align*}
Noticing the cancellation we have
\begin{align*}
\sum_{i=1}^4 \sigma_{1, \textup{nes}, \textup{nes}, 4}^i =&
\sum_{\substack{I_3 \in \calD^n_{\good} \\ 2^{-M} < \ell(I_3) < 2^{N}}}
\sum_{\substack{J_3 \in \calD^m_{\good} \\ 2^{-M} < \ell(J_3) < 2^{N}}}  
\langle f_3, h_{I_3} \otimes h_{J_3}\rangle   \langle T(1, 1), h_{I_3} \otimes h_{J_3}\rangle \\
&\sum_{\substack{k \ge 0\\ 2^k \le 2^{N-1}/\ell(I_3)}} [\langle f_1 \rangle_{I_3^{(k)} \times J_3} \langle f_2 \rangle_{I_3^{(k)} \times J_3}
- \langle f_1 \rangle_{I_3^{(k+1)} \times J_3} \langle f_2 \rangle_{I_3^{(k+1)} \times J_3}],
\end{align*}
where we understand that $[\langle f_1 \rangle_{I_3^{(k)} \times J_3} \langle f_2 \rangle_{I_3^{(k)} \times J_3}
- \langle f_1 \rangle_{I_3^{(k+1)} \times J_3} \langle f_2 \rangle_{I_3^{(k+1)} \times J_3}]$ is replaced with $\langle f_1 \rangle_{I_3^{(k)} \times J_3}\langle f_2 \rangle_{I_3^{(k)} \times J_3}$ if $2^k = 2^{N-1}/\ell(I_3)$. This 
last sum collapses to $\langle f_1 \rangle_{I_3 \times J_3} \langle f_2 \rangle_{I_3 \times J_3}$. Therefore, we have shown that
$$
\sum_{i=1}^4 \sigma_{1, \textup{nes}, \textup{nes}, 4}^i = \langle \Pi_{\calD^n_{\good}, \calD^m_{\good}, T(1,1)}^{N,M} (f_1, f_2), f_3 \rangle,
$$
where $\Pi_{\calD^n_{\good}, \calD^m_{\good}, T(1,1)}^{N,M} (f_1, f_2) $ equals
$$
\sum_{\substack{I_3 \in \calD^n_{\good},\, J_3 \in \calD^m_{\good} \\ 2^{-M} < \ell(I_3), \, \ell(J_3) < 2^{N}}}  \langle T(1, 1), h_{I_3} \otimes h_{J_3}\rangle\langle f_1 \rangle_{I_3 \times J_3} \langle f_2 \rangle_{I_3 \times J_3}
h_{I_3} \otimes h_{J_3}.
$$

It remains to consider the term $\sigma_{1, \textup{nes}, \textup{nes}, 1}^1$.

\subsubsection*{The term $\sigma_{1, \textup{nes}, \textup{nes}, 1}^1$}
This term clearly produces a sum of shifts of the type $(0,1,k), (0,1,v)$, $k, v \ge 1$, provided that we check the estimate
$$
|I_2|^{-1/2} |J_2|^{-1/2} |\langle T(s_{I_2} \otimes s_{J_2}, 1), h_{I_3} \otimes h_{J_3} \rangle| \lesssim \Big(\frac{\ell(I_3)}{\ell(I_2)}\Big)^{\alpha/2} \frac{|I_3|^{1/2}}{|I_2|} 
\Big(\frac{\ell(J_3)}{\ell(J_2)}\Big)^{\alpha/2} \frac{|J_3|^{1/2}}{|J_2|}, 
$$
and the three other completely analogous estimates.
This estimate is checked using similar estimates as in Lemma \ref{lem:kernel_example} (but only the full kernel representation is needed). Therefore, we are done with this term,
and so with the whole Nested/Nested part.

\subsubsection*{Remaining parts}
As mentioned, the other parts are essentially easier than the one we just handled, namely Nested/Nested. We now briefly indicate how they are handled.
First, we cover the lemmata which give the existence of common ancestors (the existence of the common parent was obvious above because of the nestedness).
The following lemma gives the existence of the common dyadic parent in the separated case.
\begin{lem}
Let $I_1, I_2 \in \calD^n$ and $I_3 \in \calD^n_{\good}$ be such that $\ell(I_3) \le \ell(I_1) = 2\ell(I_2)$ and
$\max( d(I_1, I_3), d(I_2, I_3)) > \ell(I_3)^{\gamma_n} \ell(I_2)^{1-\gamma_n}$. Then there exists a cube $K \in \calD^n$ so that $I _1\cup I_2 \cup I_3 \subset K$ and
$$
\max(d(I_3,I_1), d(I_3,I_2)) \gtrsim \ell(I_3)^{\gamma_n}\ell(K)^{1-\gamma_n}.
$$
\end{lem}
The existence of the common dyadic parent in the diagonal case is given next.
\begin{lem}
Let $I_1, I_2 \in \calD^n$ and $I_3 \in \calD^n_{\good}$ be such that we have $\ell(I_3) \le \ell(I_1) = 2\ell(I_2)$,
$\max( d(I_1, I_3), d(I_2, I_3)) \le \ell(I_3)^{\gamma_n} \ell(I_2)^{1-\gamma_n}$ and
either $I_1 \cap I_3 = \emptyset$ or $I_1 = I_3$ or $I_2 \cap I_3 = \emptyset$.
Then there exists a cube $K \in \calD^n$ so that $I_1 \cup I_2 \cup I_3 \subset K$ and
$\ell(K) \le 2^r\ell(I_3)$.
\end{lem}
The results are by Hyt\"onen (see \cite{Hy} and \cite{Hy2})
in the linear case, and the above formulations are proved in Li--Martikainen--Ou--Vuorinen \cite{LMOV}.
Denote the smallest common dyadic ancestor by $I_1 \vee I_2 \vee I_3$ (which exists by the above lemmata).
Using these results we arrange the desired shift structure by reorganising the summations to the cases
where $I_1 \vee I_2 \vee I_3 = K$ for a given $K \in \calD^n$ (and similarly for the summations over $\calD^m$).

As the Nested/Nested is the only place, where a full paraproduct appears, we
only get shifts or partial paraproducts from the remaining terms. The collapse of the various partial paraproducts require the same kind of arguments as above.

Having taken care of the structural considerations, things boil down to proving the desired bounds required by the shifts and partial paraproducts.
Now, we explain the role of the assumption spt$\,f_i \subset B(0, c_02^{N})$. This ensures that e.g. all the appearing cubes $I_i$ satisfy
$d(I_i, I_{i'}) \le 2c_02^N$. So if the smallest cube $I_3$ satisfies $\ell(I_3) = 2^N$, then
$$
 \max( d(I_1, I_3), d(I_2, I_3)) \le 2c_02^N < \ell(I_3)^{\gamma_n}\ell(I_2)^{1-\gamma_n}
$$
if $c_0$ is chosen appropriately. This means that $\ell(I_3) < 2^N$ in the separated sum, which is key for obtaining
the desired estimates using H\"older bounds of the kernels. On the other hand, in the diagonal sum it does not matter
whether the Haar functions are cancellative or not (we demonstrate this below). And, in the various nested cases we sum as above.

Therefore, in the majority of the cases we use the kernel estimates like we did above, and combine these with the quantitative bounds for $I_1 \vee I_2 \vee I_3$ given by the above two lemmata. However, we did not explain the role of the diagonal BMO assumptions nor the weak boundedness assumptions yet. The need arises in the cases
Diagonal/Nested (and Nested/Diagonal ) and Diagonal/Diagonal respectively.  To conclude, we give an idea how the diagonal BMO assumptions
are used in the Diagonal/Nested case.

\subsubsection*{Diagonal/Nested and the diagonal BMO assumptions}
If we extract the partial paraproduct here as previously, we are e.g. left with the task of showing the BMO bound
$$
|\langle T(h_{I_1} \otimes 1, 1_{I_2} \otimes 1), h_{I_3} \otimes a_J\rangle| \lesssim |J|,
$$ 
whenever $I_1, I_2$ and $I_3$ are like in the diagonal summation, $J \subset \R^m$ is a cube, and $a_J$ is a function such that spt$\,a_J \subset J$, $|a_J| \le 1$ and $\int a_J = 0$.
If $I_1 \cap I_3 = \emptyset$ or $I_2 \cap I_3 = \emptyset$, this is not difficult. So we only consider the remaining case where $I_1 = I_3$ and $I_2 \in \textup{ch}(I_1)$.
Next, we use the familiar splitting $1 = 1_{3J} + 1_{(3J)^c}$, and also split
$
h_{I_1} = |I_1|^{-1/2} \sum_{I' \in \textup{ch}(I_1)} \epsilon_{I'} 1_{I'},
$
where $\epsilon_{I'} = \pm 1$. So we are left with terms like
$$
|I_1|^{-1} \sum_{I', I'' \in \textup{ch}(I_1)} |\langle T(1_{I'} \otimes 1_{(3J)^c}, 1_{I_2} \otimes 1), 1_{I''} \otimes a_J \rangle|
$$
and
$$
|I_1|^{-1} \sum_{I', I'' \in \textup{ch}(I_1)} |\langle T(1_{I'} \otimes 1_{3J}, 1_{I_2} \otimes 1_{3J}), 1_{I''} \otimes a_J\rangle|.
$$
In the first sum if $I' \ne I_2$ or $I'' \ne I_2$, we use the full kernel representation. If $I' = I'' = I_2$ we use the partial
kernel representation with $K_{1_{I_2}, 1_{I_2}, 1_{I_2}}$. In the second sum if $I' \ne I_2$ or $I'' \ne I_2$ we use
the partial kernel representation with $K_{1_{3J}, 1_{3J}, 1_{a_J}}$. Finally, in the second sum the case $I' = I'' = I_2$ requires us to bound
$
|I_1|^{-1} |\langle T(1_{I_2} \otimes 1_{3J}, 1_{I_2} \otimes 1_{3J}), 1_{I_2} \otimes a_J\rangle|.
$
But this is dominated by $|J|$ using one of the diagonal BMO assumptions. This ends our treatment of $\Sigma^1$, and thus the proof of
the a priori bounded case.

\subsection*{$T$ is not a priori bounded}
It is possible to first prove a representation theorem in a certain finite set up, where no a priori boundedness is needed.
Reductions of this type in the linear one-parameter situation appear in
\cite{GH} and \cite{Hy3}. We omit the technical details in our setting as they are similar.
A corollary of such a special representation is the boundedness of $T$, say from $L^3 \times L^3 \to L^{3/2}$. After
this, we can run the above argument. 

\end{proof}
\section{Boundedness results for DMOs}\label{sec:ModelBou}
\subsection{Weighted bounds for shifts and partial paraproducts}\label{sec:modelBW}
It is well-known that
\begin{equation}\label{eq:basicinfty}
\|f\|_{L^2(v)}^2 \lesssim_{[v]_{A_{\infty}(\R^n)}} \int_{\R^n} \sum_{I \in \calD^n} |\Delta_I f|^2  v 
\end{equation}
for $f \colon \R^n \to \C$ and $v \in A_{\infty}(\R^n)$. The fact that this holds for $A_{\infty}$ weights is important for us (due to the nature
of bilinear weights). This can be proved using some good lambda estimates as in \cite{CWW}.
See also e.g. Theorem 3.4 in \cite{DIPTV} for an explicit proof of this in a more general setting. 
The following $A_{\infty}$ extrapolation result is \cite[Theorem 2.1]{CUMP}.
\begin{lem}\label{thm:extra}
Let $(f,g)$ be a pair of positive functions defined on $\mathbb R^m$. Suppose that there exists some $0<p_0<\infty$ such that for every $w\in A_\infty(\mathbb R^m)$ we have
\[
\int_{\mathbb R^m} f^{p_0} w \le C([w]_{\infty}, p_0) \int_{\mathbb R^m} g^{p_0} w.
\]
Then for all $0<p<\infty$ and $w\in A_\infty(\mathbb R^m)$ we have
\[
\int_{\mathbb R^m} f^{p} w \le C([w]_{\infty}, p) \int_{\mathbb R^m} g^{p} w.
\]
\end{lem}
Therefore, by extrapolation \eqref{eq:basicinfty} improves to
$$
\int_{\R^n} \Big( \sum_j |f_j|^2 \Big)^{p/2}v \lesssim_{[v]_{A_{\infty}(\R^n)}} \int_{\R^n}\Big( \sum_j \sum_{I \in \calD^n} |\Delta_I f_j|^2 \Big)^{p/2} v, \qquad p \in (0,\infty).
$$
A direct corollary is that
\begin{equation}\label{eq:LowerSF1}
\|f\|_{L^p(v)}^p \lesssim \iint_{\R^{n+m}} \Big(\sum_{I \in \calD^n} |\Delta_I^1 f|^2\Big)^{p/2} v
\end{equation}
and
\begin{equation}\label{eq:ainftyprod}
\|f\|_{L^p(v)}^p 
\lesssim \iint_{\R^{n+m}} \Big(\sum_{\substack{I \in \calD^n \\ J \in \calD^m}} |\Delta_{I \times J} f|^2\Big)^{p/2} v
\end{equation}
for $f \colon \R^{n+m} \to \C$, $p \in (0, \infty)$ and weights $v$ satisfying
$\esssup_{x_2 \in \R^m} [v(\cdot, x_2)]_{A^{\infty}(\R^n)} < \infty$
and $\esssup_{x_1 \in \R^m} [v(x_1, \cdot)]_{A^{\infty}(\R^m)} < \infty$ (and the implicit constants depend on these norms).
We will apply
this with $v = v_3 = w_1^{r/p} w_2^{r/q} \in A_{2r}(\R^{n} \times \R^m)$, where $p,q \in (1, \infty)$, $1/p+1/q=1/r$ and $w_1 \in A_p(\R^n \times \R^m)$, $w_2 \in A_q(\R^n \times \R^m)$.

\subsubsection*{Shifts}
We prove the weighted estimates for shifts.
\begin{prop}\label{prop:weightedShifts}
Let $1 < p, q < \infty$ and $1/2 < r < \infty$ satisfy $1/p+1/q = 1/r$, $w_1 \in A_p(\R^n \times \R^m)$ and $w_2 \in A_q(\R^n \times \R^m)$
be bi-parameter weights, and set $v_3 := w_1^{r/p} w_2^{r/q}$.
Suppose that $(S_{\omega})_{\omega}$ is a collection of bilinear bi-parameter shifts of the same type, where
$S_{\omega}$ is defined in the grid $\calD_{\omega}$.
Then we have
$$
\|\E_{\omega} S_{\omega}(f_1, f_2)\|_{L^r(v_3)} \le C([w_1]_{A_p(\R^n \times \R^m)}, [w_2]_{A_q(\R^n \times \R^m)}) \|f_1\|_{L^p(w_1)} \|f_2\|_{L^q(w_2)}.
$$
\end{prop}
\begin{rem}
Here and below the non-averaged cases, where there is only one operator in a fixed grid, are true and easier to prove.
\end{rem}
\begin{proof}[Proof of Proposition \ref{prop:weightedShifts}]
We allow the implicit constants to depend on $[w_1]_{A_p(\R^\times \R^m)}$ 
and $[w_2]_{A_p(\R^\times \R^m)}$. We first prove the statement in a fixed grid directly. Let $S = S_{\calD}$ for some fixed $\calD=\calD^n \times \calD^m$.

The easiest case is when we have $h_{I_3}^0 \otimes h_{J_3}^0$ (i.e. all the non-cancellative functions are outside)
 -- then we don't need the $A_{\infty}$ square function
estimates at all, and we can work direcly. So we omit writing more about this case.

Consider the case
$
S(f_1, f_2) = \sum_{K \times V \in \calD}A_{K, V} (f_1,f_2), 
$
where $A_{K , V}(f_1,f_2)$ has the form
\begin{align*}
 \sum_{\substack{I_1, I_2, I_3 \in \calD^n \\ I_i^{(k_i)}  = K}} 
&\sum_{\substack{J_1, J_2, J_3 \in \calD^m \\ J_i^{(v_i)} = V}} a_{K,V, (I_i), (J_j)} 
\langle f_1, h_{I_1}^0 \otimes h_{J_1}\rangle \langle f_2,  h_{I_2} \otimes h_{J_2}\rangle  h_{I_3} \otimes h_{J_3}^0.
\end{align*}
Using the lower square function estimate \eqref{eq:LowerSF1} we get that
\begin{align*}
\|S(f_1, f_2) & \|_{L^r(v_3)} 
 \lesssim \Big\| \Big( \sum_{K \in \calD^n} 
\Big| \sum_{V \in \calD^m} A_{K, V}(f_1,f_2) \Big|^2 \Big)^{1/2} \Big \|_{L^r(v_3)}.
\end{align*}
Using the cancellative Haar functions and the normalisation of the coefficients $a_{K,V,(I_i),(J_i)}$
we get
$
|A_{K,V} (f_1,f_2)| 
\le M_{\calD} (\Delta_{V,v_1}^{2} f_1)
 M_{\calD} (\Delta_{K \times V}^{k_2, v_2} f_2).
 $
 Combining these gives that
\begin{align*}
\|S(f_1, f_2)\|_{L^2(v_3)} 
&\lesssim  \Big\| \Big( \sum_{K \in \calD^n} 
\Big| \sum_{V \in \calD^m} M_{\calD} (\Delta_{V,v_1}^{2} f_1)
 M_{\calD} (\Delta_{K \times V}^{k_2, v_2} f_2) \Big|^2 \Big)^{1/2} \Big \|_{L^r(v_3)} \\
&\le \Big\| \Big( \sum_{V \in \calD^m} [M_{\calD} \Delta_{V,v_1}^{2} f_1]^2 \Big)^{1/2} \Big\|_{L^p(w_1)} 
\Big\| \Big( \sum_{K\times V \in \calD} [M_{\calD} \Delta_{K \times V}^{k_2, v_2} f_2]^2 \Big)^{1/2} \Big\|_{L^q(w_2)},
\end{align*}
which is clearly dominated by $\|f_1\|_{L^p(w_1)} \|f_2\|_{L^q(w_2)}$.

There is no essential difference with any of the other cases. If we have $h_{I_3} \otimes h_{J_3}$ just use \eqref{eq:ainftyprod} in the beginning instead.
We are done proving the non-averaged case. For the averaged case of the proposition, notice that it follows from the non-averaged case if e.g. $p=q=4$ and $r=2$ (the point being that $r > 1$).
Bilinear extrapolation \cite{DU, GM} then gives the averaged version for all exponents.
\end{proof}
\subsubsection*{Partial paraproducts}
We need the following lemma.
\begin{lem}\label{lem:fs}
Let $v\in A_\infty(\mathbb R^m)$  and $b\in \BMO(\mathbb R^m)$. Then for a bilinear paraproduct $\pi_b$ in $\R^m$ we have
\begin{equation}\label{eq:fsest}
\int_{\mathbb R^m}| \pi_{b} (g_1, g_2)| v \lesssim [v]_{A_\infty(\R^m)} \|b\|_{\BMO} \int_{\mathbb R^m} (M_{\mathcal D^m}g_1 M_{\mathcal D^m}g_2) v.
\end{equation}
\end{lem}
\begin{proof}
Using sparse domination in the form sense for $\pi_b$, as stated e.g. in \cite{LMOV}, we have
\begin{align*}
\int_{\mathbb R^m}| \pi_{b } (g_1, g_2)| v 
&=\Big\langle \pi_{b } (g_1, g_2), \text{sgn}(\pi_{b } (g_1, g_2))v\Big\rangle \\
&\lesssim \|b \|_{\BMO}\sum_{Q\in \mathcal S} \langle |g_1 |\rangle_{Q} \langle |g_2| \rangle_{Q}v(Q)\\
&\le \|b \|_{\BMO}\sum_{Q\in \mathcal S} \Big(\big\langle (M_{\mathcal D^m} g_1 )^{1/2} (M_{\mathcal D^m}g_2 )^{1/2}  \big\rangle_Q^v\Big)^2v(Q)\\
&\lesssim \|b \|_{\BMO}[v]_{A_\infty(\R^m)} \| (M_{\mathcal D^m} g_1 )^{1/2} (M_{\mathcal D^m}g_2 )^{1/2}\|_{L^2(v)}^2,
\end{align*}
where in the last step we have used the Carleson embedding theorem.
\end{proof}
Now we use the $A_{\infty}$ extrapolation, Lemma \ref{thm:extra}, in a more crucial way than above.
Notice that via Lemma \ref{thm:extra} we can improve Lemma \ref{lem:fs} as follows.
By extrapolating the estimate \eqref{eq:fsest} we get that if $v \in A_{\infty}(\R^m)$ and 
$b_k \in \BMO(\R^m)$ with $\| b_k \|_{\BMO(\R^m)} \le 1$, $k \in \Z$, then for all $p \in (0,\infty)$ we have
\begin{equation}\label{eq:fs1}
\int_{\R^m} \Big| \sum_{k \in \Z} \pi_{b_k} (g_{1,k},g_{2,k}) \Big|^p v
\le C([v]_{A_\infty},p) \int_{\R^m} \Big(\sum_{k \in \Z} M_{\calD^m} g_{1,k} M_{\calD^m}g_{2,k}  \Big)^p v.
\end{equation}
Similarly, one concludes from \eqref{eq:fs1} that if
$\|b_{j,k} \|_{\BMO(\R^m)} \le 1$ and $v \in A_{\infty}(\R^m)$, 
then
\begin{equation}\label{eq:fs2}
\begin{split}
\int_{\R^m} &\Big(\sum_{j \in \Z} \Big| \sum_{k \in \Z} \pi_{b_{j,k}} (g_{1,j,k},g_{2,j,k})\Big|^p\Big)^{q/p}v \\
&\le C([v]_{A_\infty},p,q)
\int_{\R^m} \Big(\sum_{j \in \Z} \Big| \sum_{k \in \Z} M_{\calD^m} g_{1,j,k} M_{\calD^m}g_{2,j,k}\Big|^p\Big)^{q/p}v
\end{split}
\end{equation}
for every $p,q \in (0, \infty)$.

We are ready to prove the weighted bound for the partial paraproducts.
\begin{prop}\label{prop:weightedParProd}
Let $1 < p, q < \infty$ and $1/2 < r < \infty$ satisfy $1/p+1/q = 1/r$, $w_1 \in A_p(\R^n \times \R^m)$ and $w_2 \in A_q(\R^n \times \R^m)$
be bi-parameter weights, and set $v_3 := w_1^{r/p} w_2^{r/q}$.
Suppose that $(P_{\omega})_{\omega}$ is a collection of partial paraproducts of the same type, where
$P_{\omega}$ is defined in the grid $\calD_{\omega}$.
Then
$$
\|\E_{\omega} P_{\omega}(f_1, f_2)\|_{L^r(v_3)} \le C([w_1]_{A_p(\R^n \times \R^m)}, [w_2]_{A_q(\R^n \times \R^m)}) \|f_1\|_{L^p(w_1)} \|f_2\|_{L^q(w_2)}.
$$

\end{prop}
\begin{proof}
Similarly as in Proposition \ref{prop:weightedShifts}, by bilinear extrapolation it is enough to prove the non-averaged version.
We first consider the case that $P$ has the form
\[
P(f_1, f_2):= \sum_{K\in \mathcal D^n} \sum_{\substack{I_1, I_2, I_3 \in \mathcal D^n\\ I_i^{(k_i)}=K}} h_{I_3}^0\otimes \pi_{b_{K, (I_i)}} (\langle f_1, h_{I_1} \rangle_1, \langle f_2, h_{I_2} \rangle_1),
\]
where $\pi_{b_K, (I_i)}$ is a bilinear paraproduct with
\[
\| b_{K, (I_i)} \|_{\BMO(\R^m)} \le \frac {|I_1|^{1/2} |I_2|^{1/2} |I_3|^{1/2}}{|K|^2}=:a_{K,(I_i)}.
\]

In this case, notice that we may apply \eqref{eq:fs1} to conclude that
\begin{equation}\label{eq:Applyfs1}
\begin{split}
\|P & (f_1, f_2)\|_{L^r(v_3)} \\
&\lesssim \Big\| \sum_{K \in \calD^n} \sum_{\substack{I_1,I_2,I_3 \in \calD^n \\ I_i^{(k_i)}=K}}
a_{K,(I_i)}h^0_{I_3} \otimes M_{\calD^m} \langle f_1,h_{I_1} \rangle_1 M_{\calD^m} \langle f_2,h_{I_2} \rangle_1 \Big\|_{L^r(v_3)} \\
& = \Big\| \sum_{K \in \calD^n} \sum_{\substack{I_1,I_2,I_3 \in \calD^n \\ I_i^{(k_i)}=K}}
a_{K,(I_i)} h^0_{I_3} \otimes \langle \varphi^1_{\calD^n}(f_1),h_{I_1} \rangle_1 \langle \varphi^1_{\calD^n} (f_2),h_{I_2} \rangle_1 \Big\|_{L^r(v_3)},
\end{split}
\end{equation}
where we recall that
$%\label{eq:AuxiliaryFunction}
\varphi^1_{\calD^n}(f) = \sum_{I \in \calD^n} h_I \otimes M_{\calD^m} \langle f_1, h_I \rangle_1.
$

For a fixed $x_2 \in \R^m$ the last function in \eqref{eq:Applyfs1} as a function of $x_1$ is a one-parameter bilinear shift
acting on $\varphi^1_{\calD^n}(f_1)(\cdot, x_2)$ and $\varphi^1_{\calD^n}(f_2)(\cdot,x_2)$.
Therefore, the weighted boundedness of one-parameter bilinear shifts implies that
$\|P  (f_1, f_2)\|_{L^r(v_3)}$ is dominated by
\begin{equation}\label{eq:AfterApplyfs1}
\begin{split}
\Big( \int_{\R^m} &\| \varphi^1_{\calD^n}(f_1)(\cdot,x_2) \|_{L^p(w_1(\cdot,x_2))}^r
 \| \varphi^1_{\calD^n}(f_2)(\cdot,x_2) \|_{L^q(w_2(\cdot,x_2))}^r \ud x_2\Big)^{1/r}\\
& \lesssim  \|\varphi^1_{\calD^n}(f_1) \|_{L^p(w_1)}\|\varphi^1_{\calD^n}(f_2) \|_{L^q(w_2)}
\lesssim \|f_1 \|_{L^p(w_1)}\|f_2 \|_{L^q(w_2)}.
\end{split}
\end{equation}

By symmetry it remains to consider the case when $P$ has the  form
\[
P(f_1, f_2):= \sum_{K\in \mathcal D^n} \sum_{\substack{I_1, I_2, I_3 \in \mathcal D^n\\ I_i^{(k_i)}=K}} h_{I_3}\otimes \pi_{b_{K, (I_i)}} (\langle f_1, h_{I_1}^0 \rangle_1, \langle f_2, h_{I_2} \rangle_1).
\]
Applying the lower square function estimate \eqref{eq:LowerSF1} and the estimate \eqref{eq:fs2}, 
we have similarly as in \eqref{eq:Applyfs1} that
\begin{equation}\label{eq:Applyfs2}
\begin{split}
\|P & (f_1, f_2)\|_{L^r(v_3)}\\
&\lesssim \Big\| \Big( \sum_{K \in \calD^n } \Big| \sum_{\substack{I_1,I_2,I_3 \in \calD^n \\ I_i^{(k_i)}=K}}
h_{I_3} \otimes \pi_{b_{K, (I_i)}} (\langle f_1, h_{I_1}^0 \rangle_1, \langle f_2, h_{I_2} \rangle_1) \Big|^2 \Big)^{1/2} \Big\|_{L^r(v_3)} \\
& \lesssim \Big\| \Big( \sum_{K \in \calD^n } \Big| \sum_{\substack{I_1,I_2,I_3 \in \calD^n \\ I_i^{(k_i)}=K}}
a_{K,(I_i)} h_{I_3} \otimes \langle M_{\calD}  f_1, h_{I_1}^0 \rangle_1  
\langle \varphi^1_{\calD^n}(f_2), h_{I_2} \rangle_1 \Big|^2 \Big)^{1/2}\Big\|_{L^r(v_3)} ,
\end{split}
\end{equation}
where we also estimated that 
$ M_{\calD^m} \langle f_1,h^0_{I_1} \rangle_1 \le \langle M_{\calD}  f_1, h_{I_1}^0 \rangle_1$. 
With a fixed $x_2 \in \R^m$, consider the function in the last line in \eqref{eq:Applyfs2} as a function of $x_1 \in \R^m$. 
It is the square function of a one-parameter bilinear shift. 
Therefore, similarly as in \eqref{eq:Applyfs1} and \eqref{eq:AfterApplyfs1}, we have that the last term 
in \eqref{eq:Applyfs2} is dominated by
$$
\| M_{\calD}  f_1\|_{L^p(w_1)} \| \varphi^1_{\calD^n}(f_2)\|_{L^q(w_2)} \\
\lesssim \|  f_1\|_{L^p(w_1)} \| f_2\|_{L^q(w_2)}.
$$
We are done.
\end{proof}

\subsection{Unweighted boundedness of full paraproducts}
\begin{prop}\label{prop:FullProd}
Let $1 < p, q \le \infty$ and $1/2 < r < \infty$ satisfy $1/p+1/q = 1/r$.
Suppose that $(\Pi_{\omega})_{\omega}$ is a collection of bilinear bi-parameter full paraproducts of the same type, where $\Pi_{\omega}$ is
defined in the grid $\calD_{\omega}$.
Then we have
$$
\|\E_{\omega} \Pi_{\omega}(f_1, f_2)\|_{L^r(\R^{n+m})} \lesssim \|f_1\|_{L^p(\R^{n+m})} \|f_2\|_{L^q(\R^{n+m})}.
$$
\end{prop}
\begin{proof}
The proof of the Banach range boundedness $p,q,r \in (1,\infty)$ is a standard argument using duality.
Suppose now that $1/p + 1/q = 1/r$ with $1 < p, q < \infty$ and $\frac 12< r<1$, and $E \subset \R^{n+m}$ with $0 < |E| < \infty$.
We prove that for some $E' \subset E$ with $|E'| \ge |E|/2$ we have
\begin{align*}
|\langle \E_\omega \Pi_\omega (f_1, f_2), f_3 \rangle| &\lesssim  \|f_1\|_{L^p(\R^{n+m})} \|f_2\|_{L^q(\R^{n+m})} |E|^{1/r'}, \qquad |f_3| \le 1_{E'}.
\end{align*}
We omit the details of this core part of the proof, as later on we have to perform the argument in the commutator setting, which is quite a bit more demanding, and we give the
full details there. However, we take this opportunity to explain the completely standard interpolation argument, which can always be used to end the proof after such a weak type estimate.
This general scheme is used e.g. in \cite{MPTT}.

At this point we know that
\begin{align}
\label{eq:BRange} &\|\E_\omega \Pi_\omega(f_1, f_2)\|_{L^r(\R^{n+m})} 
\lesssim \|f_1\|_{L^p(\R^{n+m})} \|f_2\|_{L^q(\R^{n+m})}, \quad \text{if } r > 1, \\
\label{eq:Weak}& \|\E_\omega \Pi_\omega(f_1, f_2)\|_{L^{r, \infty}(\R^{n+m})} 
\lesssim \|f_1\|_{L^p(\R^{n+m})} \|f_2\|_{L^q(\R^{n+m})}, \quad \text{if }r< 1.
\end{align}
Notice that for $r = 1$ we may easily get that
if $0 < |E_i| < \infty$, $i = 1,2,3$, $|f_1| \le 1_{E_1}$ and $|f_2| \le 1_{E_2}$ then there exists $E_3' \subset E_3$ so that $|E_3'| \ge |E_3|/2$, and so that
for all $|f_3| \le 1_{E_3'}$ we have
$$
|\langle \E_\omega \Pi_\omega(f_1, f_2), f_3\rangle| \lesssim |E_1|^{1/p}|E_2|^{1/p'}. 
$$
This follows by taking convex combinations of the estimates \eqref{eq:BRange}, \eqref{eq:Weak}.
Then use e.g. Theorem 3.8 (or rather its proof) in Thiele's book \cite{Th:Book} to update all of our estimates that are either
weak type (if $r < 1$) or restricted weak type (if $r=1$) into strong type bounds. The cases $p = \infty$ or $q =  \infty$ follow by duality.
\end{proof}
\begin{rem}\label{rem:tensorfullb}
If $b = b_1 \otimes b_2$, where $\|b_1\|_{\BMO(\R^n)}, \|b_2\|_{\BMO(\R^m)} \le 1$, we get the weighted estimate for full paraproducts. Denoting $b_K^1 = \langle b_1, h_K\rangle$, $b_V^2 = \langle b_2, h_V\rangle$ we e.g. have
\begin{align*}
&\Big\| \sum_{K, V} b_K^1 b_V^2 \langle f_1 \rangle_{K \times V} \langle f_2 \rangle_{K \times V} h_K \otimes h_V \Big\|_{L^r(v_3)} \\
&\lesssim \Big\| \Big(\sum_K |b_K^1|^2 \frac{1_K}{|K|} \otimes \Big| \sum_V b_V^2 \langle f_1 \rangle_{K \times V} \langle f_2 \rangle_{K \times V} h_V  \Big|^2 \Big)^{1/2} \Big\|_{L^r(v_3)} \\
&\lesssim \Big\| \Big( \sum_K |b_K^1|^2 \frac{1_K}{|K|} \otimes [ M \langle f_1 \rangle_{K,1} M \langle f_2 \rangle_{K,1}]^2\Big)^{1/2} \Big\|_{L^r(v_3)} \\
&\lesssim \Big( \E \Big\| \sum_K \epsilon_K b_K^1 h_K \otimes  \langle M^2 f_1 \rangle_{K,1}  \langle M^2 f_2 \rangle_{K,1} \Big\|_{L^r(v_3)}^r \Big)^{1/r}
\lesssim \| M^1 M^2 f_1 M^1 M^2 f_2\|_{L^r(v_3)},
\end{align*}
from which the claim follows. We used \eqref{eq:fs2} and Kahane--Khintchine. Other forms of full paraproducts are similar but even easier.
\end{rem}

\section{Weighted, quasi-Banach and mixed-norm estimates for CZOs}\label{sec:TBou}
Recall that a bilinear bi-parameter SIO $T$ is free of full paraproducts if the following holds:
$
\langle S(1, 1), h_I \otimes h_J \rangle = 0
$
for all
$
S \in \{T, T^{*1}, T^{*2}, T^{1*}_1, T^{2*}_1, T^{1*}_2, T^{2*}_2, T^{1*, 2*}_{1,2}, T^{1*, 2*}_{2,1}\}
$
and all cubes $I \subset \R^n$, $J \subset \R^m$.
We begin with the improved bounds that hold in this case.
\begin{thm}\label{cor:T1-1}
Let $T$ be a bilinear bi-parameter CZO that is free of full paraproducts.
Then we have
\begin{equation}\label{eq:weightedT}
\|T(f_1, f_2)\|_{L^r(v_3)} \lesssim_{[w_1]_{A_p}, [w_2]_{A_q}} \|f_1\|_{L^p(w_1)} \|f_2\|_{L^q(w_2)} 
\end{equation}
for all $1 < p, q < \infty$ and $1/2 < r < \infty$ satisfying $1/p+1/q = 1/r$, and for all bi-parameter weights $w_1 \in A_p(\R^n \times \R^m)$, $w_2 \in A_q(\R^n \times \R^m)$ with $v_3 := w_1^{r/p} w_2^{r/q}$. 

We also have
\begin{equation}\label{eq:weightedTmixed}
\|T(f_1, f_2)\|_{L^{r_1}(v_{3,1}; L^{r_2}(v_{3,2}))} \lesssim_{[w_{i, j}]} \|f_1\|_{L^{p_1}(w_{1,1}; L^{p_2}(w_{2,1}))}\|f_2\|_{L^{q_1}(w_{2,1}; L^{q_2}(w_{2,2}))}
\end{equation}
for all $1 < p_i, q_i < \infty$ and $1/2 < r_i < \infty$ with $1/p_i + 1/q_i = 1/r_i$, for all
$w_{1,1} \in A_{p_1}(\R^n)$, $w_{2,1} \in A_{q_1}(\R^n)$, $w_{1,2} \in A_{p_2}(\R^m)$, $w_{2,2} \in A_{q_2}(\R^m)$
with $v_{3,1}:= w_{1,1}^{r_1/p_1}w_{2,1}^{r_1/q_1}$ and $v_{3,2}:= w_{1,2}^{r_2/p_2}w_{2,2}^{r_2/q_2}$.

In the unweighted case we have
$$
\|T(f_1, f_2)\|_{L^{r_1}(\R^n; L^{r_2}(\R^m))} \lesssim \|f_1\|_{L^{p_1}(\R^n; L^{p_2}(\R^m))}\|f_2\|_{L^{q_1}(\R^n; L^{q_2}(\R^m))}
$$
for all $1 < p_i, q_i \le \infty$ and $1/2 < r_i < \infty $ with $1/p_i + 1/q_i = 1/r_i$, except that if $r_2 < 1$ we have
to assume $\infty \not \in \{p_1, q_1\}$.
\end{thm}
\begin{rem}
Let $1 < u_1, u_2 < \infty$ and $1/2 < u_3 < \infty$ satisfy $1/u_1+1/u_2 = 1/u_3$.
Recall that weighted bounds imply vector-valued bounds, so we have in the situation of the previous theorem that for instance
\begin{align*}
\Big\| \Big(\sum_j |T(f_1^j, f_2^j)|^{u_3} \Big)^{1/u_3} \Big\|_{L^r(v_3)} 
\lesssim
\Big\| \Big(\sum_j |f_1^j|^{u_1} \Big)^{1/u_1} \Big\|_{L^p(w_1)} 
\Big\| \Big(\sum_j |f_2^j|^{u_2} \Big)^{1/u_2} \Big\|_{L^q(w_2)},
\end{align*}
whenever $1 < p, q < \infty$ and $1/2 < r < \infty$ satisfy $1/p+1/q = 1/r$, and $w_1 \in A_p(\R^n \times \R^m)$, $w_2 \in A_q(\R^n \times \R^m)$ with $v_3 := w_1^{r/p} w_2^{r/q}$.
\end{rem}
\begin{proof}[Proof of Theorem \ref{cor:T1-1}]
The implicit constants are allowed to depend on $A_p$ characteristics below.
Write the pointwise identity
\begin{equation}\label{eq:pointwiseRep}
T(f_1, f_2) = C_T \mathop{\sum_{k = (k_1, k_2, k_3) \in \Z_+^3}}_{v = (v_1, v_2, v_3) \in \Z_+^3} \alpha_{k, v} 
\sum_{u} \mathbb{E}_{\omega} U^{v}_{k, u, \mathcal{D}{\omega}}(f_1, f_2),
\end{equation}
where the DMOs are shifts or partial paraproducts. The results for averages of model operators in
Propositions \ref{prop:weightedShifts} and \ref{prop:weightedParProd} imply \eqref{eq:weightedT} by using $\|\sum_i g_i\|_{L^r(v_3)}^r \le \sum_i \| g_i  \|_{L^r(v_3)}^r$, if $r < 1$, or by using the normal triangle inequality otherwise.

We now simply use extrapolation again to get mixed-norm estimates.
 Suppose $1/p_2 + 1/q_2 = 1/r_2$, where $1 < p_2, q_2 < \infty$, $1/2 < r_2 < \infty$.
Fix some weights $w_{1,2} \in A_{p_2}(\R^m)$, $w_{2,2} \in A_{q_2}(\R^m)$ and set $v_{3,2}:= w_{1,2}^{r_2/p_2}w_{2,2}^{r_2/q_2}$.
Define
$$
H_{r_2, v_{3,2}}(x_1) = H(x_1):= \Big(\int_{\R^m} |T(f_1,f_2)(x_1,x_2)|^{r_2}v_{3,2}(x_2) \ud x_2\Big)^{1/r_2},
$$
$$
F_{1, p_2, w_{1,2}}(x_1) = F_1(x_1):= \Big(\int_{\R^m} |f_1(x_1,x_2)|^{p_2}w_{1,2}(x_2) \ud x_2\Big)^{1/p_2},
$$
and 
$$
F_{2, q_2, w_{2,2}}(x_1) = F_{2}(x_1):= \Big(\int_{\R^m} |f_2(x_1,x_2)|^{q_2}w_{2,2}(x_2) \ud x_2\Big)^{1/q_2}.
$$
Suppose $w_{1,1} \in A_{p_2}(\R^n)$, $w_{2,1} \in A_{q_2}(\R^n)$ and define $v_{3,1}:= w_{1,1}^{r_2/p_2}w_{2,1}^{r_2/q_2}$.
Let $w_1 := w_{1,1} \otimes w_{1,2} \in A_{p_2}(\R^n \times \R^m)$, $w_2 := w_{2,1} \otimes w_{2,2} \in A_{q_2}(\R^n\times \R^m)$
and define the related weight $v_{3} := w_1^{r_2/p_2} w_2^{r_2/q_2}$.

Applying \eqref{eq:weightedT} with $p=p_2$, $q=q_2$, $r=r_2$ and with the weights $w_1$, $w_2$ and $v_3$ we have that
$
\| H \|_{L^{r_2}(v_{3,1})} 
\lesssim \| F_1 \|_{L^{p_2}(w_{1,1})}\| F_2 \|_{L^{q_2}(w_{2,1})}.
$
Because this holds for all weights $w_{1,1} \in A_{p_2}(\R^n)$ and $w_{2,1} \in A_{q_2}(\R^n)$ with $v_{3,1} = w_1^{r_2/p_2}w_2^{r_2/q_2}$,
then bilinear extrapolation (see e.g. \cite{DU}) implies that
\begin{equation}\label{eq:OuterExtrapolated}
\| H \|_{L^{r_1}(v_{3,1})} 
\lesssim \| F_1 \|_{L^{p_1}(w_{1,1})}\| F_2 \|_{L^{q_1}(w_{2,1})}
\end{equation}
holds for all exponents $1 < p_1, q_1 < \infty$, $1/2 < r_1 < \infty$, and all weights 
$w_{1,1} \in A_{p_1}(\R^n)$, $w_{2,1} \in A_{q_1}(\R^n)$ and $v_{3,1}:= w_1^{r_1/p_1}w_2^{r_1/q_1}$. But \eqref{eq:OuterExtrapolated}
is exactly what is claimed in \eqref{eq:weightedTmixed}.

Let us now discuss the unweighted case of these mixed-norm estimates, when we can also
allow cases when some of the exponents $p_1, q_1, p_2, q_2$ are equal to $\infty$.
The starting point is the estimate
\begin{equation}\label{eq:single-exponent}
\|T(f_1, f_2)\|_{L^r(\R^{n+m})} \lesssim \|f_1\|_{L^p(\R^{n+m})} \|f_2\|_{L^q(\R^{n+m})}
\end{equation}
valid for all $1 < p, q \le \infty$, $1/2 < r < \infty$ with $1/p + 1/q = 1/r$. Notice that in this inequality
we can simply allow the case $\infty \in \{p,q\}$ by duality.

Fix now $1 < p_2, q_2 \le \infty$ with $1/p_2 + 1/q_2 = 1/r_2$ and $1 \le r_2 < \infty$, and define the Banach spaces
$E_1 = L^{p_2}(\R^m)$, $E_2 = L^{q_2}(\R^m)$ and $E_3 = L^{r_2}(\R^m)$.
Next, for all nice functions $f_i \colon \R^n \to F(\R^m) := \{g\colon g \colon \R^m \to \C\}$
such that spt$\,f_i \cap \textup{spt}\, f_j = \emptyset$ for some $i,j$,
we have using Section \ref{sec:reform} that
$$
\langle T(f_1, f_2), f_3\rangle = \int_{\R^n} \int_{\R^n} \int_{\R^n} \bla U_1(x_1, y_1, z_1)(f_1(y_1), f_2(z_1)), f_3(x_1)\bra\ud x_1 \ud y_1 \ud z_1,
$$
where we have the natural size estimate and the H\"older estimate
$$
\|U_1(x_1, y_1, z_1) - U_1(x_1', y_1, z_1)\|_{E_1 \times E_2 \to E_3}
\lesssim \frac{|x_1-x_1'|^{\alpha}}{(|x_1-y_1| + |x_1-z_1|)^{2n+\alpha}}
$$
whenever $|x_1-x_1'| \le \max(|x_1-y_1|, |x_1-z_1|)/2$ together with the symmetric one.
Since we know from \eqref{eq:single-exponent} that
$$
\|T(f_1, f_2)\|_{L^{r_2}(\R^n; E_3)} \lesssim \|f_1\|_{L^{p_2}(\R^n; E_1)} \|f_2\|_{L^{q_2}(\R^n; E_2)},
$$
we get as in the scalar-valued case (and with the same classical Calder\'on--Zygmund decomposition proof) that
$
T \colon L^{1}(\R^n; E_1) \times  L^{1}(\R^n; E_2) \to  L^{1/2,\infty}(\R^n; E_3).
$
Such a weak type estimate implies a pointwise sparse domination -- see for instance \cite{Li} for a proof in the scalar case.
However, essentially the same proof works in the Banach-valued case.
From the sparse domination it follows directly that
$$
\|T(f_1, f_2)\|_{L^{r_1}(\R^n; E_3)} \lesssim \|f_1\|_{L^{p_1}(\R^n; E_1)} \|f_2\|_{L^{q_1}(\R^n; E_2)} 
$$
for all  $1 < p_1, q_1 \le \infty$ with $1/p_1 + 1/q_1 = 1/r_1$, $1/2 < r_1 < \infty$.
Combining everything we have shown the theorem.
\end{proof}
We also have the following unweighted estimates for all bilinear bi-parameter CZOs.
\begin{thm}\label{cor:T1-2}
Let $T$ be a bilinear bi-parameter CZO.
Then we have
$$
\|T(f_1, f_2)\|_{L^r(\R^{n+m})} \lesssim \|f_1\|_{L^p(\R^{n+m})} \|f_2\|_{L^q(\R^{n+m})} 
$$
for all $1 < p, q \le \infty$ and $1/2 < r < \infty$ satisfying $1/p+1/q = 1/r$. 
Moreover, we have the mixed-norm estimates
$$
\|T(f_1,f_2)\|_{L^{r_1}(\R^n; L^{r_2}(\R^m))} \lesssim \|f_1\|_{L^{p_1}(\R^n; L^{p_2}(\R^m))} \|f_2\|_{L^{q_1}(\R^n; L^{q_2}(\R^m))}
$$
for all $1 < p_i, q_i \le \infty$ with $1/p_i + 1/q_i = 1/r_i$, $i = 1,2$,
$1/2 < r_1 < \infty$, $1 \le r_2 < \infty$. 
\end{thm}
\begin{proof}
Write again the pointwise identity \eqref{eq:pointwiseRep}, which this time contains all model operators.
Then argue as before but also use Proposition \ref{prop:FullProd}. This gives the desired non-mixed estimate in the full range, which
can be lifted to the claimed mixed-norm estimate by exactly the same operator-valued argument as above.
\end{proof}

\section{Decompositions and upper bounds for first order commutators}\label{sec:com1}
We begin by considering first order commutators.
\begin{thm}\label{thm:mainCOM1}
Let $1 < p,q \le \infty$ and $1/2 < r < \infty$ satisfy $1/p + 1/q = 1/r$, and let $b \in \bmo(\R^{n+m})$.
Suppose $T$ is a bilinear bi-parameter CZO.
Then we have
$$
\|[b, T]_1(f_1, f_2)\|_{L^r(\R^{n+m})} \lesssim \|b\|_{\bmo(\R^{n+m})} \|f_1\|_{L^p(\R^{n+m})} \|f_2\|_{L^q(\R^{n+m})},
$$
and similarly for $[b, T]_2$.
\end{thm}
\begin{proof}
Using Theorem \ref{thm:rep} write the pointwise identity
$$
[b,T]_1(f_1, f_2) = C_T \mathop{\sum_{k = (k_1, k_2, k_3) \in \Z_+^3}}_{v = (v_1, v_2, v_3) \in \Z_+^3} \alpha_{k, v} 
\sum_{u} \mathbb{E}_{\omega} [b,U^{v}_{k, u, \mathcal{D}_{\omega}}]_1(f_1, f_2).
$$
\emph{Averages} of commutators of model operators map in the full range with a bound polynomial in complexity -- this is proved carefully below (see
Propositions \ref{prop:com1ofShiftsQuasiBanach}, \ref{prop:com1PP} and \ref{prop:com1FP}).
Using this we get the claim.
\end{proof}
The corresponding results for iterated commutators are proved in Section \ref{sec:iterated}.

\subsection{Martingale difference expansions of products}\label{sec:marprod}
The idea is that a product $bf$ paired with Haar functions is expanded in the bi-parameter fashion only if both of the Haar functions are cancellative. In a mixed
situation we expand only in $\R^n$ or $\R^m$, and in the remaining fully non-cancellative situation we do not expand at all.
When pairing with a non-cancellative Haar function we add and subtract a suitable average of $b$.

Let $\calD^n$ and $\calD^m$ be some fixed dyadic grids in $\R^n$ and $\R^m$, respectively, and write $\calD= \calD^n \times \calD^m$.
In what follows we sum over $I \in \calD^n$ and $J \in \calD^m$.
\subsubsection*{Paraproduct operators}
Let us first define certain paraproduct operators:
\begin{align*}
A_1(b,f) &= \sum_{I, J} \Delta_{I \times J} b \Delta_{I \times J} f, \,\,
A_2(b,f) = \sum_{I, J} \Delta_{I \times J} b E_I^1\Delta_J^2 f, \\
A_3(b,f) &= \sum_{I, J} \Delta_{I \times J} b \Delta_I^1 E_J^2  f, \,\,
A_4(b,f) = \sum_{I, J} \Delta_{I \times J} b \bla f \bra_{I \times J},
\end{align*}
and
\begin{align*}
A_5(b,f) &= \sum_{I, J} E_I^1 \Delta_J^2 b \Delta_{I \times J} f, \,\,
A_6(b,f) = \sum_{I, J}  E_I^1 \Delta_J^2 b  \Delta_I^1 E_J^2  f, \\
A_7(b,f) &= \sum_{I, J} \Delta_I^1 E_J^2  b \Delta_{I \times J} f, \,\,
A_8(b,f) = \sum_{I, J}  \Delta_I^1 E_J^2 b E_I^1 \Delta_J^2 f.
\end{align*}
The operators are grouped into two collections, since they are handled differently (using product BMO or little BMO estimates, respectively).
Also recall that $\bmo \subset \BMO_{\textup{prod}}$.
When the underlying grid needs to be written, we write $A_{i, \calD}(b, f)$. When desired, these operators can be written with Haar functions using the standard formulas.
Recall that then one has to be slightly careful when a term like $h_I h_I$ or $h_J h_J$ appears (as they do e.g. when expanding $A_1$ using Haar functions).
This really can be of the form $h_I^{\epsilon_1} h_I^{\epsilon_2}$ for possibly different $\epsilon_1, \epsilon_2$. However, the only property we will
use is that $|h_I h_I| = 1_I / |I|$, i.e. we always treat such products as non-cancellative objects.

We also define
$$
a^1_1(b,f) = \sum_I \Delta_I^1 b \Delta_I^1 f \qquad \textup{and} \qquad
a^1_2(b,f) = \sum_I \Delta_I^1 b E_I^1 f.
$$
The operators $a^2_1(b,f)$ and $a^2_2(b,f)$ are defined analogously. Again, we can also e.g. write
$a^1_{1, \calD^n}(b,f)$ to emphasise the underlying dyadic grid.

\begin{lem}\label{lem:basicAa}
Let $H_b$ be $A_i(b,\cdot)$, $i=1,\ldots, 8$, or $a_j^1(b,\cdot)$, $a_j^2(b,\cdot)$, $j=1,2$. Let also $b\in \bmo(\R^{n+m})$, $p \in (1,\infty)$ and $w \in A_p(\R^n \times \R^m)$.
Then we have
\begin{align*}
\|H_b f\|_{L^p(w)} \lesssim C([w]_{A_p(\R^n \times \R^m)}) \|b\|_{\bmo(\R^{n+m})} \|f\|_{L^p(w)}.
\end{align*}
\end{lem}
\begin{proof}
The proofs follow by completely standard arguments using $H^1$-$\BMO$ type duality estimates. See e.g. \cite{HPW} (the operators
$a_j^1(b,\cdot)$, $a_j^2(b,\cdot)$ do not appear in \cite{HPW}, and the operators $A_i(b,\cdot)$ appear in a bit different Haar function forms).
\end{proof}

Let now $f \in L^p(\R^{n+m})$ for some $p \in (1,\infty)$, and $b \in \bmo(\R^{n+m})$.
We know that $b \in L^{p}_{\loc}(\R^{n+m})$ by the John--Nirenberg valid for little BMO.
For $I_0 \in \calD^n$ and $ J_0 \in \calD^m$ we will now introduce our expansions of $\langle bf, h_{I_0} \otimes h_{J_0}\rangle$, $\big \langle bf, h_{I_0} \otimes \frac{1_{J_0}}{|J_0|}\big\rangle$ and $\langle bf \rangle_{I_0 \times J_0}$.

We first consider $\langle bf, h_{I_0} \otimes h_{J_0} \rangle$. There holds
$$
1_{I_0 \times J_0} b
= \sum_{\substack{I_1\times J_1 \in \calD \\ I_1 \times J_1 \subset I_0 \times J_0}}\Delta_{I_1 \times J_1} b
+\sum_{\substack{J_1 \in \calD^m \\ J_1 \subset J_0}} E^1_{I_0} \Delta^2_{J_1} b
+ \sum_{\substack{I_1 \in \calD^n \\ I_1 \subset I_0}} \Delta^1_{I_1} E^2_{J_0} b
+ E_{I_0 \times J_0} b.
$$
Let us denote these terms by $I_j$, $j=1,2,3,4$, in the respective order.
We have the corresponding decomposition of $f$, whose terms we denote by $II_i$, $i=1,2,3,4$. Calculating carefully the pairings
$\langle I_j II_i, h_{I_0} \otimes h_{J_0} \rangle$ we see that
\begin{equation}\label{eq:biparEX}
\langle bf, h_{I_0} \otimes h_{J_0} \rangle = \sum_{i=1}^8 \langle A_i(b, f), h_{I_0} \otimes h_{J_0} \rangle + \langle b \rangle_{I_0 \times J_0} \langle f, h_{I_0} \otimes h_{J_0} \rangle.
\end{equation}

We now consider $\big \langle bf, h_{I_0} \otimes \frac{1_{J_0}}{|J_0|}\big\rangle$.
This time we write
$
1_{I_0} b  = \sum_{\substack{I_1 \in \calD^n \\ I_1 \subset I_0}}\Delta_{I_1}^1 b + E_{I_0}^1 b,
$
and similarly for $f$. Calculating $\langle bf, h_{I_0} \rangle_1$ we see that
\begin{equation}\label{eq:1EX}
\begin{split}
\Big \langle bf, h_{I_0} \otimes \frac{1_{J_0}}{|J_0|}\Big\rangle &= \sum_{i=1}^2 \Big\langle a_i^1(b,f), h_{I_0} \otimes \frac{1_{J_0}}{|J_0|} \Big\rangle \\
&+ \bla (\langle b \rangle_{I_0,1} - \langle b \rangle_{I_0 \times J_0}) \langle f, h_{I_0}\rangle_1\bra_{J_0}
+ \langle b \rangle_{I_0 \times J_0} \Big \langle f, h_{I_0} \otimes \frac{1_{J_0}}{|J_0|}\Big\rangle.
\end{split}
\end{equation}
When we have $\langle bf \rangle_{I_0 \times J_0}$ we do not expand at all:
\begin{equation}\label{eq:noEX}
\langle bf \rangle_{I_0 \times J_0} = \langle (b-\langle b \rangle_{I_0 \times J_0})f \rangle_{I_0 \times J_0}
+ \langle b \rangle_{I_0 \times J_0} \langle f\rangle_{I_0 \times J_0}.
\end{equation}

All of our commutators are simply decomposed using \eqref{eq:biparEX}, \eqref{eq:1EX} (and its symmetric form) and \eqref{eq:noEX} whenever
the relevant pairings/averages appear. This splits the commutator into several parts. Only the parts which  come from the last terms in \eqref{eq:biparEX}, \eqref{eq:1EX} and \eqref{eq:noEX} need to be combined with another part of the commutator of the same form. 
All the other parts are handled separately.

\subsection{Adapted maximal functions}\label{sec:AdapMaxFunc}
For $b \in \BMO(\R^n)$ and $f \colon \R^n \to \C$ define
$$
M_bf = \sup_I \frac{1_I}{|I|} \int_I |b-\langle b \rangle_I| |f|.
$$
In the situation $b \in \bmo(\R^{n+m})$ and $f \colon \R^{n+m} \to \C$ define $M_b$ analogously 
but take the supremum over all rectangles. The dyadic variants could also be defined, and denoted by
$M_{\calD^n, b}$ and $M_{\calD, b}$.
For a little BMO function $b \in \bmo(\R^{n+m})$ define 
$$
\varphi_{\calD^m, b}^2(f) = \sum_{J \in \calD^m} M_{\langle b \rangle_{J,2}} \langle f, h_J \rangle_2 \otimes h_J,
$$
and similarly define $\varphi_{\calD^n, b}^1(f)$. For our later usage it is important to not to use the dyadic
variant $M_{\calD^n, \langle b \rangle_{J,2}}$, as it would induce an unwanted dependence on $\calD^n$ (which has relevance
in some randomisation considerations). 
\begin{lem}\label{lem:bmaxbounds}
Suppose $\|b_i\|_{\BMO(\R^n)} \le 1$, $1 < u, p < \infty$ and $w \in A_p(\R^n)$. Then we have
\begin{equation}\label{eq:vMb}
\Big\| \Big( \sum_i [M_{b_i} f_i]^u \Big)^{1/u} \Big\|_{L^p(w)} \lesssim C([w]_{A_p(\R^n)}) \Big\| \Big( \sum_i |f_i|^u \Big)^{1/u} \Big\|_{L^p(w)}.
\end{equation}
The same bound holds with $\|b_i\|_{\bmo(\R^n \times \R^m)} \le 1$ and $w \in A_p(\R^n \times \R^m)$.
For a function $b$ with $\|b\|_{\bmo(\R^n \times \R^m)} \le 1$ we also have
$$
\|\varphi_{\calD^m, b}^2(f)\|_{L^p(w)} \le C([w]_{A_p(\R^n \times \R^m)})\|f\|_{L^p(w)}, \qquad 1 < p < \infty,\, w \in A_p(\R^n \times \R^m).
$$
\end{lem}
\begin{proof}
We begin by proving the bound $\|M_b f\|_{L^p(w)} \le C([w]_{A_p(\R^n \times \R^m)})\|f\|_{L^p(w)}$ -- the proof is the same in the one-parameter case.
Fix $w \in A_p(\R^n \times \R^m)$ and choose $s = s([w]_{A_p(\R^n \times \R^m)}) \in (1,p)$ so that $[w]_{A^{p/s}(\R^n \times \R^m)} \le C([w]_{A_p(\R^n \times \R^m)})$.
This can be done using the reverse H\"older inequality -- the well-known bi-parameter version is stated and proved e.g. in Proposition 2.2. of \cite{HPW}.
Using H\"older's inequality and the John--Nirenberg for little bmo we get that
$
M_b f \le C(s) M_s f = C([w]_{A_p(\R^n \times \R^m)})M_s f.
$
Now, using that $M \colon L^q(v) \to L^q(v)$ for all $q \in (1,\infty)$ and $v \in A_q(\R^n \times \R^m)$ we have
$$
\|M_b f\|_{L^p(w)} \le C([w]_{A_p(\R^n \times \R^m)})\| M |f|^s \|_{L^{p/s}(w)}^{1/s} \le C([w]_{A_p(\R^n \times \R^m)}) \|f\|_{L^p(w)}.
$$
The bi-parameter version of \eqref{eq:vMb} (and \eqref{eq:vMb} itself) now follow by extrapolation.

Next, using the estimate \eqref{eq:vMb} and the fact that $\| \langle b \rangle_{J,2} \|_{\BMO(\R^n)} \lesssim 1$
we get
\begin{equation*}
\begin{split}
\|\varphi^2_{\calD^m, b}(f)\|_{L^p(w)}
&\le C([w]_{A_p(\R^n \times \R^m)}) \Big\| \Big(\sum_{J} \big(M_{\langle b \rangle_{J,2}} \langle f, h_{J}\rangle_2 \big)^2 \otimes \frac{1_{J}}{|J|} \Big)^{1/2} 
\Big\|_{L^p(w)} \\
& \le C([w]_{A_p(\R^n \times \R^m)}) \Big\| \Big(\sum_{J} \big|\langle f, h_{J}\rangle_2 \big|^2 \otimes \frac{1_{J}}{|J|} \Big)^{1/2} 
\Big\|_{L^p(w)} \\
&\le C([w]_{A_p(\R^n \times \R^m)}) \| f \|_{L^p(w)}.
\end{split}
\end{equation*}

\end{proof}
We collect a few trivial, but useful, lemmas below.
\begin{lem}\label{lem:maximalbound}
For $I \in \calD^n$ and $J \in \calD^m$ we have
$$
\big|\bla (\langle b \rangle_{J, 2} -\langle b \rangle_{I \times J} )\langle f, h_J\rangle_2 \bra_I\big| 
\le \Big\langle \varphi^2_{\calD^m, b}(f), \frac{1_I}{|I|} \otimes h_J\Big\rangle
$$
and
$$
\big|\bla (b-\langle b \rangle_{I \times J})f \bra_{I \times J}\big| \lesssim \langle  M_b f \rangle_{I \times J}.
$$
\end{lem}
\begin{proof}
There holds
\begin{equation*}
\big|\bla (\langle b \rangle_{J, 2} -\langle b \rangle_{I \times J} )\langle f, h_J\rangle_2 \bra_I\big| \le \langle  M_{\langle b \rangle_{J,2}} \langle f, h_J\rangle_2 \rangle_{I}
= \Big\langle \varphi^2_{\calD^m, b}(f), \frac{1_I}{|I|} \otimes h_J\Big\rangle,
\end{equation*}
where the last inequality follows from orthogonality. The second claimed inequality is even more immediate.
\end{proof}

The validity of the following lemma is obvious.
\begin{lem}\label{lem:bmobound}
Suppose $I^{(i)} = Q^{(q)} = K$ and $J^{(j)} = R^{(r)} = V$. If $\|b\|_{\bmo(\R^{n+m})} = 1$ then we have
$
|\langle b \rangle_{Q \times R} - \langle b \rangle_{I \times J}| \lesssim \max(i, j, q, r).
$
\end{lem}

\subsection{Banach range boundedness of commutators of DMOs}\label{sec:BanachforModels}
For the Banach range theory of commutators we only need the fact that all the DMOs from Section \ref{sec:defbilinbiparmodel}
are of the following general type. Fix dyadic grids $\calD^n$ and $\calD^m$.
Let $U = U^v_k$, $0 \le k_i \in \Z$ and $0 \le v_i \in \Z$, $i=1,2,3$, be a bilinear bi-parameter operator such that 
\begin{equation*}
\begin{split}
\langle U(f_1,f_2),f_3 \rangle
= \sum_{\substack{K \in \calD^n \\ V \in \calD^m}} 
\sum_{\substack{I_1, I_2, I_3 \in \calD^n \\ I_1^{(k_1)} = I_2^{(k_2)} = I_3^{(k_3)} = K}} 
&\sum_{\substack{J_1, J_2, J_3 \in \calD^m \\ J_1^{(v_1)} = J_2^{(v_2)} = J_3^{(v_3)} = V}} a_{K, V, (I_i), (J_j)} \\
&\times \langle f_1, \wt h_{I_1} \otimes  \wt h_{J_1}\rangle \langle f_2, \wt h_{I_2} \otimes \wt h_{J_2}\rangle  
\langle f_3, \wt  h_{I_3} \otimes \wt h_{J_3} \rangle,
\end{split}
\end{equation*}
where $a_{K, V, (I_i), (J_j)}$ are constants and for all $i=1,2,3$ we have $ \wt h_{I_i}= h_{I_i}$ for all $I_i \in \calD^n$ or $\wt h_{I_i}= h_{I_i}^0$ for all $I_i\in \calD^n$,
and similarly with the functions $\wt h_{J_j}$. We assume that for all $p,q,r \in (1,\infty)$ with $1/p + 1/q = 1/r$ we have
\begin{equation}\label{eq:bb}
\begin{split}
\sum_{\substack{K \in \calD^n \\ V \in \calD^m}} 
\sum_{\substack{I_1, I_2, I_3 \in \calD^n \\ I_1^{(k_1)} = I_2^{(k_2)} = I_3^{(k_3)} = K}} 
&\sum_{\substack{J_1, J_2, J_3 \in \calD^m \\ J_1^{(v_1)} = J_2^{(v_2)} = J_3^{(v_3)} = V}} \big| a_{K, V, (I_i), (J_j)} \\
&\times \langle f_1, \wt h_{I_1} \otimes  \wt h_{J_1}\rangle \langle f_2, \wt h_{I_2} \otimes \wt h_{J_2}\rangle  
\langle f_3, \wt  h_{I_3} \otimes \wt h_{J_3} \rangle \big | \\
& \lesssim \| f_1 \|_{L^p(\R^{n+m})} \| f_2 \|_{L^q(\R^{n+m})}\| f_3 \|_{L^{r'}(\R^{n+m})}. 
\end{split}
\end{equation}
We do not assume anything else about the constants $a_{K, V, (I_i), (J_j)}$.
In particular, $U$ can be a bilinear bi-parameter shift, a partial paraproduct or a full paraproduct.

The proof of the following Banach range boundedness result is now surprisingly easy.
\begin{prop}\label{prop:com1ofmodelBanach}
Let $p,q,r \in (1,\infty)$, $1/p + 1/q = 1/r$, $0 \le k_i \in \Z$ and $0 \le v_i \in \Z$, $i=1,2,3$. 
Let $U = U^v_k$ be a general bilinear bi-parameter DMO satisfying
\eqref{eq:bb}. In particular, $U$ can be a bilinear bi-parameter shift, a partial paraproduct or a full paraproduct.
Then for $b$ such that $\|b\|_{\bmo(\R^{n+m})} = 1$ we have
$$
\|[b, U]_1(f_1, f_2)\|_{L^r(\R^{n+m})} \lesssim (1+\max(k_i, v_i)) \|f_1\|_{L^p(\R^{n+m})} \|f_2\|_{L^q(\R^{n+m})}.
$$

\end{prop}

\begin{proof}
 We separately treat the different possible combinations of cancellative and non-cancellative Haar functions.
 The proof depends only on what Haar functions  we have paired with $f_1$ and $f_3$ in $\langle U(f_1, f_2), f_3\rangle$.
 All the model operators fall into one of the following cases:

\begin{enumerate}
\item  We have $\langle f_1,  h_{I_1} \otimes  h_{J_1}\rangle
\langle f_3,  h_{I_3} \otimes  h_{J_3}\rangle$.

\item   We have $\langle f_1, h_{I_1}^0 \otimes h_{J_1}\rangle \langle f_3,  h_{I_3} \otimes  h_{J_3}\rangle$, or one of the three other symmetric
cases. 

\item  We have $\langle f_1,  h_{I_1}^0 \otimes  h_{J_1}\rangle
\langle f_3,  h_{I_3} \otimes  h_{J_3}^0\rangle$, or the symmetric case.

\item We have $\langle f_1,  h_{I_1}^0 \otimes  h_{J_1}^0\rangle
\langle f_3,  h_{I_3} \otimes  h_{J_3}\rangle$, or the symmetric case.

\item We have $\langle f_1,  h_{I_1}^0 \otimes  h_{J_1}\rangle
\langle f_3,  h_{I_3}^0 \otimes  h_{J_3}\rangle$, or the symmetric case.

\item We have  $\langle f_1,  h_{I_1}^0 \otimes  h_{J_1}^0\rangle
\langle f_3,  h_{I_3}^0 \otimes  h_{J_3}\rangle$, or one of the other three symmetric cases.

\item We have  $\langle f_1,  h_{I_1}^0 \otimes  h_{J_1}^0\rangle
\langle f_3,  h_{I_3}^0 \otimes  h_{J_3}^0\rangle$.

\end{enumerate}

\textbf{Case 1.}  
Expand using \eqref{eq:biparEX}.
Using Lemma \ref{lem:bmobound}, the boundedness property \eqref{eq:bb} and the boundedness of
the operators $A_i(b, \cdot)$, $i = 1,\ldots, 8$, we have that
$$
|\langle [b,U]_1(f_1, f_2), f_3\rangle| \lesssim \max(k_i, v_i) \|f_1\|_{L^p(\R^{n+m})} \|f_2\|_{L^q(\R^{n+m})} \|f_3\|_{L^{r'}(\R^{n+m})}.
$$

\textbf{Case 2.}
This time we expand using \eqref{eq:biparEX} and \eqref{eq:1EX}. Then we use Lemma \ref{lem:maximalbound}, Lemma \ref{lem:bmobound}, the boundedness property \eqref{eq:bb},
the boundedness of the operators $A_i(b, \cdot)$, $i = 1,\ldots, 8$, the boundedness of the operators $a_i^1(b,\cdot)$, $a_i^2(b,\cdot)$, $i = 1,2$, and Lemma
\ref{lem:bmaxbounds}.
This gives us the desired bound. For example, one calculates like
\begin{align*}
&\sum_{\substack{K \in \calD^n \\ V \in \calD^m}} \sum_{\substack{I_1, I_2, I_3 \in \calD^n \\ I_1^{(k_1)} = I_2^{(k_2)} = I_3^{(k_3)} = K}} 
\sum_{\substack{J_1, J_2, J_3 \in \calD^m \\ J_1^{(v_1)} = J_2^{(v_2)} = J_3^{(v_3)} = V}} |a_{K, V, (I_i), (J_j)}| \\
&\times \big|\bla (\langle b \rangle_{J_1,2}-\langle b \rangle_{I_1 \times J_1}) \langle f_1,h_{J_1} \rangle_2,h^0_{I_1} \bra \big|
|\langle f_2,  \wt h_{I_2} \otimes \wt h_{J_2}\rangle|
|\langle f_3, h_{I_3} \otimes h_{J_3} \rangle| \\
&\lesssim
\sum_{\substack{K \in \calD^n \\ V \in \calD^m}} \sum_{\substack{I_1, I_2, I_3 \in \calD^n \\ I_1^{(k_1)} = I_2^{(k_2)} = I_3^{(k_3)} = K}} 
\sum_{\substack{J_1, J_2, J_3 \in \calD^m \\ J_1^{(v_1)} = J_2^{(v_2)} = J_3^{(v_3)} = V}} |a_{K, V, (I_i), (J_j)}| \\
&\times \langle \varphi^2_{\calD^m, b}(f_1), h_{I_1}^0 \otimes h_{J_1}\rangle
|\langle f_2,  \wt h_{I_2} \otimes \wt h_{J_2}\rangle|
|\langle f_3, h_{I_3} \otimes h_{J_3} \rangle| \\
&\lesssim \| \varphi^2_{\calD^m, b}(f_1) \|_{L^p(\R^{n+m})} \|f_2\|_{L^q(\R^{n+m})} \|f_3\|_{L^{r'}(\R^{n+m})} \\
&\lesssim \|f_1\|_{L^p(\R^{n+m})} \|f_2\|_{L^q(\R^{n+m})} \|f_3\|_{L^{r'}(\R^{n+m})}.
\end{align*}

\textbf{Cases 3.-7.}
We keep operating as above, and so these cases are very similar.
\end{proof}

\subsection{Quasi--Banach estimates for commutators of averages of DMOs}\label{sec:quasiviarest}
\subsubsection*{Duality lemma}
We present a technical modification of the $H^1$-$\BMO$ type duality estimate.

For a sequence of scalars $\{a_{R}\}_{R \in \calD_{\omega}}$ denote
$$
\|\{a_{R}\}_{R \in \calD_{\omega}}\|_{\BMO_{\textup{prod}}(\calD_{\omega})} := 
\sup_{\Omega} \Big( \frac{1}{|\Omega|} \mathop{\sum_{R \in \calD_{\omega}}}_{R \subset \Omega} |a_{R}|^2  \Big)^{1/2},
$$
where the supremum is taken over those sets $\Omega \subset \R^{n+m}$ such that $|\Omega| < \infty$ and such that for every $x \in \Omega$ there exist
$K \times V \in \calD_{\omega}$ so that $x \in K \times V \subset \Omega$.
\begin{lem}\label{lem:H1-BMO-Modified}
Let $\omega=(\omega_1,\omega_2)$ be a random parameter and $F \subset \R^{n+m}$. 
Suppose $\calC \subset \calD_0$ is a collection of rectangles such that
$|R \cap F | \ge \frac{99}{100} |R|$ for all $R \in \calC$.
Let  $\{a_{R + \omega }\}_{R \in \calC}$ and $\{b_{R}\}_{R \in \calC}$ be two collections of scalars.
Then
$$
\sum_{R \in \calC} | a_{R + \omega} b_{R} |
\lesssim \| \{ a_{R+ \omega} \}_{R \in \calC} \|_{\BMO_{\textup{prod}}(\calD_\omega)}
\iint_F \Big( \sum_{R \in \calC} |b_{R}|^2 \frac{1_{R}}{|R|} \Big)^{1/2}.
$$
\end{lem}

\begin{proof}
Write
$
S = \big( \sum_{R \in \calC}|b_{R}|^2 \frac{1_{R}}{|R|} \big)^{1/2}.
$
We may suppose that $\| 1_F S \|_{L^1} < \infty$.
For $u \in \Z$ let $\Omega_u=\{ 1_F S> 2^{-u}\}$ and $\wt{\Omega}_u= \{M1_{\Omega_u} >c\}$,
where $c=c(n,m) \in (0,1)$ is small enough. 
Define the collections
$
\widehat{\calC}_u
=\Big\{ R \in \calC \colon |R \cap \Omega_u| \ge  |R|/100\Big\}
$
and write $\calC_u=\widehat{\calC}_u \setminus \widehat{\calC}_{u-1}$. 

Let $R \in \calC$ be such that $b_{R} \not=0$. Then for all
$x=(x_1,x_2) \in R \cap F$ there holds that
$
0 < |b_{R}||R|^{-1/2} \le 1_F(x) S(x).
$
Since $|R \cap F | \ge \frac{99}{100} |R|$, this implies that  $R \in \widehat{\calC}_u$ for all large  $u \in \Z$.
On the other hand, since $|\Omega_u| \to 0$, as $u \to -\infty$, we have $R \not \in \widehat{\calC}_u$ for all small $u \in \Z$.
Therefore, (because $\Omega_{u-1} \subset \Omega_u$ for all $u$) it holds that
$$
\sum_{R \in \calC} | a_{R+\omega} b_{R} |
= \sum_{u \in \Z} \sum_{R \in \calC_u} | a_{R+\omega} b_{R} |.
$$

For fixed $u$ estimate
$$
\sum_{R \in \calC_u} | a_{R+\omega} b_{R} |
\le \Big( \sum_{R \in \calC_u}  |a_{R+\omega}|^2\Big)^{1/2}
\Big(\sum_{R \in \calC_u} |b_{R}| ^2\Big)^{1/2}. 
$$
Suppose $R \in \calC_u$. Because $|R \cap \Omega_u| \ge \frac{1}{100}|R|$, there holds that 
$
R+\omega \subset 3R \subset \wt \Omega_u
$
if $c = c(n,m)$ is fixed small enough.
Therefore, since $|\wt \Omega_u| \lesssim |\Omega_u|$, we get that
$$
\Big( \sum_{R \in \calC_u}  |a_{R+\omega}|^2\Big)^{1/2}
\lesssim  \| \{ a_{R+ \omega} \}_{R \in \calC} \|_{\BMO_{\textup{prod}}(\calD_\omega)} | \Omega_u |^{1/2}.
$$

Let again $R \in \calC_u$. Every $R \in \calC$ satisfies by assumption that $|R \cap F| \ge \frac{99}{100} |R|$, and because $R \not \in \widehat \calC_{u-1}$,
there holds that $|R \cap \Omega_{u-1}^c| \ge \frac{99}{100} |R|$. Thus, we have $|R \cap F \cap \Omega_{u-1}^c| \ge \frac{98}{100} |R|$, and this gives
(noting also that $R \subset \wt \Omega_u$ for every $R \in \calC_u$) that
$$
\Big(\sum_{R \in \calC_u} |b_{R}|^2\Big)^{1/2}
\lesssim \Big(\iint \displaylimits _{\wt \Omega_u \setminus \Omega_{u-1}}1_F\sum_{R \in \calC_u} |b_{R}| ^2\frac{1_R}{|R|}\Big)^{1/2}
\le \Big(\iint \displaylimits _{\wt \Omega_u \setminus \Omega_{u-1}}1_FS^2\Big)^{1/2} 
\lesssim 2^{-u}|\Omega_u|^{1/2}.
$$
The claim follows by summing over $u$.
\end{proof}

\begin{rem}\label{lem:H1-BMO-Modified-1par}
Set
$$
\|\{a_{V}\}_{V \in \calD^m_{\omega_2}} \|_{\BMO(\calD^m_{\omega_2})} := 
\sup_{V_0 \in \calD^m_{\omega_2}}\Big( \frac{1}{|V_0|} \sum_{\substack{V \in \calD^m_{\omega_2} \\ V \subset V_0}} |a_V|^2\Big)^{1/2}.
$$

For weak type estimates of partial paraproducts we use the following special case of Lemma \ref{lem:H1-BMO-Modified}.
Let $F \subset \R^{n+m}$ and $K_0 \in \calD^n_0$. Suppose $\calC \subset \calD^m_0$ is a collection of cubes such that
$|(K_0 \times V) \cap F | \ge \frac{99}{100} |K_0 \times V|$ for every $V \in \calC$. Let $\omega_2$ be a random parameter,
and let $\{a_{V+\omega_2}\}_{V \in \calC}$ be a collection of scalars. Then, for all scalars  $\{b_V\}_{V \in \calC}$
we have
$$
\sum_{V \in \calC} | a_{V+\omega_2}b_V| 
\lesssim \| \{ a_{V+\omega_2} \}_{V \in \calC} \|_{\BMO(\calD^m_{\omega_2})} \iint_F \frac{1_{K_0}}{|K_0|} \otimes \Big( \sum_{V \in \calC}  |b_V|^2 \frac{1_V}{|V|} \Big)^{1/2}.
$$
This follows from Lemma \ref{lem:H1-BMO-Modified}. To see this, let $\wt \calC=\{K \times V \in \calD_0  \colon K=K_0,  V \in \calC\}$.
For $K \times V \in \calC$ define $\wt a_{K \times (V+\omega_2)} = a_{V+\omega_2}$. 
Then we have
$$
\| \{\wt a_{K \times (V+\omega_2)}\}_{K \times V \in \wt \calC} \|_{\BMO_{\textup{prod}}(\calD^n_0 \times \calD^m_{\omega_2})}
= \frac{1}{|K_0|^{1/2}} \| \{a_{V+\omega_2}\}_{V \in \calC} \|_{\BMO( \calD^m_{\omega_2})}.
$$
\end{rem}

We move on to proving the quasi--Banach estimates.

\subsubsection*{Shifts}
\begin{prop}\label{prop:com1ofShiftsQuasiBanach}
Let $\|b\|_{\bmo(\R^{n+m})}  = 1$, and let $1 < p, q \le \infty$ and $1/2 < r < \infty$ satisfy $1/p+1/q = 1/r$.
Suppose that $(S^{v}_{k, \omega})_{\omega}$ is a collection of bilinear bi-parameter shifts of the same type and of complexity $(k,v)$, where
$S^{v}_{k, \omega}$ is defined in the grid $\calD_{\omega}$.
Then we have
$$
 \|\mathbb{E}_{\omega}[b,S^{v}_{k, \omega}]_1(f_1, f_2)\|_{L^r(\R^{n+m})} 
 \lesssim (1+\max(k_i, v_i)) \|f_1\|_{L^p(\R^{n+m})} \|f_2\|_{L^q(\R^{n+m})}.
$$
\end{prop}
\begin{proof}
We omit this shift case, and only show the corresponding proof for partial and full paraproducts below. However, we demonstrate the iterated case for shifts later.
\end{proof}

\subsubsection*{Partial paraproducts}\label{sec:ComPartialP}
Below it is convenient to write the partial paraproducts in the following way. Let $\omega = (\omega_1, \omega_2) \in (\{0,1\}^n)^{\Z} \times (\{0,1\}^m)^{\Z}$ and
$k = (k_1, k_2, k_3)$, $k_1, k_2, k_3 \ge 0$. For each $K, I_1, I_2, I_3 \in \calD^n_{\omega_1}$ with $I_i^{(k_i)} = K$ and $V \in \calD^m_{\omega_2}$
we are given a constant $a_{K, V, (I_i)}^{\omega}$ such that for all $K, I_1, I_2, I_3 \in \calD^n_{\omega_1}$ with $I_i^{(k_i)} = K$ we have
$$
\|\{ a_{K, V, (I_i)}^{\omega} \}_{V \in \calD^m_{\omega_2}} \|_{\BMO(\calD^m_{\omega_2})}
\le \frac{|I_1|^{1/2} |I_2|^{1/2} |I_3|^{1/2}}{|K|^2}.
$$
We consider a partial paraproduct $P_{k, \omega}$ of complexity $k$ and of a particular form
\begin{equation}\label{eq:ParticularPartialP}
\begin{split}
\langle P_{k, \omega}(f_1,f_2), f_3\rangle =  \sum_{K \in \calD^n_{\omega_1}} \sum_{\substack{I_1, I_2, I_3 \in \calD^n_{\omega_1} \\ I_i^{(k_i)} = K}} \sum_{V \in \calD^m_{\omega_2}}
&a_{K, V, (I_i)}^{\omega} \bla  f_1, h_{I_1}^0 \otimes h_V  \bra  \\
&\times \Big\langle f_2, h_{I_2} \otimes \frac{1_V}{|V|} \Big\rangle \Big\langle f_3, h_{I_3} \otimes \frac{1_V}{|V|} \Big\rangle.
\end{split}
\end{equation}

Before we estimate the commutators of partial paraproducts 
we define two auxiliary operators and prove their boundedness. 
Given $\omega_2$ set for $g \colon \R^m \to \C$ and $f \colon \R^{n+m} \to \C$ that
$$
\wt{S}_{\omega_2}g
= \Big(\sum_{V \in \calD^m_0} |\langle g, h_{V+\omega_2} \rangle|^2 \frac{1_V}{|V|} \Big)^{1/2}, \quad
\wt{S}^2_{\omega_2}f= \Big(\sum_{V \in \calD^m_0} |\langle f, h_{V+\omega_2} \rangle_2|^2 \otimes  \frac{1_V}{|V|} \Big)^{1/2}.
$$ 
Recall that $\E_{\omega}=\E_{\omega_1} \E_{\omega_2}$. Let  $\Phi_1$ and $\Phi_2^l$, $0 \le l \in \Z$, be the operators
$$
\Phi_1(f)=\E_{\omega_2}M^1 \wt{S}^2_{\omega_2}( \varphi^2_{\omega_2,b} f), \quad 
\Phi_2^l(f)=\Big( \sum_{K \in \calD^n_0} \E_{\omega_1}(M^1\Delta^1_{K+\omega_1,l} \varphi^1_{\omega_1} f)^2 \Big)^{1/2},
$$
where  $\varphi^2_{\omega_2,b} f_1 := \varphi^2_{\calD^m_{\omega_2},b} f_1$ is the function from Section \ref{sec:AdapMaxFunc}
and $\varphi^1_{\omega_1}f:= \varphi^1_{\calD^n_{\omega_1}}f$ was introduced in Lemma \ref{lem:standardEst2}.

\begin{lem}\label{lem:AuxBdd}
We have for all $l \in \Z$, $l \ge 0$, that
$$
\|\Phi_1(f)\|_{L^s(\R^{n+m})}+\|\Phi_2^l(f)\|_{L^s(\R^{n+m})} \lesssim \|f\|_{L^s(\R^{n+m})}, \qquad s \in (1,\infty),
$$
where the bound is independent of $l$.
\end{lem}
\begin{proof}
We use weights and extrapolation.
This is useful with $\Phi_2^l$ in order to reduce to $L^2$ estimates where we can take the expectation out. With $\Phi_1$ we could do without weights
by estimating directly in $L^s$. Take $w \in A_2(\R^n \times \R^m)$. 

Notice that for all $V \in \calD^m_0$ we have
$
|\langle f, h_{V+\omega_2} \rangle_2| \otimes  \frac{1_V}{|V|^{1/2}} \lesssim M^2 ( \langle f, h_{V+\omega_2} \rangle_2 \otimes h_{V+\omega_2}).
$
This implies
$$
\| \wt{S}^2_{\omega_2}f \|_{L^2(w)} \lesssim \Big\| \Big( \sum_{V \in \calD^m_{\omega_2}} [M^2 ( \langle f, h_{V} \rangle_2 \otimes h_{V})]^2 \Big)^{1/2} \Big\|_{L^2(w)}
\le C([w]_{A_2}) \|f\|_{L^2(w)},
$$
where we used weighted maximal function and weighted square function estimates. Lemma \ref{lem:bmaxbounds} says that
$\|\varphi_{\omega_2, b}^2(f)\|_{L^2(w)} \le C([w]_{A_2}) \|f\|_{L^2(w)}$, and obviously $M^1$ satisfies the same bound. The $L^2(w)$
result for $\Phi_1$ follows, and we can extrapolate.

Next, we have
\begin{align*}
\|\Phi_2^l(f)\|_{L^2(w)}^2 = \E_{\omega_1} \Big\| \Big( \sum_{K \in \calD^n_{\omega_1}} (M^1\Delta^1_{K,l} \varphi^1_{\omega_1} f)^2 \Big)^{1/2} \Big\|_{L^2(w)}^2
\le C([w]_{A_2}) \|f\|_{L^2(w)}^2
\end{align*}
using weighted maximal function and weighted square function estimates and Lemma \ref{lem:standardEst2}. 
We can extrapolate to finish.
\end{proof}

We are ready to treat the first order commutator of partial paraproducts.

\begin{prop}\label{prop:com1PP}
Let $\|b\|_{\bmo(\R^{n+m})}  = 1$, and let $1 < p, q \le \infty$ and $1/2 < r < \infty$ satisfy $1/p+1/q = 1/r$.
Suppose that $(P_{k, \omega})_{\omega}$ is a collection of partial paraproducts of the same type and of fixed complexity $k = (k_1, k_2, k_3)$, where
$P_{k, \omega}$ is defined in the grid $\calD_{\omega}$.
Then we have
$$
 \|\mathbb{E}_{\omega}[b,P_{k,\omega}]_1(f_1, f_2)\|_{L^r(\R^{n+m})} 
 \lesssim (1+\max k_i) \|f_1\|_{L^p(\R^{n+m})} \|f_2\|_{L^q(\R^{n+m})}.
$$
\end{prop}

\begin{proof}
The case $r>1$ is clear by Proposition \ref{prop:com1ofmodelBanach}.
Our remaining task is to prove a  weak type estimate when $r<1$, which combined with the Banach range boundedness
implies Proposition \ref{prop:com1PP} via interpolation. This general scheme was already explained in the proof of Proposition \ref{prop:FullProd}.

All the different forms of partial paraproducts are treated in a similar way. Here we assume that the operators  are of the  form
\eqref{eq:ParticularPartialP}, and therefore the commutator is split using the identity \eqref{eq:1EX}.
This leads to several terms, of which the two coming from the last term in \eqref{eq:1EX} need to be combined -- this part of the commutator produces the dependence on the complexity via Lemma \ref{lem:bmobound}. All the other parts of the commutator
are estimated separately. Here we focus on the part of the commutator involving 
$\bla (\langle b \rangle_{V,2} - \langle b \rangle_{I_1 \times V}) \langle f_1, h_{V} \rangle_2 , h_{I_1}^0\bra$.
All the other parts are estimated relatively similarly -- see also the proof of Proposition \ref{prop:com1FP} (the full paraproduct case), where
a different type of term is handled.

Let $\Phi_1$ and $\Phi_2^{k_2}$ be the auxiliary operators from Lemma \ref{lem:AuxBdd}, and recall the operator
$\varphi^2_{\omega_2,b}$ involved in the definition of $\Phi_1$. By Lemma \ref{lem:maximalbound} we have that
\begin{equation}\label{eq:phi_bDomination}
\big | \bla (\langle b \rangle_{V,2} - \langle b \rangle_{I_1 \times V}) 
\langle f_1, h_{V} \rangle_2 , h_{I_1}^0\bra \big |
\le \langle \varphi^2_{\omega_2,b} f_1, h^0_{I_1} \otimes h_V \rangle.
\end{equation}
Since $\Phi_1$ and $\Phi^{k_2}_2$ are bounded, it is enough to show that for all $f_1 \in L^p(\R^{n+m})$, $f_2 \in L^q(\R^{n+m})$ with
$\| \Phi_1(f_1)\Phi_2^{k_2}(f_2) \|_{L^r(\R^{n+m})}=1$ and $E \subset \R^{n+m}$ with $0<|E| < \infty$
there exists $E' \subset E$ with $|E'| \ge \frac{99}{100}|E|$ so that 
\begin{equation}\label{eq:JustAsPartialP}
\begin{split}
\sum_{\substack{K \in \calD^n_0 \\ V \in \calD^m_0}} \Lambda_{K,V}(f_1,f_2,f_3)
\lesssim |E|^{1/r'}
\end{split}
\end{equation} 
holds for all $f_3$ such that $|f_3| \le 1_{E'}$. Here $\Lambda_{K,V}$ is defined to act on three functions by
\begin{equation*}
\begin{split}
\Lambda_{K,V}(f_1,f_2,f_3)=
\E_\omega \sum_{\substack{I_1, I_2, I_3 \in \calD^n_{\omega_1} \\ I_i^{(k_i)} = K+\omega_1}} 
\Big| & a_{K+\omega_1, V+\omega_2, (I_i)}^{\omega} 
\langle \varphi^2_{\omega_2,b} f_1, h^0_{I_1} \otimes h_{V+\omega_2} \rangle \\ 
&\times \Big\langle f_2, h_{I_2} \otimes \frac{1_{V+\omega_2}}{|V|} \Big\rangle 
\Big \langle f_3, h_{I_3}\otimes \frac{1_{V+\omega_2}}{|V|} \Big \rangle \Big|.
\end{split}
\end{equation*}

We turn to prove \eqref{eq:JustAsPartialP}. Define the sets
$
\Omega_u =  \{ \Phi_1(f_1)\Phi_2^{k_2}(f_2) > C2^{-u} | E | ^{-1/r}  \}
$, $u \ge 0$.
For a small enough $c=c(n,m) \in (0,1)$ define the enlargement by 
$
\wt{\Omega}_u= \{M1_{\Omega_u} > c \}.
$
Define $E'=E \setminus \wt{\Omega}_0$. 
By choosing the constant $C$ in the definition of the sets $\Omega_u$ to be large, we have $|E'| \ge \frac{99}{100} |E |$.
Let $\widehat{\calR}_u$ be the collection of rectangles
$
\widehat{\calR}_u= \big\{ R \in \calD_0 \colon
|R \cap \Omega_u| \ge  |R|/100 \big \},
$
and write $\calR_u = \widehat{\calR}_u \setminus \widehat{\calR}_{u-1}$ when $u \ge 1$.

Now, we fix an arbitrary function $f_3$  such that $|f_3| \le 1_{E'}$ and consider \eqref{eq:JustAsPartialP} with $f_3$.
Let $K \times V \in \calD_0$. 
First, we show that if $\Lambda_{K,V} (f_1,f_2,f_3) \not=0$ then $K \times V \in \widehat \calR_u$ for some $u$.
We have for all $\omega=(\omega_1,\omega_2)$ and every $(x_1,x_2) \in K \times V$ that
\begin{equation*}
\begin{split}
\sum_{\substack{I_1 \in \calD^n_{\omega_1} \\ I_1^{(k_1)} = K+\omega_1}}
\frac{|I_1|^{1/2}}{|K||V|^{1/2}}| \langle \varphi^2_{\omega_2,b} f_1, h^0_{I_1} \otimes h_{V+\omega_2} \rangle|
& \le \sum_{\substack{I_1 \in \calD^n_{\omega_1} \\ I_1^{(k_1)} = K+\omega_1}} \frac{|I_1|^{1/2}}{|K|} \langle \wt S^2_{\omega_2} (\varphi^2_{\omega_2,b} f_1), h_{I_1}^0 \rangle_1(x_2) \\
&\lesssim M^1 \wt S^2_{\omega_2} (\varphi^2_{\omega_2,b} f_1)(x_1,x_2)
\end{split}
\end{equation*}
and
\begin{equation*}
\begin{split}
\sum_{\substack{I_2 \in \calD^n_{\omega_1} \\ I_2^{(k_2)} = K+\omega_1}} 
&\frac{|I_2|^{1/2}}{|K|} \Big | \Big \langle f_2, h_{I_2} \otimes \frac{1_{V+\omega_2}}{|V|} \Big \rangle \Big |
\lesssim 
 \sum_{\substack{I_2 \in \calD^n_{\omega_1} \\ I_2^{(k_2)} = K+\omega_1}} \frac{|I_2|^{1/2}}{|K|}
 M\langle f_2, h_{I_2} \rangle_1(x_2) \\
&= \sum_{\substack{I_2 \in \calD^n_{\omega_1} \\ I_2^{(k_2)} = K+\omega_1}} \frac{|I_2|^{1/2}}{|K|}
 \langle  \varphi^1_{\omega_1} f_2, h_{I_2} \rangle_1 (x_2)
\lesssim M^1(\Delta^1_{K+\omega_1,k_2} \varphi^1_{\omega_1} f_2) (x_1,x_2).
\end{split}
\end{equation*}
Therefore, if $\Lambda_{K,V} (f_1,f_2,f_3) \not=0$, then for
every $x \in K \times V$ there holds that
\begin{equation*}
\begin{split}
0 < \E_\omega \sum_{\substack{I_1,I_2 \in \calD^n_{\omega_1} \\ I_i^{(k_i)} = K+\omega_1}}
 \frac{|I_1|^{1/2}|I_2|^{1/2}}{|K|^2|V|^{1/2}} 
\Big| \langle \varphi^2_{\omega_2,b} f_1, h^0_{I_1} & \otimes h_{V+\omega_2} \rangle
 \Big \langle f_2, h_{I_2} \otimes \frac{1_{V+\omega_2}}{|V|} \Big \rangle \Big | \\
& \lesssim \Phi_1(f_1)(x) \Phi_2^{k_2}(f_2)(x).
\end{split}
\end{equation*}
The inequality ``$<$'' holds since the integrand is positive for $\omega$ in a set of positive measure.
From this it follows that $K \times V \subset \Omega_u$ if $u$ is large enough, and so in particular $K \times V \in \widehat \calR_u$.

If $K \times V \in \widehat{\calR}_u$, then for all $\omega$ there holds that $(K \times V)+\omega \subset (3K) \times (3V) \subset \wt \Omega_u$.
The constant $c=c(n,m)$ in the definition of $\wt \Omega_u$ is chosen so that this inclusion holds.
If $K \times V \in \widehat \calR _0$, then $(K \times V)+\omega \subset \wt \Omega_0 \subset (E')^c$ for all $\omega$, which combined with the fact that
$|f_3| \le 1_{E'}$ implies that $\Lambda_{K,V}(f_1,f_2,f_3)=0$.

With the above observations we have that
$$
\sum_{\substack{K \in \calD^n_0 \\ V \in \calD^m_0}} \Lambda_{K,V}(f_1,f_2,f_3)
= \sum_{u=1}^\infty \sum_{K \times V \in \calR_u} 
\Lambda_{K,V}(f_1,f_2,f_3).
$$
We fix $u$ and estimate the corresponding term.

First, notice that if $K \times V \in \calR_u$, then 
$\Lambda_{K,V}(f_1,f_2,f_3)=\Lambda_{K,V}( f_1,f_2,1_{\wt \Omega_u}f_3)$.
For $K \in \calD^n_0$ define
$
\calC_{K,u}
=\{V \in \calD^m_0 \colon K \times V \in \calR_u\}
$,
which allows us to write $\sum_{K \times V \in \calR_u}=\sum_{K \in \calD^n_0} \sum_{V \in \calC_{K,u}}$.
If $K \in \calD^n_0$, then every $V \in \calC_{K,u}$ satisfies $|(K \times V) \cap \Omega_{u-1}^c | \ge \frac{99}{100} |K \times V|$.
Therefore, Remark \ref{lem:H1-BMO-Modified-1par} gives that
\begin{equation}\label{eq:ApplyH1BMOLemma}
\begin{split}
& \sum_{K \in \calD^n_0}  \sum_{V \in \calC_{K,u}}\Lambda_{K,V}( f_1,f_2,1_{\wt \Omega_u}f_3) 
 \lesssim \sum_{K \in \calD^n_0} \E_\omega \sum_{\substack{I_1, I_2, I_3 \in \calD^n_{\omega_1} \\ I_i^{(k_i)} = K+\omega_1}} 
 \bigg[\frac{ |I_1|^{1/2}|I_2|^{1/2}|I_3|^{1/2}}{|K|^2} \\
&\hspace{1cm}\times \iint \displaylimits_{\wt\Omega_u \setminus \Omega_{u-1}}
\frac{1_{K}}{|K|} \otimes \wt S_{\omega_2} \langle \varphi^2_{\omega_2,b} f_1,  h^0_{I_1}  \rangle_1 M \langle f_2, h_{I_2} \rangle_1 M \langle  1_{\wt \Omega_u}f_3, h_{I_3} \rangle_1 \bigg], 
\end{split}
\end{equation}
where  we used the estimate 
$
\big | \big \langle f_2, h_{I_2} \otimes \frac{1_{V+\omega_2}}{|V|} \big \rangle \big | 
\lesssim M \langle f_2, h_{I_2} \rangle_1(x_2)
$, $x_2 \in V$, and the same estimate with $1_{\wt \Omega_u}f_3$. 
We were able to insert the restriction $\wt \Omega_u$ to the integration area since $K \times V \subset \wt \Omega_u$ for every $K \times V \in \calR_u$.

Notice that 
$\wt S_{\omega_2} \langle \varphi^2_{\omega_2,b} f_1,  h^0_{I_1}  \rangle_1(x_2) 
\le \langle \wt S_{\omega_2}^2  (\varphi^2_{\omega_2,b} f_1),  h^0_{I_1}  \rangle_1(x_2)$ 
for all $x_2$, and 
recall that $M \langle f_2, h_{I_2} \rangle_1=\langle  \varphi^1_{\omega_1} f_2, h_{I_2} \rangle_1$. 
Thus, the inner sum over the cubes $I_i$ in the right hand side of \eqref{eq:ApplyH1BMOLemma} 
 is dominated by
\begin{equation*}
\begin{split}
 %&
 %\iint \displaylimits_{\wt\Omega_u \setminus \Omega_{u-1}} 
 %1_{K} \otimes 
 %\langle \wt S^2_{\omega_2}  (\varphi^2_{\omega_2,b} f_1)  \rangle_{K+\omega_1,1}
 %\langle | \Delta^1_{K+\omega_1,k_2} \varphi^1_{\omega_1} f_2 | \rangle_{K+\omega_1,1} \\
%& \hspace{4cm} \times \langle | \Delta^1_{K+\omega_1,k_3} \varphi^1_{\omega_1} (1_{\wt \Omega _u}f_3) | \rangle_{K+\omega_1,1} \\
%& \lesssim 
\iint \displaylimits_{\wt\Omega_u \setminus \Omega_{u-1}}
M^1 \wt S^2_{\omega_2 }(\varphi^2_{\omega_2,b} f_1)
M^1(\Delta^1_{K+\omega_1,k_2} \varphi^1_{\omega_1} f_2)
M^1(\Delta^1_{K+\omega_1,k_3} \varphi^1_{\omega_1} (1_{\wt \Omega_u}f_3)).
\end{split}
\end{equation*}
Taking expectation $\E_{\omega}=\E_{\omega_1} \E_{\omega_2}$, using H\"older's inequality with respect to $\omega_1$ 
and summing over $K \in \calD^n_0$ shows that 
\begin{equation*}
\begin{split}
\sum_{K \in \calD^n_0}  \sum_{V \in \calC_{K,u}}\Lambda_{K,V}( f_1,f_2,1_{\wt \Omega_u}f_3) 
\lesssim \iint \displaylimits_{\wt\Omega_u \setminus \Omega_{u-1}} \Phi_1(f_1) \Phi_2^{k_2}(f_2)
\Phi_2^{k_3}(1_{\wt \Omega_u} f_3).
\end{split}
\end{equation*}

By definition we have $\Phi_1(f_1)(x) \Phi_2^{k_2}(f_2)(x) \lesssim 2^{-u}|E|^{-1/r}$ for all $x \in \Omega_{u-1}^c$. Also, just by using the 
$L^2$-boundedness of $\Phi_2^{k_3}$ and the fact that $\| f_3\|_{L^\infty} \le 1$ there holds that
$$
\iint \displaylimits_{\wt\Omega_u \setminus \Omega_{u-1}}
\Phi_2^{k_3}(1_{\wt \Omega_u} f_3)
\lesssim |\wt\Omega_u| \lesssim |\Omega_u| \lesssim 2^{ur} |E|.
$$
These combined give that
$
\sum_{K \in \calD^n_0}  \sum_{V \in \calC_{K,u}}
\Lambda_{K,V}(f_1,f_2,f_3)  \lesssim 2^{-u(1-r)} |E|^{1/r'},
$
which can be summed over $u$ since $r < 1$. This finishes the proof of \eqref{eq:JustAsPartialP},
and therefore we have proved the desired weak type estimate for the part of the commutator involving the terms
$\bla (\langle b \rangle_{V,2} - \langle b \rangle_{I_1 \times V}) \langle f_1, h_{V} \rangle_2 , h_{I_1}^0\bra$.
The corresponding estimate for the other parts   gives the weak type estimate for 
$\mathbb{E}_{\omega}[b,P_{k,\omega}]_1$, when $P_{k,\omega}$ is of the form \eqref{eq:ParticularPartialP}.
This ends our treatment of Proposition \ref{prop:com1PP}.

\end{proof}

\subsubsection*{Full paraproducts}
\begin{prop}\label{prop:com1FP}
Let $\|b\|_{\bmo(\R^{n+m})}  = 1$, and let $1 < p, q \le \infty$ and $1/2 < r < \infty$ satisfy $1/p+1/q = 1/r$.
Suppose that $(\Pi_{\omega})_{\omega}$ is a collection of full paraproducts of the same type, where
$\Pi_{\omega}$ is defined in the grid $\calD_{\omega}$.
Then we have
$$
 \|\mathbb{E}_{\omega}[b,\Pi_{\omega}]_1(f_1, f_2)\|_{L^r(\R^{n+m})} 
 \lesssim \|f_1\|_{L^p(\R^{n+m})} \|f_2\|_{L^q(\R^{n+m})}.
$$
\end{prop}
\begin{proof}
As in the partial paraproduct case we are done after showing the weak type estimate for $r<1$.
The different forms of full paraproducts are handled separately with analogous arguments. Let us consider here the case where every $\Pi_\omega$ is of the form
\begin{equation}\label{eq:OneFullPara}
\langle \Pi_{\omega}(f_1, f_2), f_3 \rangle = \sum_{K \times V \in \calD_{\omega}} 
a_{K,V}^{\omega} \langle f_1 \rangle_{K \times V} 
\Big \langle f_2, \frac{1_K}{|K|} \otimes h_V \Big \rangle  
\Big \langle f_3, h_K \otimes \frac{1_V}{|V|} \Big \rangle,
\end{equation}
where
$
\|\{a_{K,V}^{\omega}\}_{K \times V \in \calD_{\omega}}\|_{\BMO_{\textup{prod}}(\calD_{\omega})} \le 1.
$

The commutators $[b,\Pi_\omega]$ are again split with the identities from Section \ref{sec:marprod}, this time
using \eqref{eq:1EX} and \eqref{eq:noEX}.
The resulting terms are handled separately with similar arguments. 
Here we consider the term
\begin{equation}\label{eq:OneFromComFullP}
\E_\omega \sum_{K \times V \in \calD_{\omega}} 
a_{K,V}^{\omega} \langle f_1 \rangle_{K \times V} 
\Big \langle f_2, \frac{1_K}{|K|} \otimes h_V \Big \rangle  
\Big \langle a^1_{i,\omega_1}(b, f_3), h_K \otimes \frac{1_V}{|V|} \Big \rangle
\end{equation}
for some $i \in \{1,2\}$.

This time, let $\Phi_1$ and $\Phi_2$ be the auxiliary operators
$$
\Phi_1(f)
=\E_{\omega_2}\Big( \sum_{V \in \calD^m_0} (M \langle f, h_{V+\omega_2} \rangle_2)^2 \otimes \frac{1_V}{|V|} \Big)^{1/2}
$$ 
and 
$$
\Phi_2(f)
=\E_{\omega_1}\Big( \sum_{K \in \calD^n_0} \frac{1_K}{|K|} \otimes (M \langle a^1_{i,\omega_1}(b,f), h_{K+\omega_1} \rangle_1)^2 \Big)^{1/2}.
$$ 
Similarly as in Lemma \ref{lem:AuxBdd}  
these are bounded in $L^s$ for every $s \in (1, \infty)$, which uses the fact that 
$a^1_{i,\omega_1}(b,\cdot)$ is bounded by Lemma \ref{lem:basicAa}.

Let now $f_1 \in L^p(\R^{n+m})$ and $f_2 \in L^q(\R^{n+m})$ be such that  $\| Mf_1 \Phi_1(f_2) \|_{L^r}=1$ and let $E \subset \R^{n+m}$ with $0 < |E| < \infty$.
We show that there exists a set $E' \subset E$ with $|E'| \ge \frac{99}{100}|E|$ so that 
\begin{equation}\label{eq:EstOneFromComFullP}
 \sum_{K \times V \in \calD_{0}} \Lambda_{K,V}(f_1,f_2,f_3)
 \lesssim |E|^{1/r'}
\end{equation}
holds for all $f_3$ such that $|f_3| \le 1_{E'}$,
where $\Lambda_{K,V}$ acts on a triple of functions by
\begin{equation*}
\begin{split}
\Lambda_{K,V}(f_1,f_2,f_3)
=\E_\omega
\Big |a_{K+\omega_1,V+\omega_2}^{\omega} \langle f_1 \rangle_{(K \times V)+\omega} 
&\Big \langle f_2, \frac{1_{K+\omega_1}}{|K|} \otimes h_{V+\omega_2} \Big \rangle  \\
&\times \Big \langle a^1_{i,\omega_1}(b, f_3), h_{K+\omega_1} \otimes \frac{1_{V+\omega_2}}{|V|} \Big \rangle \Big |.
\end{split}
\end{equation*}
Since $\| Mg_1 \Phi_1(g_2)\|_{L^r} \lesssim \| g_1 \|_{L^p} \| g_2 \|_{L^q}$ for all $g_1$ and $g_2$, this gives the estimate that we want for
the term \eqref{eq:OneFromComFullP}. Together with the corresponding estimates for all the other parts of 
$\E_\omega[b,\Pi_\omega]$, this proves the weak type estimate we wanted to show.

We turn to prove \eqref{eq:EstOneFromComFullP}. For $u \ge 0$
let $\Omega_u=\{Mf_1\Phi_1(f_2) > C2^{-u} |E|^{-1/r}\}$ and
$\wt \Omega_u =\{M1_{\Omega_u} > c\}$, where $c=c(n,m) \in (0,1)$ is a small constant. 
Set $E' = E \setminus \wt \Omega_0$.  By choosing the constant $C$ to be large enough,
we have that $|E'| \ge \frac{99}{100}|E|$.
Then, define $\widehat{\calR}_u= \{R \in \calD_0 \colon |R \cap \Omega_u| \ge \frac{1}{100}|R|\}$ for $u \ge 0 $ and
$\calR_u=\widehat{\calR}_u \setminus \widehat{\calR}_{u-1}$ for $u \ge 1$.

Suppose $K \times V \in \calD_0$ is such that $\Lambda_{K,V} (f_1,f_2,f_3) \not=0$.
Then, wee see that
$$
0 <  \E_\omega \frac{1}{|V|^{1/2}} 
\Big |\langle f_1 \rangle_{(K \times V)+\omega} 
\Big \langle f_2, \frac{1_{K+\omega_1}}{|K|} \otimes h_{V+\omega_2} \Big \rangle \Big |
\lesssim Mf_1(x)\Phi_1(f_2)(x)
$$
for all $x \in K \times V $. The first ``$<$'' holds since the integrand is positive for $\omega$ in a set of positive measure.
Thus, $K \times V \subset \Omega_u$ for large enough $u$, so $K \times V \in \widehat{\calR}_u$.

If $R \in \widehat{\calR}_u$, then $R +\omega \subset 3R \subset  \wt \Omega_u$ for all $\omega$, which is based the fact that $c=c(n,m)$ in the definition of $\wt \Omega_u$
is small enough. Notice that 
\begin{equation}\label{eq:Localisation}
\Big \langle a^1_{i,\omega_1}(b, f_3), h_{K+\omega_1} \otimes \frac{1_{V+\omega_2}}{|V|} \Big \rangle
=\Big \langle a^1_{i,\omega_1}(b, 1_{(K\times V)+\omega }f_3), h_{K+\omega_1} \otimes \frac{1_{V+\omega_2}}{|V|} \Big \rangle.
\end{equation}
Thus, if $K \times V \in \widehat{\calR}_0$, then $\Lambda_{K,V}(f_1,f_2,f_3)=0$ since $(K \times V)+\omega \subset (E')^c$ for all $\omega$, and $|f_3| \le 1_{E'}$.

Now, we have that
$$
\sum_{K \times V \in \calD_{0}} \Lambda_{K,V}(f_1,f_2,f_3)
= \sum_{u=1}^\infty \sum_{K \times V \in \calR_u} \Lambda_{K,V}(f_1,f_2,1_{ \wt \Omega_u}f_3),
$$
where it was legitimate to replace $f_3$ with $1_{ \wt \Omega_u}f_3$ because of \eqref{eq:Localisation}. We fix one $u$ and 
estimate the related term.

If $K \times V \in \calR_u$, then by definition $|(K \times V) \cap \Omega_{u-1}^c| \ge \frac{99}{100} |K \times V|$.
Therefore, using Lemma \ref{lem:H1-BMO-Modified} and then
the estimate
$
\big |\big \langle f_2, \frac{1_{K+\omega_1}}{|K|} \otimes h_{V+\omega_2 } \big\rangle \big |
\lesssim M \langle f_2, h_{V+\omega_2} \rangle_2 (x_1)
$, $x_1 \in K$,
and a corresponding estimate related to $f_3$, we have that
\begin{equation*}
\begin{split}
\sum_{K \times V \in \calR_u} \Lambda_{K,V}(f_1,f_2,1_{\wt \Omega_u} f_3) 
%\lesssim \E_\omega  & \iint   \displaylimits_{\wt \Omega_u \setminus \Omega_{u-1}} 
%Mf_1 \Big( \sum_{V \in \calD^m_0} (M \langle f_2 , h_{V+\omega_2}\rangle_2)^2 \otimes \frac{1_V}{|V|} \Big)^{\frac 12} \\
%&\times \Big( \sum_{K \in \calD^n_0} \frac{1_K}{|K|} \otimes (M \langle a^1_{i,\omega_1}(b,1_{\wt \Omega_u}f_3) , h_{K+\omega_1}\rangle_1)^2  \Big)^{\frac 12} \\
\lesssim  \iint   \displaylimits_{\wt \Omega_u \setminus \Omega_{u-1}}  Mf_1 \Phi_1(f_2) \Phi_2(1_{\wt \Omega_u} f_3).
\end{split}
\end{equation*}
The restriction to $\wt \Omega_u$ in the integration came from the fact that every $R \in \calR_u$ satisfies $R \subset \wt \Omega_u$.
Since $Mf_1(x) \Phi_1(f_2)(x) \lesssim 2^{-u}|E|^{-1/r}$ for $x \in \Omega_{u-1}^c$, the operator $\Phi_2$ is $L^2$ bounded and $\|f_3\|_{L^\infty} \le 1$,
there holds that 
$$
\iint   \displaylimits_{\wt \Omega_u \setminus \Omega_{u-1}}  Mf_1 \Phi_1(f_2) \Phi_2(1_{\wt \Omega_u} f_3)
\lesssim 2^{-u}|E|^{-1/r} |\wt \Omega_u|
\lesssim 2^{-u(1-r)}|E|^{1/r'}.
$$
This can be summed over $u$ since $r < 1$, which finishes the proof of \eqref{eq:EstOneFromComFullP}.
\end{proof}

\section{Upper bounds for iterated commutators}\label{sec:iterated}
We give the details only for shifts in this section (in the first order case we gave the details for partial and full paraproducts).
First, we record the following lemma -- the proof is straightforward with our by now familiar method. We omit the details but see
Lemma 5.1 of \cite{LMV:Bloom}.
Parts of this lemma also appear in \cite{HPW}.
\begin{lem}\label{lem:weightedoneparcommutator}
Let $\|b_1\|_{\bmo(\R^{n+m})} = \|b_2\|_{\bmo(\R^{n+m})} = 1$, $1 < p < \infty$ and $w \in A_p(\R^n \times \R^m)$.
Then for $i = 1, \ldots, 8$ and $j=1,2$ we have
$$
\| [b_2, A_i(b_1, f)] \|_{L^p(w)} + \| [b_2, a_j^1(b_1, f)] \|_{L^p(w)} \le C([w]_{A_p(\R^n \times \R^m)}) \|f\|_{L^p(w)}.
$$
\end{lem}
Next, we set up some additional notation and record  the boundedness of certain deterministic square functions.
Let $i,j \in \Z$,  $i, j \ge 0$. Suppose that we have a family of operators $U=\{U_ {\omega}\}_{\omega}$ such that for all $\omega$ there holds
$$
\| U_{\omega} f \|_{L^2(w)} \le C([w]_{A_2(\R^n \times \R^m)})\| f \|_{L^2(w)}, \quad f \in L^2(w),
$$
for all $w \in A_2(\R^n \times \R^m)$. 
Define
$$
S_{U}^{i,j} f
=\Big(\sum_{K \times V \in \calD_0} \E_{\omega} (M \Delta^{i,j}_{(K+\omega_1) \times (V+\omega_2)} U_{\omega}f)^2 \Big)^{1/2}.
$$
Similarly, given $U=\{U_ {\omega_1}\}_{\omega_1}$ or $U=\{U_ {\omega_2}\}_{\omega_2}$ we set
$$
S^1_{i,U}=\Big(\sum_{K \in \calD_0^n} \E_{\omega_1} (M \Delta^{1}_{K+\omega_1,i} U_{\omega_1} f)^2 \Big)^{1/2}, \,\,\,
S^2_{j,U}=\Big(\sum_{V \in \calD_0^m} \E_{\omega_2} (M \Delta^{2}_{V+\omega_2,j} U_{\omega_2} f)^2 \Big)^{1/2}.
$$
We write $S^{i,j}$, $S^1_{i}$ and $S^2_j$ if there is no $U$ present.
The proof of the next lemma follows from extrapolation similarly as in Lemma \ref{lem:AuxBdd}.
\begin{lem}\label{lem:DetSquareShift}
For all $p \in (1, \infty)$ and $w \in A_p(\R^n\times\R^m)$ there holds
$$
\| S^{i,j}_Uf\|_{L^p(w)}+\| S^{1}_{i,U}f\|_{L^p(w)}+\| S^{2}_{j,U}f\|_{L^p(w)}
\le C([w]_{A_p(\R^n\times\R^m)})  \| f\|_{L^p(w)}.
$$
\end{lem}
We are ready to handle the iterated commutators of shifts.

\begin{prop}\label{prop:com2ofmodelQuasiBanach}
Let $\|b_1\|_{\bmo(\R^{n+m})} = \|b_2\|_{\bmo(\R^{n+m})} = 1$, and let $1 < p, q \le \infty$ and $1/2 < r < \infty$ satisfy $1/p+1/q = 1/r$.
Suppose that $(S_{\omega})_{\omega} = (S^{v}_{k, \omega})_{\omega}$ is a collection bilinear bi-parameter shifts of the same type and of complexity $(k,v)$, each
defined in the grid $\calD_{\omega}$.
Then we have
$$
 \|\mathbb{E}_{\omega}[b_2,[b_1,S_{\omega}]_1]_2(f_1, f_2)\|_{L^r(\R^{n+m})} 
 \lesssim (1+\max(k_i, v_i))^2 \|f_1\|_{L^p(\R^{n+m})} \|f_2\|_{L^q(\R^{n+m})},
$$
and similarly for $\mathbb{E}_{\omega}[b_2,[b_1,S_{\omega}]_1]_1$.
\end{prop}
\begin{proof}
We focus on the harder quasi--Banach range proof, which contains the decomposition of the iterated commutator that is
used also in the Banach range case. Some comments about the Banach range case are made along the way.
We only consider the iterated commutator of the form $[b_2,[b_1,S_{\omega}]_1]_2$, the other being analogous.

The shifts we consider here are of the form
\begin{equation*}
\langle S_{\omega}(f_1,f_2),f_3 \rangle
=  \sum_{\substack{K \in \calD^n_0 \\ V \in \calD^m_0}}  A^{\omega}_{K \times V}(f_1, f_2, f_3),
\end{equation*}
where
\begin{equation*}
\begin{split}
A^{\omega}_{K \times V}(f_1, f_2, f_3) 
= \sum_{\substack{I_1, I_2, I_3 \in \calD^n_{\omega_1} \\ I_i^{(k_i)} = K+\omega_1}} &
\sum_{\substack{J_1, J_2, J_3 \in \calD^m_{\omega_2} \\ J_i^{(v_i)}  = V+\omega_2}} 
a^{\omega}_{K+\omega_1, V+\omega_2, (I_i), (J_j)}  \\
&\times  \langle f_1,  h_{I_1}^0 \otimes   h_{J_1}\rangle \langle f_2,  h_{I_2} \otimes  h_{J_2}\rangle  
\langle f_3, h_{I_3} \otimes  h_{J_3}^0 \rangle.
\end{split}
\end{equation*}
Shifts of other form are handled analogously.

The first step of the decomposition is to expand $\langle [b_1,S_{\omega}]_1(f_1,f_2),f_3 \rangle$ using our usual rules.
This leads to multiple terms, and their contributions to the second order commutator $[b_2,[b_1,S_{\omega}]_1]_2$ are considered separately.
One of the terms is
\begin{equation}\label{eq:term1}
\sum_{\substack{K \in \calD^n_0 \\ V \in \calD^m_0}}  A^{\omega}_{K \times V}(f_1, f_2, a_{i,\omega_1}^1(b_1, f_3)).
\end{equation}
We consider its contribution first, and briefly comment the others at the end. 

We need to prove the following key statement. Let $r < 1$, and
suppose $f_1 \in L^p(\R^{n+m})$ and $f_2 \in L^q(\R^{n+m})$ satisfy
\begin{align*}
\Big\| S^2_{v_1}f_1 S^{k_2,v_2}f_2 + \sum_{j=1}^8  S^2_{v_1}f_1 S^{k_2,v_2}_{A_{j,b_2}} f_2 \Big\|_{L^r} = 1,
\end{align*}
where $A_{j,b_2}$ denotes the family $\{A_{j,\omega}(b_2, \cdot)\}_{\omega}$, and
the square functions are defined as above.
Given $E \subset \R^{n+m}$ with $0 < |E| < \infty$
there exists a subset $E' \subset E$ with $|E'| \ge \frac{99}{100} |E|$  so that for all functions $f_3$
satisfying $|f_3| \le 1_{E'}$ there holds
\begin{equation}\label{eq:ResAveShiftExEst2}
\begin{split}
\Big|\E_{\omega}\sum_{\substack{K \in \calD^n_0 \\ V \in \calD^m_0}} 
[A^{\omega}_{K \times V}(f_1, f_2, a_{i,\omega_1}^1(b_1, b_2f_3)) 
-A^{\omega}_{K \times V}&(f_1, b_2f_2, a_{i,\omega_1}^1(b_1, f_3))] 
 \Big| \\
 &\lesssim (1+\max(k_i,v_i))|E|^{1/r'}.
\end{split}
\end{equation}
In view of Lemma \ref{lem:DetSquareShift} this is enough.

We define
$$
\Omega_u =\Big\{S^2_{v_1}f_1 S^{k_2,v_2}f_2 +  \sum_{j=1}^8 S^2_{v_1}f_1 S^{k_2,v_2}_{A_{j,b_2}} f_2
> C_0 2^{-u}|E|^{-1/r}\Big\}, \quad u \ge 0,
$$
and 
$
\wt \Omega_u = \{M1_{\Omega_u}> c_1\},
$
where $c_1>0$ is a small enough dimensional constant. Then we
choose $C_0=C_0(c_1)$ so large that the set $E':=E \setminus \wt \Omega_0$ satisfies $|E'|  \ge \frac{99}{100} |E|$.
Define the collections
$
\widehat \calR_u =\big\{ R \in \calD_0 \colon |R \cap \Omega_u | \ge |R|/2\big\},
$
and set $\calR_u = \widehat\calR_u \setminus \widehat \calR_{u-1}$ for $u \ge 1$. 
Fix now some function $f_3$ such that $|f_3| \le 1_{E'}$.
We abbreviate
$$
\Lambda_{K \times V}^{\omega}(f_1,f_2,f_3) = 
A^{\omega}_{K \times V}(f_1, f_2, a_{i,\omega_1}^1(b_1, b_2f_3)) 
- A^{\omega}_{K \times V}(f_1, b_2f_2, a_{i,\omega_1}^1(b_1, f_3)).
$$
Notice  the localisation property
$$
\Lambda_{K \times V}^{\omega}(f_1,f_2,f_3)
=\Lambda_{K \times V}^{\omega}(f_1,f_2,1_{(K+\omega_1) \times (V+\omega_2)} f_3).
$$
With some thought, similarly as in the proofs of Proposition \ref{prop:com1PP} and Proposition \ref{prop:com1FP}, 
we see that we may write 
\begin{align*}
\E_{\omega}\sum_{\substack{K \in \calD^n_0 \\ V \in \calD^m_0}} 
 \Lambda^{\omega}_{K \times V} (f_1,f_2,f_3 )
=\sum_{u=1}^\infty \sum_{K \times V\in \calR_u} 
\E_{\omega} \Lambda^{\omega}_{K \times V} (f_1,f_2, 1_{\wt \Omega_u} f_3 ).
\end{align*}

We now fix $u$, and our goal is to prove
$$
\sum_{K \times V\in \calR_u} 
\E_{\omega} |\Lambda^{\omega}_{K \times V} (f_1,f_2, 1_{\wt \Omega_u} f_3 )| \lesssim
(1+\max(k_i,v_i))2^{-u(1-r)}|E|^{1/r'}.
$$
By adding and subtracting
$
A^{\omega}_{K \times V}(f_1, f_2, b_2a_{i,\omega_1}^1(b_1, 1_{\wt \Omega_u}f_3)),
$
we see that
\begin{equation}\label{eq:Com2Split1}
\begin{split}
\Lambda^{\omega}_{K \times V} (f_1,f_2, 1_{\wt \Omega_u} f_3 ) 
&= -A^{\omega}_{K \times V}(f_1, f_2, [b_2,a_{i,\omega_1}^1(b_1, \cdot)](1_{\wt \Omega_u}f_3)) \\
&+ A^{\omega}_{K \times V}(f_1, f_2, b_2a_{i,\omega_1}^1(b_1, 1_{\wt \Omega_u}f_3)) \\
&- A^{\omega}_{K \times V}(f_1, b_2f_2, a_{i,\omega_1}^1(b_1, 1_{\wt \Omega_u}f_3)).
\end{split}
\end{equation}

We now consider the first term of \eqref{eq:Com2Split1}. The corresponding term in the Banach range proof can be handled immediately by just using Lemma \ref{lem:weightedoneparcommutator}. Here we need to work a bit harder. Denote the family $\{ [b_2, a^1_{i,\omega_1}(b_1, \cdot)]\}_{\omega_1}$ 
by $[b_2,a^1_{i,b_1}]$. We see that
\begin{equation*}
\begin{split}
\sum_{K\times V \in \calR_u}  & \E_{\omega} 
\big| A^{\omega}_{K \times V}(f_1, f_2, [b_2,a_{i,\omega_1}^1(b_1, \cdot)](1_{\wt \Omega_u}f_3)) \big| \\
& \lesssim \int_{\wt \Omega_u \setminus \Omega_{u-1}} S^{2}_{v_1}f_1 S^{k_2,v_2} f_2 S^1_{k_1,[b_2,a^1_{i,b_1}]}(1_{\wt \Omega_u}f_3).
\end{split}
\end{equation*}
From here the estimate can be concluded in the familiar way, using that 
$
S^{2}_{v_1}f_1 S^{k_2,v_2} f_2 \lesssim 2^{-u}|E|^{-1/r}
$
in the complement of $\Omega_{u-1}$ and that the square function $S^1_{k_1,[b_2,a^1_{i,b_1}]}$ is bounded. The boundedness
of this square function follows from Lemmas \ref{lem:weightedoneparcommutator} and \ref{lem:DetSquareShift}.

We now consider the last difference in \eqref{eq:Com2Split1}. In the Banach range this can simply be handled
by using the boundedness of $[b_2, S_{\omega}]_2$ and $a^1_{i, \omega_1}(b_1, \cdot)$.
Here we need to split this further by expanding the products $b_2a_{i,\omega_1}^1(b_1, 1_{\wt \Omega_u}f_3)$
and $b_2f_2$ using our usual expansions rules determined by the form $A^{\omega}_{K \times V}$.

This produces several terms, one of them being $A^{\omega}_{K \times V}(f_1, A_{j, \omega}(b_2,f_2), a_{i,\omega_1}^1(b_1, 1_{\wt \Omega_u}f_3))$, $j=1,\dots,8$.
Its contribution can be handled by estimating as
\begin{equation*}
\begin{split}
 \sum_{K\times V \in \calR_u}   \E_{\omega} 
\big| A^{\omega}_{K \times V}(&f_1, A_{j, \omega}(b_2,f_2), a_{i,\omega_1}^1(b_1, 1_{\wt \Omega_u}f_3))\big| \\
& \lesssim \int_{\wt \Omega_u \setminus \Omega_{u-1}} S^{2}_{v_1}f_1 S^{k_2,v_2}_{A_{j,b_2}} f_2 S^1_{k_3,a^1_{i,b_1}}(1_{\wt \Omega_u}f_3),
\end{split}
\end{equation*}
and then concluding as usual. The other produced terms from the expansion are handled similarly. We have proved
\eqref{eq:ResAveShiftExEst2}, and this finishes the treatment of the contribution of \eqref{eq:term1} to $[b_2,[b_1,S_{\omega}]_1]_2$.

We explain what happens if instead of \eqref{eq:term1} we consider another term
coming from the decomposition of $\langle [b_1,S_{\omega}]_1(f_1,f_2),f_3 \rangle$. The term that we handled
is the only one which uses the add and subtract trick from above, which lead to \eqref{eq:Com2Split1}. The other ones
are handled by directly expanding the products $b_2f_2$ and $b_2f_3$ in the usual way. 
Some of the terms are
particularly easy in the Banach range case, since they do not even require expansions and essentially reduce to the boundedness of $[b_2, S_{\omega}]_2$.
We omit these details.
\end{proof}

We end this section by stating the corresponding result for CZOs.
\begin{thm}\label{thm:mainCOM3}
Let $\|b_1\|_{\bmo(\R^{n+m})} = \|b_2\|_{\bmo(\R^{n+m})} = 1$, and let $1 < p, q \le \infty$ and $1/2 < r < \infty$ satisfy $1/p+1/q = 1/r$.
Let $T$ be a bilinear bi-parameter CZO.
Then we have
$$
 \|[b_2,[b_1,T]_1]_2(f_1, f_2)\|_{L^r(\R^{n+m})} 
 \lesssim \|f_1\|_{L^p(\R^{n+m})} \|f_2\|_{L^q(\R^{n+m})},
$$
and similarly for $[b_2,[b_1,T]_1]_1$.
\end{thm}

\section{Lower bounds for commutators}\label{sec:LowerBounds}
We quickly prove lower bounds by developing bilinear bi-parameter analogs of some methods present in the very recent paper \cite{Hy5} by Hyt\"onen.
Let $K$ be a kernel satisfying the standard estimates of a full kernel (as in Section \ref{ss:FK}). We also assume that
$K$ is uniformly non-degenerate. In our setup this means that for all $y \in \R^{n+m}$ and $r_1, r_2 > 0$ there exists $x \in \R^{n+m}$
such that $|x_1-y_1| > r_1$, $|x_2 - y_2| > r_2$ and
\begin{equation}\label{eq:UND}
|K(x, y, y)| \gtrsim \frac{1}{r_1^{2n} r_2^{2m}}.
\end{equation}
\begin{rem}
Regarding the assumed H\"older conditions of the kernel $K$, similarly as in \cite{Hy5}, a weaker modulus of continuity could be enough. 
Also, this bilinear non-degeneracy condition could perhaps be relaxed somewhat. We did not pursue these avenues.
\end{rem}
\begin{exmp}
Define the kernels of the bilinear one-parameter Riesz transforms $R_i^n$, $i = 1, \ldots, 2n$, by setting for $(x,y,z) \in \R^{3n}$ with $x \ne y$ or $x \ne z$ that
$$
K_{R_i^n}(x,y,z) = \frac{x_i - y_i}{( |x-y|^2 + |x-z|^2)^{(2n+1)/2}} \, \textup{ or } K_{R_i^n}(x,y,z) = \frac{x_i - z_i}{( |x-y|^2 + |x-z|^2)^{(2n+1)/2}}
$$
depending whether $i = 1,\ldots n$ or $i = n+1, \ldots, 2n$, respectively.   Let $R_{i,j} = R_{i,j}^{n,m} := R_i^n \otimes R_j^m$ denote the corresponding bilinear
bi-parameter Riesz transform on $\R^{n+m}$, and let $K_{i,j}$ denote its full kernel. Obviously $K_{i,j}$ satisfies \eqref{eq:UND}.
\end{exmp}
Notice that \eqref{eq:UND} implies the following: given a rectangle $R = I \times J$ and $C_0 > 0$, there exists a rectangle $\tilde R = \tilde I \times \tilde J$
such that $\ell(I) = \ell(\tilde I)$, $\ell(J) = \ell(\tilde J)$, $d(I, \tilde I) \ge C_0\ell(I)$, $d(J, \tilde J) \ge C_0\ell(J)$ and such that
for some $\sigma \in \C$ with $|\sigma| = 1$ we have for all $x \in \tilde R$ and $y,z \in R$
that
$$
\Re \sigma K(x,y,z) \gtrsim_{C_0} \frac{1}{|R|^2}.
$$
To see this, let, for a big enough constant $A = A(C_0)$, the centre $c_{\tilde R}$ of $\tilde R$ be the point $x$ given by \eqref{eq:UND}
applied to $y = c_R$, $r_1 = A\ell(I)$ and $r_2 = A\ell(J)$. Choose $\sigma$ so that
$\sigma K(c_{\tilde R}, c_R, c_R) = |K(c_{\tilde R}, c_R, c_R)|$, and use mixed H\"older and size estimates repeatedly.

Let $k \ge 1$ and $b \in L^k_{\textup{loc}}(\R^{n+m})$.
Let $\gamma_1, \gamma_2 \in \{0, 1, \ldots, k\}$ be such that $\gamma_1 + \gamma_2 = k$, $r > 0$ and $C_ 0 > 0$.
We define $\Gamma = \Gamma(K, b, k, r, \gamma_1, \gamma_2, C_0)$ by setting
$$
\Gamma = \sup \frac{1}{|R|^{1/r}} \Big\| x \mapsto 1_{\tilde R}(x) \iint_{A \times A} (b(x) - b(y))^{\gamma_1} (b(x) - b(z))^{\gamma_2} K(x,y,z) \ud y \ud z \Big\|_{L^{r, \infty}(\R^{n+m)}},
$$
where the supremum is taken over all rectangles $R, \tilde R$ with
$\ell(I) = \ell(\tilde I)$, $\ell(J) = \ell(\tilde J)$, $d(I, \tilde I) \ge C_0\ell(I)$ and $d(J, \tilde J) \ge C_0\ell(J)$, and all subsets $A \subset R$.
\begin{rem}
It is clear that if $T$ is a bilinear bi-parameter singular integral with a full kernel $K$, then the function inside the $L^{r,\infty}$ norm in $\Gamma$ is equal to
$$1_{\tilde R}(x) [b,\ldots[b, [b, T]_{i_1}]_{i_2}\ldots]_{i_k}(1_A, 1_A)(x),$$
where $i_j \in \{1,2\}$ with exactly $\gamma_1$ of them being $1$ and $\gamma_2$ of them being $2$.
Therefore, if $1/p + 1/q = 1/r$
for some $p,q \in (1, \infty]$ and this iterated commutator maps $L^p \times L^q \to L^{r, \infty}$, then $\Gamma$ is dominated by this norm. However, the constant
$\Gamma$ is significantly weaker and depends only on the kernel $K$ and some off-diagonal assumptions.
\end{rem}
\begin{prop}
Suppose $K$ is a uniformly non-degenerate bilinear bi-parameter full kernel, $k \ge 1$ and $b \in L^k_{\textup{loc}}(\R^{n+m})$ is real-valued.
Let $\gamma_1, \gamma_2 \in \{0, 1, \ldots, k\}$ be such that $\gamma_1 + \gamma_2 = k$, $r > 0$ and $C_ 0 > 0$. Then for 
$\Gamma = \Gamma(K, b, k, r, \gamma_1, \gamma_2, C_0)$ we have
$
\|b\|_{\bmo(\R^{n+m})} \lesssim \Gamma^{1/k}.
$
\end{prop}
\begin{proof}
Let $R = I \times J$ be a fixed rectangle. As we saw above we can find $\tilde R = \tilde I \times \tilde J$
such that $\ell(I) = \ell(\tilde I)$, $\ell(J) = \ell(\tilde J)$, $d(I, \tilde I) \ge C_0\ell(I)$, $d(J, \tilde J) \ge C_0\ell(J)$ and such that
for some $\sigma \in \C$ with $|\sigma| = 1$ we have for all $x \in \tilde R$ and $y,z \in R$
that
$
\Re \sigma K(x,y,z) \gtrsim_{C_0} |R|^{-2}.
$
For $t \in \R$ let $t_+ = \max(t,0)$. For an arbitrary $\alpha \in \R$ and $x \in \tilde R \cap \{b \ge \alpha\}$ we have
\begin{align*}
\Big( \frac{1}{|R|} \int_R (\alpha - b)_+ \Big)^k &\le \frac{1}{|R|^2} \iint_{[R \cap \{b \le \alpha\}]^2} (b(x) - b(y))^{\gamma_1} (b(x) - b(z))^{\gamma_2}
\ud y \ud z \\
& \lesssim \Re \sigma \iint_{[R \cap \{b \le \alpha\}]^2} (b(x) - b(y))^{\gamma_1} (b(x) - b(z))^{\gamma_2} K(x,y,z) \ud y \ud z.
\end{align*}
Letting $\alpha$ be a median of $b$ in $\tilde R$ we have that
\begin{align*}
|R|^{1/r} &\Big( \frac{1}{|R|} \int_R (\alpha - b)_+ \Big)^k \lesssim 
|\tilde R \cap \{b \ge \alpha\}|^{1/r} \Big( \frac{1}{|R|} \int_R (\alpha - b)_+ \Big)^k \\
&\lesssim \Big\| x \mapsto 1_{\tilde R}(x) \iint_{[R \cap \{b \le \alpha\}]^2} (b(x) - b(y))^{\gamma_1} (b(x) - b(z))^{\gamma_2} K(x,y,z) \ud y \ud z \Big\|_{L^{r, \infty}(\R^{n+m})}.
\end{align*}
Using the definition of $\Gamma$ we get
$
\int_R (\alpha - b)_+ \lesssim \Gamma^{1/k} |R|.
$
Combining with the symmetric estimate
$
\int_R (b-\alpha)_+ \lesssim \Gamma^{1/k} |R|
$
we are done.
\end{proof}

\begin{appendix}
\section{Bilinear bi-parameter multipliers}
We show that the multipliers $T_m$ studied e.g. in \cite{MPTT}
are paraproduct free (even the partial paraproducts vanish) bilinear bi-parameter CZOs.
Recall that in Grafakos--Torres \cite{GT} it was shown that the bilinear one-parameter multipliers of Coifman--Meyer \cite{CM} are bilinear one-parameter singular integrals satisfying $T1$ type assumptions. So what we do is in this spirit but in the bi-parameter setting.

Let $m \in L^{\infty}(\R^{n+m} \times \R^{n+m})$
be smooth outside the set $\Delta:= \{(\xi, \eta) \in \R^{n+m} \times \R^{n+m} \colon |\xi_1| + |\eta_1| = 0 \textup{ or } |\xi_2| + |\eta_2| = 0\}$, and assume that it satisfies
the estimate
$$
|\partial^{\alpha_1}_{\xi_1} \partial^{\alpha_2}_{\xi_2}  \partial^{\beta_1}_{\eta_1}  \partial^{\beta_2}_{\eta_2}  m(\xi, \eta)| \lesssim
(|\xi_1| + |\eta_1|)^{-|\alpha_1| - |\beta_1|} (|\xi_2| + |\eta_2|)^{-|\alpha_2| - |\beta_2|}. 
$$
For Schwartz functions $f_1, f_2 \colon \R^{n+m} \to \C$ and $x \in \R^{n+m}$ define
$$
T_m(f_1, f_2)(x) = \iint_{\R^{n+m}} \iint_{\R^{n+m}} m(\xi, \eta) \widehat f_1(\xi) \widehat f_2(\eta) e^{2\pi i x \cdot (\xi + \eta)} \ud \xi \ud \eta,
$$
where
$
 \widehat f_1(\xi) = \iint_{\R^{n+m}} f_1(x) e^{-2\pi i x \cdot \xi}\ud x.
$

Fix a smooth function $p_1\colon\R^{2n} \to [0,1]$ so that $p_1(\xi_1, \eta_1) = 1$ when $|\xi_1| + |\eta_1| \le 1$ and $p_1(\xi_1, \eta_1) = 0$ when $|\xi_1| + |\eta_1| \ge 2$. Set $\delta_1(\xi_1, \eta_1) = p_1(\xi_1, \eta_1) - p_1(2\xi_1, 2\eta_1)$. Define $\delta_2$ similarly in $\R^{2m}$, and set
$$
m_{j_1, j_2}(\xi, \eta) = m(\xi, \eta) \delta_1(2^{-j_1}\xi_1, 2^{-j_1}\eta_1) \delta_2(2^{-j_2}\xi_2, 2^{-j_2}\eta_2), \qquad j_1, j_2 \in \Z.
$$
For $J \in \N$ define
$$
m_J(\xi, \eta) = \sum_{|j_1|, |j_2| \le J} m_{j_1, j_2}(\xi, \eta).
$$
We show that $T_{m_J}$ is a paraproduct free (in the strong sense that also the partial paraproducts vanish) bilinear bi-parameter singular integral in our class uniformly on $J$.

\subsection*{Full kernel representation}
Set $K_J(x,y) = \sum_{|j_1|, |j_2| \le J} K_{j_1, j_2}(x,y)$, where
$$
K_{j_1, j_2}(x,y) = \iint_{\R^{n+m}} \iint_{\R^{n+m}} m_{j_1, j_2}(\xi, \eta) e^{2\pi i (x \cdot \xi + y\cdot \eta)} \ud \xi \ud \eta.
$$
Notice that $K_J$ is a Schwartz function as $m_J$ is, and that for all $x \in \R^{n+m}$ we have
$$
T_{m_J}(f_1, f_2)(x) = K_J * (f_1 \otimes f_2)(x,x) = \iint_{\R^{n+m}} \iint_{\R^{n+m}} K_J(x-y, x-z)f_1(y)f_2(z)\ud z \ud y.
$$
We claim that $(x,y,z) \mapsto K_J(x-y,x-z)$ satisfies the estimates required from a full kernel uniformly on $J$.
Using standard arguments involving integration by parts as in Stein's book \cite{Stein:book} pp. 245--246 (which is the linear
one-parameter case), we see that
$$
\sum_{j_1, j_2 \in \Z} |\partial^{\alpha_1}_{x_1} \partial^{\alpha_2}_{x_2}  \partial^{\beta_1}_{y_1}  \partial^{\beta_2}_{y_2} K_{j_1, j_2}(x,y)|
\lesssim (|x_1|+|y_1|)^{-2n - |\alpha_1| - |\beta_1|}(|x_2|+|y_2|)^{-2m - |\alpha_2| - |\beta_2|}.
$$
These derivative bounds imply the desired full kernel estimates for $K_J$, uniformly on $J$.

\subsection*{Partial kernel representation}
Let $f_i = f_i^1 \otimes f_i^2$, where $f_i^1\colon \R^n \to \C$ and $f_i^2 \colon \R^m \to \C$, $i = 1,2,3$. We write
$$
\langle T_{m_J}(f_1, f_2), f_3\rangle = \langle T_{m_{J, f_1^2, f_2^2, f_3^2}}(f_1^1, f_2^1), f_3^1\rangle,
$$
where
$$
m_{J, f_1^2, f_2^2, f_3^2}(\xi_1, \eta_1) = \iiint_{\R^m \times \R^m \times \R^m} m_J(\xi, \eta) \widehat{f_1^2}(\xi_2) \widehat{f_2^2}(\eta_2) f_3^2(x_2)
e^{2\pi i x_2 \cdot (\xi_2 + \eta_2)} \ud \xi_2 \ud \eta_2 \ud x_2
$$
and
$$
T_{m_{J, f_1^2, f_2^2, f_3^2}}(f_1^1, f_2^1)(x_1) = \iint_{\R^n \times \R^n} m_{J, f_1^2, f_2^2, f_3^2}(\xi_1, \eta_1)
\widehat{f_1^1}(\xi_1) \widehat{f_2^1}(\eta_1) e^{2\pi i x_1 \cdot (\xi_1 + \eta_1)} \ud \xi_1 \ud \eta_1.
$$
Notice that
$$
\partial^{\alpha_1}_{\xi_1} \partial^{\beta_1}_{\eta_1} m_{J, f_1^2, f_2^2, f_3^2}(\xi_1, \eta_1) = \langle T_{\partial^{\alpha_1}_{\xi_1} \partial^{\beta_1}_{\eta_1}m_J(\xi_1, \cdot, \eta_1, \cdot)}(f_1^2, f_2^2), f_3^2\rangle,
$$
where
$$
T_{\partial^{\alpha_1}_{\xi_1} \partial^{\beta_1}_{\eta_1} m_J(\xi_1, \cdot, \eta_1, \cdot)}(f_1^2, f_2^2)(x_2) = \iint_{\R^m \times \R^m} \partial^{\alpha_1}_{\xi_1} \partial^{\beta_1}_{\eta_1}m_J(\xi, \eta) \widehat{f_1^2}(\xi_2) \widehat{f_2^2}(\eta_2) e^{2\pi i x_2 \cdot (\xi_2 + \eta_2)} \ud \xi_2 \ud \eta_2.
$$
Since we have
$$
|\partial^{\alpha_2}_{\xi_2} \partial^{\beta_2}_{\eta_2} \partial^{\alpha_1}_{\xi_1} \partial^{\beta_1}_{\eta_1} m_J(\xi_1, \xi_2, \eta_1, \xi_2)|
\lesssim (|\xi_1| + |\eta_1|)^{-|\alpha_1| - |\beta_1|} (|\xi_2| + |\eta_2|)^{-|\alpha_2| - |\beta_2|},
$$
Coifman--Meyer \cite{CM} and Grafakos--Torres \cite{GT} tell us that
$$
|\langle T_{\partial^{\alpha_1}_{\xi_1} \partial^{\beta_1}_{\eta_1}m_J(\xi_1, \cdot, \eta_1, \cdot)}(f_1^2, f_2^2), f_3^2\rangle|
\lesssim (|\xi_1| + |\eta_1|)^{-|\alpha_1| - |\beta_1|} \| f_1^2 \|_{L^4(\R^m)} \| f_2^2 \|_{L^4(\R^m)} \| f_3^2 \|_{L^2(\R^m)}.
$$
But this means that the same upper bound holds for $|\partial^{\alpha_1}_{\xi_1} \partial^{\beta_1}_{\eta_1} m_{J, f_1^2, f_2^2, f_3^2}(\xi_1, \eta_1)|$.
Another application of \cite{CM}, \cite{GT} tells us that for $x_1 \in \R^n$ we have
$$
T_{m_{J, f_1^2, f_2^2, f_3^2}}(f_1^1, f_2^1)(x_1) = \iint_{\R^n \times \R^n} K_{J, f_1^2, f_2^2, f_3^2}(x_1 - y_1, x_1 - z_1) f_1^1(y_1) f_2^1(z_1)\ud y_1 \ud z_1,
$$
where $(x_1, y_1, z_1) \mapsto K_{J, f_1^2, f_2^2, f_3^2}(x_1 - y_1, x_1 - z_1)$ is a bilinear Calder\'on--Zygmund kernel with a constant
dominated by $ \| f_1^2 \|_{L^4(\R^m)} \| f_2^2 \|_{L^4(\R^m)} \| f_3^2 \|_{L^2(\R^m)}$. In particular, we have the desired partial kernel representation
in $\R^n$ with the correct bounds. Of course, we can similarly also obtain the partial kernel representation in $\R^m$.

We record here that the above application of \cite{CM}, \cite{GT} actually gives the bound
\begin{equation}\label{eq:tensorprodcont}
\begin{split}
|\langle T_{m_J}(&f_1, f_2), f_3\rangle|  = |\langle T_{m_{J, f_1^2, f_2^2, f_3^2}}(f_1^1, f_2^1), f_3^1\rangle| \\
&\lesssim  \| f_1^1 \|_{L^4(\R^n)} \| f_1^2 \|_{L^4(\R^m)}
\| f_2^1 \|_{L^4(\R^n)} \| f_2^2 \|_{L^4(\R^m)} \| f_3^1 \|_{L^2(\R^n)} \| f_3^2 \|_{L^2(\R^m)},
\end{split}
\end{equation}
i.e. boundedness for functions of the tensor product form.

\subsection*{Boundedness and cancellation assumptions}
Notice that \eqref{eq:tensorprodcont} immediately implies the weak boundedness and diagonal BMO assumptions.

We claimed that $T_{m_J}$ is a paraproduct free bilinear bi-parameter singular integral.
It remains to check that for all bounded functions
$f_i \colon \R^n \to \C$, $g_i \colon \R^m \to \C$, $i=1,2$,
all 
$$
S_J \in \{T_{m_J}, T_{m_J}^{1*}, T_{m_J}^{2*}, (T_{m_J})^{1*}_1, (T_{m_J})^{2*}_1, (T_{m_J})^{1*}_2, (T_{m_J})^{2*}_2, (T_{m_J})^{1*, 2*}_{1,2}, (T_{m_J})^{1*, 2*}_{2,1}\}
$$
and all rectangles $R \subset \R^{n+m}$ there holds
$$
\langle S_J(1 \otimes g_1, 1 \otimes g_2), h_R \rangle 
= \langle S_J(f_1 \otimes 1, f_2 \otimes 1), h_R \rangle =0.
$$
By symmetry it is enough to show
$
\langle T_{m_J}(1 \otimes g_1, 1 \otimes g_2), h_R \rangle = 0.
$
We simply have that (the integral converges as $K_J$ is a Schwartz function)
$$
T_{m_J}(1 \otimes g_1, 1 \otimes g_2)(x) = \iint_{\R^{n+m}} \iint_{\R^{n+m}} K_J(y_1, x_2-y_2, z_1, x_2-z_2) g_1(y_2) g_2(z_2) \ud y \ud z
$$
is a constant fuction on $x_1$, and therefore the claim follows using $\int_{\R^n} h_R(x_1,x_2) \ud x_1 = 0$.

\subsection*{Conclusion}
With some more care it is possible to prove that $T_m$ itself is a paraproduct free bilinear bi-parameter CZO, and not only that
$$
\langle T_{m}(f_1, f_2), f_3\rangle = \lim_{J \to \infty} \langle T_{m_J}(f_1, f_2), f_3\rangle,
$$
where $T_{m_J}$ are such uniformly on $J$. 

\end{appendix}

\end{document}